\documentclass[a4paper,twoside,10pt]{article}

\usepackage[a4paper,left=3cm,right=3cm, top=3cm, bottom=3cm]{geometry}
\usepackage[latin1]{inputenc}
\usepackage{mathrsfs}
\usepackage{graphicx}
\usepackage{epstopdf}
\usepackage{amsmath}
\usepackage{amsthm}
\usepackage{amssymb}
\usepackage{hyperref}
\usepackage{stmaryrd}
\usepackage{float}
\usepackage{bigints}
\usepackage{cite}
\usepackage{color}
\usepackage[abs]{overpic}
\usepackage[font=footnotesize,labelfont=bf]{caption}
\usepackage{cases}
\usepackage{tikz}
\usepackage{algorithmic}
\usepackage{rotating}
\usepackage{blkarray}
\usetikzlibrary{matrix,calc,arrows}
\usepackage[author={Lorenzo}]{pdfcomment}
\usepackage{soul,xcolor}
\usepackage{enumitem}  
\usepackage{verbatim}
\usepackage{graphicx}
\usepackage{subcaption}
\usepackage{mwe}
\usepackage{algorithm}
\usepackage{pifont}
\usepackage{multirow}

\restylefloat{table}
\theoremstyle{plain}
\newtheorem{thm}{Theorem}[section]

\newtheorem{lem}[thm]{Lemma}

\theoremstyle{definition}

\theoremstyle{remark}
\newtheorem{remark}{Remark}

\newcommand{\C}{\mathbb{C}}
\newcommand{\R}{\mathbb{R}}
\newcommand{\N}{\mathbb{N}}

\newcommand{\im}{\textup{i}} 
\newcommand{\E}{K} 
\newcommand{\Ee}{\E_\e} 
\newcommand{\taun}{\mathcal T_n} 
\newcommand{\e}{e} 
\newcommand{\Pip}{\Pi_{\p}^{\E}} 
\newcommand{\Pie}{\Pi^{0,\e}_\p} 
\newcommand{\PiGammaR}{\Pi^{0,\Gamma_R}_\p} 
\newcommand{\PiGammaN}{\Pi^{0,\Gamma_N}_\p} 
\newcommand{\SE}{S^\E} 
\newcommand{\un}{u_h} 
\newcommand{\vn}{v_h} 
\newcommand{\w}{w} 
\renewcommand{\k}{k} 
\newcommand{\n}{\mathbf n} 
\newcommand{\nOmega}{\n_\Omega} 
\renewcommand{\Re}{\textup{Re}} 
\renewcommand{\a}{a} 
\renewcommand{\b}{b} 
\newcommand{\aE}{\a^\E} 
\newcommand{\anE}{\aE_\h} 
\newcommand{\an}{\a_\h} 
\newcommand{\bn}{b_\h} 
\newcommand{\fn}{F_\h} 
\newcommand{\PW}{\mathbb {PW}} 
\newcommand{\PWc}{\mathbb {PW}_\p^c} 
\newcommand{\PWctilde}{\widetilde{\mathbb {PW}}_\p^c} 
\newcommand{\PWE}{\mathbb {PW}_\p(\E)} 
\newcommand{\V}{V} 
\newcommand{\VgD}{\V_{\gD}} 
\newcommand{\Vz}{\V_{0}} 
\newcommand{\VhgD}{V_{\h,\gD}} 
\newcommand{\Vhz}{V_{\h,0}} 
\newcommand{\VE}{V_\h(\E)} 
\newcommand{\VEe}{V_\h(\Ee)} 
\newcommand{\VEhat}{\widehat{V_\h}(\E)} 
\newcommand{\wlE}{w_\ell^{\E}} 
\newcommand{\wjE}{w_j^{\E}} 
\newcommand{\wle}{w_\ell^e} 
\newcommand{\wre}{w_r^e} 
\newcommand{\wje}{w_j^e} 
\newcommand{\wzE}{w_{\zeta}^\E} 
\newcommand{\wlehat}{\widehat{w}_\ell^e} 
\newcommand{\pehat}{\widehat{p}_e} 
\newcommand{\GammaD}{\Gamma_D} 
\newcommand{\GammaN}{\Gamma_N} 
\newcommand{\GammaR}{\Gamma_R} 
\newcommand{\q}{q} 
\newcommand{\qE}{\q_\E} 
\newcommand{\dtilde}{\widetilde{\mathbf d}} 
\newcommand{\dtildetilde}{\widetilde{\widetilde{\mathbf d}}} 

\newcommand{\h}{h} 
\newcommand{\hE}{\h_\E} 
\newcommand{\he}{\h_e} 
\newcommand{\p}{p} 
\newcommand{\pe}{p_e} 
\newcommand{\pE}{p_\E} 
\newcommand{\gammaIK}{\gamma_I^K} 
\newcommand{\g}{g} 
\newcommand{\gtilde}{\widetilde \g} 
\newcommand{\gD}{\g_D} 
\newcommand{\gN}{\g_N} 
\newcommand{\gR}{\g_R} 
\newcommand{\En}{\mathcal E_n} 
\newcommand{\Enb}{\mathcal E_n^B} 
\newcommand{\EnI}{\mathcal E_n^I} 
\newcommand{\EnD}{\mathcal E_n^D} 
\newcommand{\EnN}{\mathcal E_n^N} 
\newcommand{\EnR}{\mathcal E_n^R} 
\newcommand{\EE}{\mathcal E^{\E}} 
\newcommand{\x}{\textbf{\textup{x}}} 
\newcommand{\xE}{\textbf{\textup{x}}_\E} 
\newcommand{\xe}{\textbf{\textup{x}}_{e}} 
\newcommand{\dstar}{\textbf{\textup{d}}_*} 
\newcommand{\dl}{\textbf{\textup{d}}_\ell} 
\newcommand{\djj}{\textbf{\textup{d}}_j} 
\newcommand{\dm}{\textbf{\textup{d}}_m} 
\newcommand{\Nedg}{N_e} 
\newcommand{\eps}{\textup{eps}} 

\newcommand{\Nd}{N_\text{dof}} 
\newcommand{\ds}{\text{d}s}
\newcommand{\dx}{\text{d}x}
\newcommand{\lin}{\textup{span}}

\newcommand{\J}{\mathcal{J}} 
\newcommand{\Je}{\mathcal{J}_e} 
\newcommand{\Jes}{\mathcal{J}_{e_s}} 
\newcommand{\Jer}{\mathcal{J}_{e_r}} 
\newcommand{\Jeprime}{\mathcal{J}_{e}'}

\newcommand{\dof}{\textup{dof}} 

\newcommand{\nE}{\textbf{\textup{n}}_\E} 
\newcommand{\Ne}{n_\E} 
\newcommand{\ME}{\mathcal{M}_\E} 

\newcommand{\wE}{w^\E} 
\newcommand{\we}{w^\e} 
\newcommand{\wetae}{w_{\eta}^\e} 

\newcommand{\dir}{\textbf{\textup{d}}} 

\newcommand{\Ghat}{\widehat{\boldsymbol{G}}^\E}
\newcommand{\Bhat}{\widehat{\boldsymbol{B}}^\E}
\newcommand{\Dhat}{\widehat{\boldsymbol{D}}^\E}
\newcommand{\Rhat}{\widehat{\boldsymbol{R}}}
\newcommand{\Fhat}{\widehat{\boldsymbol{f}}}
\newcommand{\Ahat}{\widehat{\boldsymbol{A}}}
\newcommand{\Ihat}{\widehat{\boldsymbol{I}}^\E}
\newcommand{\Pihat}{\widehat{\boldsymbol{\Pi}}^\E}
\newcommand{\Shat}{\widehat{\boldsymbol{S}}^\E}
\newcommand{\Rhate}{\widehat{\boldsymbol{R}}^e}

\newcommand{\abold}{\textbf{\textup{a}}}
\newcommand{\bbold}{\textbf{\textup{b}}}

\newcommand{\meas}{\textup{meas}} 

\begin{tiny}
\author{
\normalsize{
}}
\end{tiny}

\date{}

\title{A nonconforming Trefftz virtual element method for the Helmholtz problem: numerical aspects}
\date{}
\author{Lorenzo Mascotto\thanks{Faculty of Mathematics, University of Vienna, 1090 Vienna, Austria (lorenzo.mascotto@univie.ac.at, ilaria.perugia@univie.ac.at, alex.pichler@univie.ac.at)},\ Ilaria Perugia\footnotemark[1],\ 
Alexander Pichler\footnotemark[1]}

\begin{document}
\maketitle
\begin{abstract}
We discuss the implementation details and the numerical performance of the recently introduced nonconforming Trefftz virtual element method \cite{ncTVEM_theory} for the 2D Helmholtz problem.
In particular, we present a strategy to significantly reduce the ill-conditioning of the original method;
such a recipe is based on an automatic filtering of the basis functions edge by edge, and therefore allows for a notable reduction of the number of degrees of freedom.
A widespread set of numerical experiments, including an application to acoustic scattering, the $\h$-, $\p$-, and $\h\p$-versions of the method, is presented.
Moreover, a comparison with other Trefftz-based methods for the Helmholtz problem shows that this novel approach results in robust and effective performance.


\end{abstract}

\section{Introduction} \label{section introduction}
Owing to their flexibility in dealing with complex geometries, Galerkin methods based on polytopal grids have been the object of an extensive study over the last years. 
Among them, we mention the discontinuous Galerkin method \cite{antonietti2016reviewDG},
the hybridized discontinuous Galerkin method \cite{cockburn_HDG}, the hybrid high-order method \cite{dipietroErn_hho}, the mimetic finite difference method \cite{BLM_MFD, lipnikov2014mimetic}, the high order boundary element method-based finite element method (FEM) \cite{Weisser_basic},
and the virtual element method (VEM)~\cite{VEMvolley, hitchhikersguideVEM}.
In this paper, we focus on the latter, which, despite its novelty, has already been used in a wide number of problems, including engineering applications.

In comparison to more standard methods, such as the FEM, the VEM has the feature that it is based on spaces of functions that are not known in closed form, but rather are defined elementwise as solutions to local partial differential equations.
Although seeming to be a hindrance at a first glance, this property allows for a natural coupling with the Trefftz setting, where the functions in the trial and test spaces belong elementwise to the kernel of the differential operator of the boundary value problem under consideration.
The advantage of incorporating properties of the problem solution in the approximating spaces is that, when solving homogeneous problems, less degrees of freedom are needed in order to achieve a given accuracy.
As typical of VEM, after defining local approximation spaces, one needs to introduce a set of degrees of freedom that allow to construct a computable method, via proper stabilizations and mappings
onto finite dimensional spaces of functions that (a) possess good approximation properties (polynomials, plane waves, \dots) and (b) are explicitly known.

In this paper, we focus on the approximation of solutions to the two dimensional homogeneous Helmholtz problem, which has already been the target of two different VE approaches.
The first one \cite{Helmholtz-VEM} is an $H^1$-conforming plane wave VEM (PWVEM), which can be interpreted as a partition of unity method \cite{BabuskaMelenk_PUMintro},
the way that the trial and test spaces consist elementwise of plane wave spaces that are eventually glued together by modulating them via a partition of unity.
On the other hand, the second and more recent approach is a nonconforming Trefftz-VEM introduced in \cite{ncTVEM_theory}. The latter combines the VE technology with the Trefftz setting in a nonconforming fashion (\textit{\`{a} la} Crouzeix-Raviart)
following the pioneering works on nonconforming VEM for elliptic problems~\cite{nonconformingVEMbasic,cangianimanzinisutton_VEMconformingandnonconforming}
and their extension to other problems \cite{cangianimanzinisutton_VEMconformingandnonconforming, ncHVEM, CGM_nonconformingStokes, gardini2018nonconforming, nc_VEM_NavierStokes, zhao2016nonconforming, VEM_fullync_biharmonic,cao2018anisotropic}.

This nonconforming Trefftz-VEM, which can be regarded as a generalization of the nonconforming harmonic VEM~\cite{ncHVEM}, is ``morally'' comparable to many other Trefftz methods for the Helmholtz equation such as
the ultra weak variational formulation \cite{cessenatdespres_basic}, the wave based method \cite{wavebasedmethod_overview},
discontinuous methods based on Lagrange multipliers \cite{farhat2001discontinuous} and on least square formulation \cite{monk1999least},
the plane wave discontinuous Galerkin method (PWDG) \cite{GHP_PWDGFEM_hversion}, and the variational theory of complex rays \cite{riou2008multiscale}; see \cite{PWDE_survey} for an overview of such methods.

It has to be mentioned that all of the above Trefftz methods are based on fully discontinuous approximation spaces. A peculiarity of the nonconforming Trefftz-VEM is that a ``weak'' notion (that is, via proper edge $L^2$ projections) of traces over the skeleton of the polytopal grid is, differently from discontinuous methods, available.

The aim of the present paper is to continue the work begun in \cite{ncTVEM_theory}, where the nonconforming Trefftz-VEM was firstly introduced, an abstract error analysis was carried out, and $\h$-version error estimates were derived.
As already mentioned in \cite{ncTVEM_theory}, the original version of the method does not result in good numerical performance, mainly because of the strong ill-conditioning of the local plane wave basis functions.

The scope of this contribution is manifold. After introducing the model problem and extending the original nonconforming Trefftz-VEM in Section~\ref{section nc Trefftz VEM}, we discuss the implementation details of the method in Section~\ref{section implementational details}. We will consider here a more general Helmholtz boundary value problem than originally done in \cite{ncTVEM_theory}, which will be reflected in the definition of the nonconforming Trefftz-VE spaces.
Then, numerical results are presented in Section \ref{section numerical results standard}, in order to clarify that, \emph{rebus sic stantibus},
the method severely suffers of ill-conditioning.
A numerical recipe based on an edgewise orthonormalization procedure to mitigate this strong ill-conditioning is presented in Section~\ref{section cure illconditioning}.
Additionally to the fact that the condition number of the resulting global matrix significantly improves, the number of degrees of freedom is reduced without deteriorating the accuracy.
To the best of our understanding, such a recipe cannot be directly applied in the framework of DG methods, see Remark~\ref{second remark on the second filtering}.
After testing the modified version of the method in several experiments, including an acoustic scattering problem, we compare its performance with that of PWVEM and PWDG.
The new approach turns out to be very competitive, when compared to existing technologies, especially in the high-order case and when approximating highly oscillatory problems.
Moreover, we numerically study the $\p$- and $\h\p$-versions of the method, experimentally assessing exponential convergence for analytic and singular solutions in the former and latter cases, respectively.

\section{The nonconforming Trefftz virtual element method} \label{section nc Trefftz VEM}

In this section, after introducing the notation and presenting the continuous model problem, we recall the nonconforming Trefftz-VEM of~\cite{ncTVEM_theory}.

Throughout the paper, we will denote by $H^s(\mathcal{D})$, $\mathcal{D} \subset \R^2$, the Sobolev space of order $s \in \N$ over the complex field $\C$. For fractional $s$, the corresponding Sobolev spaces can be defined via interpolation theory, see e.g.~\cite{Triebel}.
In addition, we will employ the standard notation for sesquilinear forms, norms and seminorms 
\begin{equation*}
(\cdot,\cdot)_{s,\mathcal{D}}, \quad \lVert \cdot \rVert_{s,\mathcal{D}}, \quad \lvert \cdot \rvert_{s,\mathcal{D}}.
\end{equation*}
The model problem we are interested in is a homogeneous Helmholtz boundary value problem with mixed boundary conditions. More precisely, given $\Omega \subset \mathbb R^2$ a bounded polygonal domain, we split its boundary $\partial \Omega$ into
\begin{equation} \label{decomposition boundary}
\partial \Omega = \overline{\GammaD} \cup \overline{\GammaN} \cup \overline{\GammaR} ,\quad \GammaD \cap \GammaN = \emptyset ,\quad \GammaD \cap \GammaR = \emptyset, \quad \GammaN \cap \GammaR = \emptyset, \quad |\GammaR|>0.
\end{equation}
The strong formulation of the continuous problem  reads
\begin{equation} \label{HH continuous problem}
\left\{
\begin{alignedat}{2}
\text{find } u \in H^1(\Omega) \text{ such that} \hspace{-1.5cm} &&\\
-\Delta u -\k^2 u &= 0 	&&\quad \text{in } \Omega\\
u &= \gD &&\quad \text{on } \GammaD\\
\nabla u \cdot \nOmega &= \gN &&\quad \text{on } \GammaN\\
\nabla u \cdot \nOmega + \im \k \theta u  &= \gR &&\quad \text{on } \GammaR,\\
\end{alignedat}
\right.
\end{equation}
where $\k>0$ is the wave number (with corresponding wave length $\lambda=\frac{2\pi}{\k}$), $\im$ is the imaginary unit, $\nOmega$ denotes the unit normal vector on $\partial \Omega$ pointing outside $\Omega$, $\theta \in \{-1,1\}$,
$\gD \in H^{\frac{1}{2}}(\GammaD)$, $\gN \in H^{-\frac{1}{2}}(\GammaN)$, and $\gR \in H^{-\frac{1}{2}}(\GammaR)$.

The corresponding weak formulation reads
\begin{equation} \label{weak continuous problem}
\begin{cases}
\text{find } u \in \VgD \text{ such that}\\
\b(u,v) = \langle F, v \rangle \quad \forall v \in \Vz,
\end{cases}
\end{equation}
where
\[
\VgD := H^1_{\gD,\GammaD}(\Omega) = \left\{ v \in H^1(\Omega) \,:\, v_{|_{\GammaD}} =\gD \right\},\quad 
\Vz := H^1_{0,\GammaD}(\Omega) = \left\{ v \in H^1(\Omega) \,:\, v_{|_{\GammaD}} =0 \right\}
\]
and
\[
\b(u,v) := \a(u,v) + \im \k \theta \int_{\Gamma_R} u \overline v \, \ds, \quad
\langle F,v \rangle := \int_{\GammaN} \gN \overline v \, \ds  +  \int _{\GammaR} \gR \overline v \, \ds \quad \forall u, v \in H^1(\Omega),
\]
with
\[
\a(u,v) := \int_{\Omega} \nabla u \cdot \overline{\nabla v} \, \dx - \k^2 \int_\Omega u \overline v \, \dx \quad \forall u,v \in H^1(\Omega).
\]
Since we are assuming that $\vert \GammaR \vert>0$, see~\eqref{decomposition boundary}, existence and uniqueness of solutions to the problem~\eqref{HH continuous problem} follow from the Fredholm alternative and a continuation argument.
\begin{thm} \label{thm well-posedness HH}
Under the assumptions \eqref{decomposition boundary} on $\Omega$, problem \eqref{HH continuous problem} is uniquely solvable.
\end{thm}
\begin{proof}
We first note that the sesquilinear form $b(\cdot,\cdot)$ in~\eqref{weak continuous problem} is continuous and satisfies a G\aa{}rding inequality \cite[p.118]{mclean2000strongly}. 
Owing to the Fredholm alternative \cite[Thm. 4.11, 4.12]{mclean2000strongly}, the problem \eqref{HH continuous problem} admits a unique solution if and only if the homogeneous adjoint problem to \eqref{HH continuous problem} with homogeneous boundary conditions,
which is obtained by switching the sign in front of the boundary integral term over $\GammaR$ in $b(u,v)$, admits only the trivial solution~$0$.\\
In order to show this, we consider the variational formulation of the homogeneous adjoint problem with homogeneous boundary conditions, we test with~$v=u$, and we take the imaginary part, thus deducing~$u=0$ on $\GammaR$.
In particular, also $\nabla u \cdot \nOmega=0$, due to the definition of the impedance trace.\\
Let now $U \subset \R^2$ be an open, connected set such that $U \cap \partial \Omega = \GammaR$ and $\meas(U \backslash \overline{\Omega})>0$.
We define $\widetilde{\Omega}:=\Omega \cup U$ and $\widetilde{u}: \, \widetilde{\Omega} \to \C$ as the extension of $u$ by zero in $ \widetilde{\Omega} \setminus \Omega$.
Then $\widetilde{u}$ solves a homogeneous Helmholtz equation in $\widetilde{\Omega}$; applying the unique continuation principle, see e.g.~\cite{aronszajn1957unique}, leads to~$\widetilde{u}=0$ in $\widetilde{\Omega}$, and therefore~$u=0$ in~$\Omega$.
\end{proof}
We highlight that the existence and the uniqueness of solutions can also be shown for more general Helmholtz-type boundary value problems, see e.g. \cite{stability_helmholtz_graham_sauter}. 

\medskip
Let now $\taun$ be a decomposition of $\Omega$ into polygons $\{\E\}$ with mesh size $h:=\max_{\E \in \taun} \, \hE$, where $\hE:=\textup{diam}(\E)$ for all $\E \in \taun$. Further, we introduce $\En$, $\EnI$ and $\Enb$, 
the set of edges, interior edges, and boundary edges of $\taun$, respectively.
We assume that the boundary edges comply with respect to the decomposition~\eqref{decomposition boundary}, that is, for all boundary edges $\e \in \Enb$,
$\e$ is contained in only one amidst $\GammaD$, $\GammaN$, and $\GammaR$.
In the sequel, we will use the following notation for the set of ``Dirichlet, Neumann, and impedance (Robin)'' edges:
\[
\EnD=\{ \e \in \Enb \, : \, \e \subseteq \GammaD \},\quad  \EnN = \{ \e \in \Enb \, : \, \e \subseteq \GammaN \}, \quad \EnR=\{ \e \in \Enb \, : \, \e \subseteq \GammaR \}.
\]
For any polygon $\E \in \taun$, we denote by $\EE$ the set of its edges, by $\xE$ its centroid, and by $\Ne$ the cardinality of $\EE$. Finally, given any $\e\in\EE$, we denote by $\xe$ its midpoint, and by $\he$ its length.
The normal unit vector pointing outside $\E$ is denoted by $\nE$.

Next, we define plane wave spaces in the bulk of the elements of $\taun$ and on the edges. To this purpose, fix $p=2q+1$, $q \in \N$, and let $\{ \dl \}_{\ell \in \J}$ be a set of pairwise different and normalized directions, where $\J:=\{1,\dots,p\}$.
For every $\E \in \taun$ and $\ell \in \J$, we define the local plane wave space on $\E$ by
\begin{equation} \label{plane wave bulk space}
\PWE:=\lin \left\{ \wlE \, , \, \ell \in \J \right\},
\end{equation}
where $\wlE(\x):={e^{\im\k \dl \cdot (\x-\xE)}}_{|_\E}$ denotes for all $\ell \in \J$ the plane wave centered in $\xE$ and travelling along the direction $\dl$.
As~$\q$ plays the same role as the polynomial degree in the approximation properties of plane wave spaces, we refer to $\q$ as {\em effective plane wave degree}.

Analogously, given any edge $\e \in \En$, we introduce $\PW_p(\e)$ as the span of the traces of plane waves generating the space $\PWE$ on $\e$, namely $\wle(\x):=e^{\im\k \dl \cdot (\x-\xe)}{}_{|_\e}$, $\ell \in \J$.

We note that, in the definition of the bulk and edge plane waves, we also consider a shift by the barycenters of the elements and the midpoints of the edges, respectively.
This actually does not change the nature of the basis since it simply results in a multiplication between a nonshifted plane wave with a constant. However, this additional notation may be of help when implementing the method, as it helps to remember when dealing with bulk and/or edge plane waves,
see Section~\ref{section implementational details}.

It holds that $\dim(\PWE)=\p$ for all $\E\in \taun$, but in general $\dim(\PW_p(\e))\le\p$ for all $\e\in \En$. In fact, if
\begin{equation} \label{filter_rel}
\djj \cdot (\x-\xe)=\dl \cdot (\x-\xe) \quad \forall \x \in \e,
\end{equation}
for some $j,\ell \in \{1,\dots,p\}$, $j>\ell$, then $\wje(\x)=\wle(\x)$ on $e$.

Thus, in order to avoid the presence of linearly dependent edge plane waves, we have to remove redundant plane waves on the edge $\e$.
Further, for theoretical purposes, we also require constant functions to be contained in the edge plane wave spaces in \cite{ncTVEM_theory};
such choice was instrumental for proving best approximation results in terms of functions in nonconforming Trefftz-VE spaces.
Therefore, we add one of the two normal vectors associated with the edge $\e$, whenever it is not already contained in the original set of directions.
This whole procedure goes under the name of \emph{filtering process} and was firstly described in \cite{ncTVEM_theory}. For the sake of completeness, we report it in Algorithm~\ref{algorithm filtering process}.
In Figure \ref{fig:directions after filtering}, we depict all possible configurations of distributions of the plane wave directions over the edges.
\begin{algorithm} 
\caption{\textit{Filtering process}}
\label{algorithm filtering process}
For all edges $e \in \En$:
\begin{enumerate}
\item Remove redundant plane waves
\begin{itemize}
\item Initialize $\Jeprime:=\J:=\{1,\dots,p\}$;
\item For all indices in $\Jeprime$, check whether \eqref{filter_rel} is satisfied;
\item Whenever this is the case for some pair $j,\ell \in \Jeprime$ with $j>\ell$, remove index $j$ from $\Jeprime$;
\end{itemize}
\item Add the constants
\begin{itemize}
\item Check whether there exists a direction $\dstar \in \{\dl\}_{\ell \in \J}$ such that
\begin{equation*}
\dstar \cdot (\x-\xe)=0 \quad \forall \x \in \e;
\end{equation*}
\item If this is the case, set $\Je:=\Jeprime$; otherwise, set $\Je:=\Jeprime \cup \{ p+1 \}$ and $w_{p+1}^e(\x):=1$. 
\end{itemize}
\end{enumerate}
\end{algorithm}
\begin{figure}[h]
\centering
\begin{subfigure}[b]{0.4\textwidth}
\centering
\includegraphics[width=0.8\textwidth]{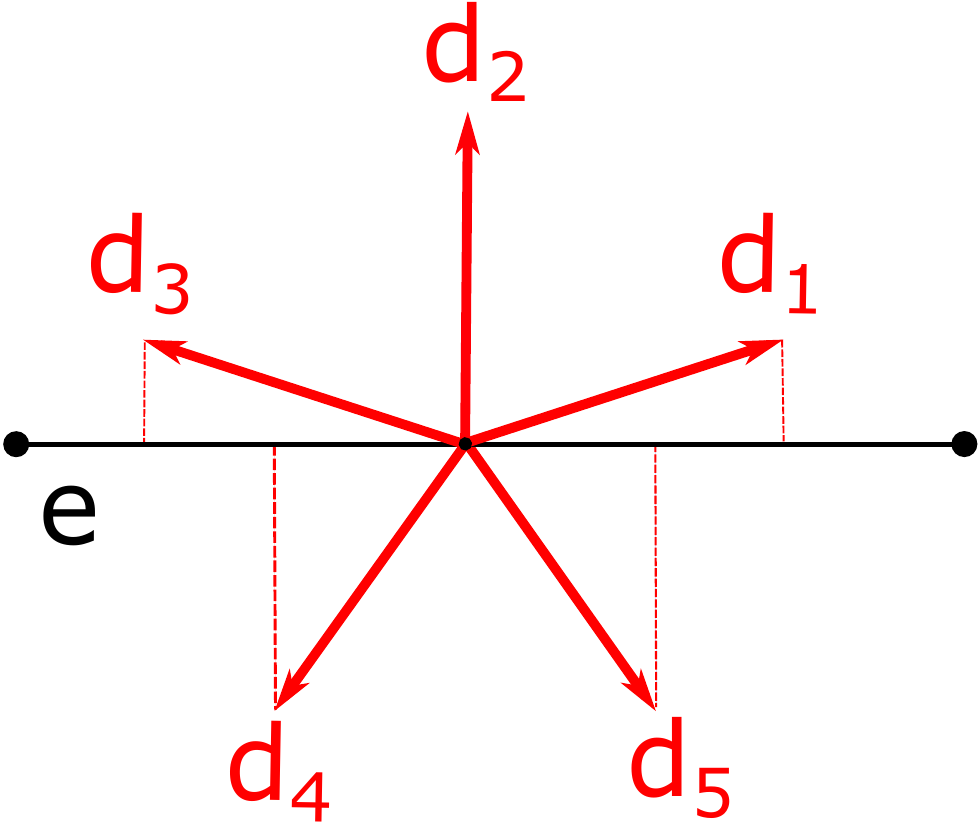}
\caption[]
{{\small No direction eliminated, orthogonal direction already included.}}    
\end{subfigure}
\hfill
\begin{subfigure}[b]{0.4\textwidth}  
\centering 
\includegraphics[width=0.8\textwidth]{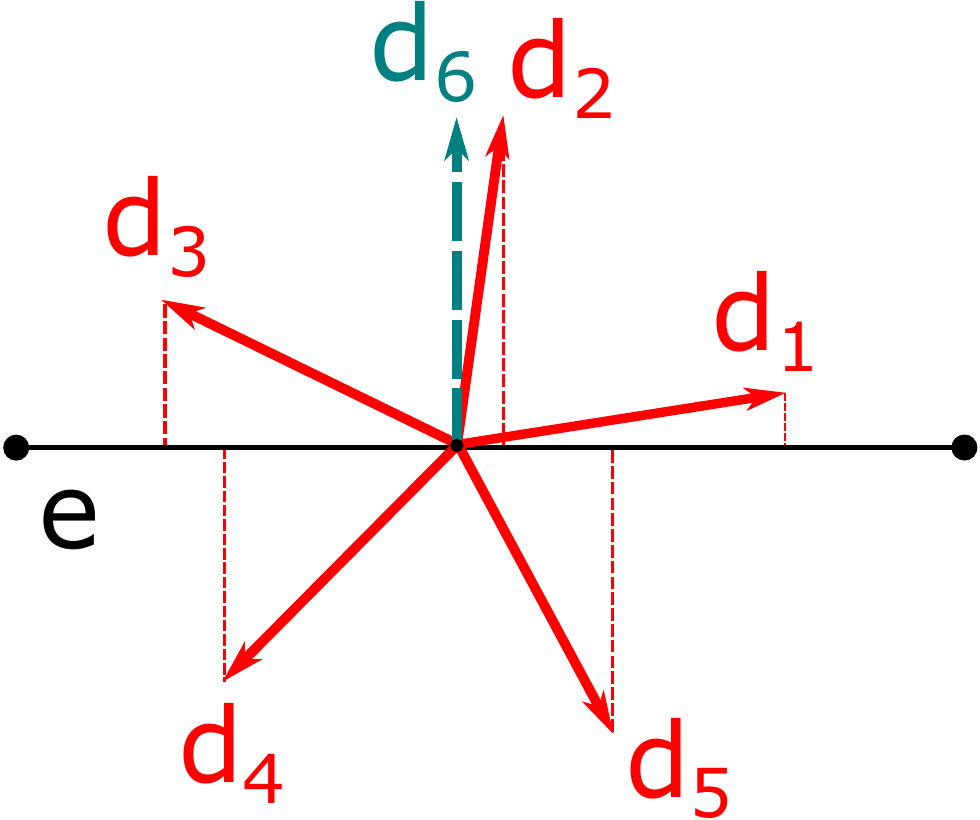}
\caption[]%
{{\small No direction eliminated, orthogonal direction not yet included.}}    
\end{subfigure}
\vskip\baselineskip
\begin{subfigure}[b]{0.4\textwidth}   
\centering 
\includegraphics[width=0.8\textwidth]{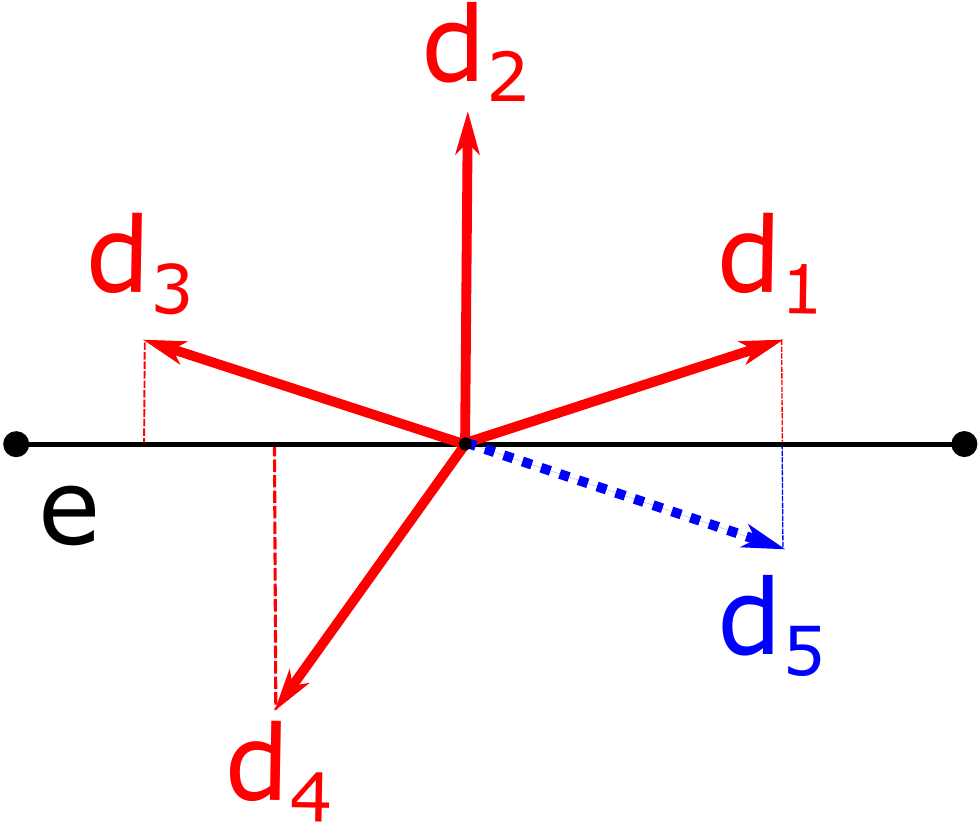}
\caption[]%
{{\small One direction eliminated, orthogonal direction already included.}}    
\end{subfigure}
\hfill
\begin{subfigure}[b]{0.4\textwidth}   
\centering 
\includegraphics[width=0.8\textwidth]{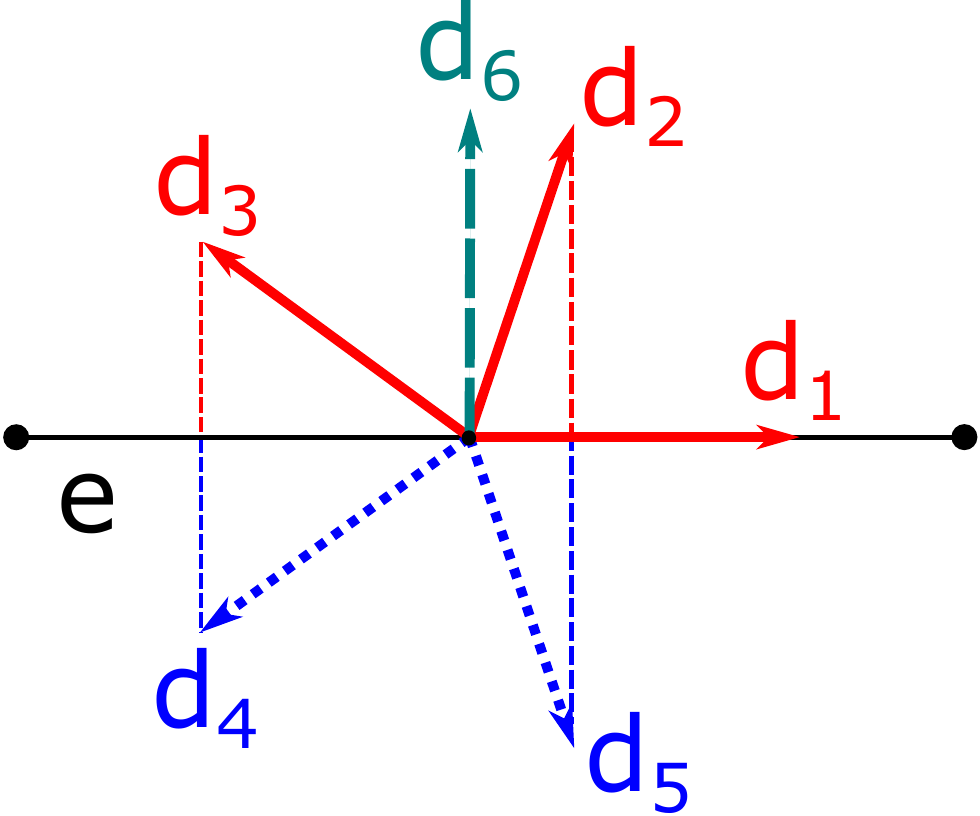}
\caption[]%
{{\small Two directions eliminated, orthogonal direction not yet included.}}    
\end{subfigure}
\caption[The average and standard deviation of critical parameters]
{\small \emph{Filtering process}.
We depict all the possible configurations. In solid lines, the directions that are kept; in dotted lines, the directions that are eliminated accordingly with \eqref{filter_rel}; in dashed lines, the orthogonal direction that has to be possibly added in order to include constants.}
\label{fig:directions after filtering}
\end{figure}

After having performed the filtering process, for every edge $e \in \En$, we define
\begin{equation} \label{edge space with constant}
\PWc(e):=\lin \left\{ \wle \, , \, \ell \in \Je \right\},
\end{equation}
and $\pe:=\dim(\PWc(e)) \le \p+1$.

Next, for any $\E \in \taun$, we introduce  the {\it local} Trefftz-VE space
\begin{equation} \label{local Trefftz-VE space}
\begin{split}
\VE:= \big\{\vn \in H^1(\E) \,\, | \,\, \Delta \vn+k^2 \vn = 0 \, \text{ in }\E, \quad \vn{}_{|_\e} \in \PWc(e) \quad &\forall e \in \EE \cap (\EnD \cup \EnN), \\ \,\gammaIK(\vn)_{|_\e} \in \PWc(e) \quad &\forall e \in \EE \setminus (\EnD \cup \EnN)  \big\},
\end{split}
\end{equation}
where we have set the element impedance trace $\gammaIK(\vn):=\nabla \vn \cdot \nE +\im \k \theta \vn$.

Note that it holds $\PWE \subset \VE$, but $\VE$ also contains other functions whose explicit representation is not available in closed form. This gives rise to the term \textit{virtual} in the name of the method. For future use, we denote $\pE:=\dim(\VE)=\sum_{e \in \EE} \pe$.

Setting $\ME:=\{1,\dots,\Ne\}$, on every $\E \in \taun$, we introduce a set of functionals defined as the moments on each edge $\e_r \in \EE$, $r \in \ME$, with respect to functions in the space $\PWc(e_r)$ given in~\eqref{edge space with constant}:
\begin{equation} \label{dofs}
\dof_{r,j}(\vn):=\frac{1}{\h_{e_r}} \int_{e_r} \vn \overline{w_j^{e_r}} \, \ds \quad \forall r \in \ME, \, \forall j \in \Jer.
\end{equation}
This set constitutes a set of degrees of freedom, as proven in the forthcoming result.

\begin{lem} \label{lemma dofs}
Assume that $\k^2$ is not a Dirichlet-Laplace eigenvalue on the element $\E$. Then, the set of functionals in~\eqref{dofs} defines a set of unisolvent degrees of freedom for the local space $\VE$ introduced in~\eqref{local Trefftz-VE space}.
\end{lem}
\begin{proof}
If $\EE\cap (\EnD \cup \EnN) = \emptyset$, the proof is identical to that of \cite[Lemma 3.1]{ncTVEM_theory}. Otherwise, we observe that, if $\vn \in \VE$ is such that all the associated functionals in~\eqref{dofs} are zero, then ${\vn}_{|_e}=0$ on each edge $e \in \EE\cap (\EnD \cup \EnN)$, due to the fact that ${\vn}_{|_e} \in \PWc(\e)$, together with the definition of the degrees of freedom. This, combined with an integration by parts, leads to
\[
\vert \vn \vert^2_{1,\E} - \k^2 \Vert \vn \Vert^2_{0,\E} - \im \k \theta \Vert \vn \Vert^2_{0,\partial \E \setminus (\GammaD \cup \GammaN)} = \int_{\partial \E \setminus (\GammaD \cup \GammaN)} \vn \overline{\gammaIK(\vn)} \, \ds=0.
\]
Taking the imaginary part finally gives $\vn = 0$ on $\partial \E \backslash (\GammaD \cup \GammaN)$, and therefore $\vn=0$ on $\partial \E$. Next, recalling that $\vn$ belongs to the kernel of the Helmholtz operator and $\k^2$ is not a Dirichlet-Laplace eigenvalue, we deduce $\vn=0$ in $\E$, which is the assertion.
\end{proof}

Having this, the set of local canonical basis functions $\{\varphi_{s,\ell}\}_{s \in \ME, \ell \in \Jes}$ associated with the set of degrees of freedom \eqref{dofs} is defined as
\begin{equation} \label{definition canonical basis}
\dof_{r,j}(\varphi_{s,\ell}) = \delta_{r,s} \delta_{j,\ell}
\quad \forall r,s \in \ME,\, \forall j \in \Jer, \, \forall \ell \in \Jes,
\end{equation}
where $\delta$ is the Kronecker delta.

Next, we construct the global Trefftz-VE space, assuming uniform $p$; the case when $p$ may vary from element to element is discussed in Section \ref{subsubsection hp version} below. We need to fix some additional notation. Firstly, we define the broken Sobolev space associated with the decomposition $\taun$ by
\begin{equation*}
H^1(\taun):=\prod_{\E \in \taun} H^1(\E) = \{ v \in L^2(\Omega): v_{|_\E} \in H^1(\E) \quad \forall \E \in \taun \},
\end{equation*}
endowed with the corresponding weighted broken Sobolev norm
\begin{equation*} 
\lVert v \rVert_{1,\k,\taun}^2:=\sum_{\E \in \taun} \lVert v \rVert_{1,\k,\E}^2 
=\sum_{\E \in \taun} \left( \lvert v \rvert_{1,\E}^2 + \k^2 \lVert v \rVert_{0,\E}^2 \right).
\end{equation*}
Secondly, we pinpoint the global nonconforming Sobolev space associated with $\taun$ incorporating in a nonconforming fashion a Dirichlet boundary datum $\gtilde \in H^{\frac{1}{2}}(\GammaD)$:
\begin{equation} \label{non conforming Sobolev}
\begin{split}
H^{1,nc}_{\gtilde}(\taun):=\{ v \in H^1(\taun)\,: \, \int_{e} (v^+ - v^-) \, \overline{\we} \, \ds &=0 \, \quad \forall \we \in \PWc(e), \, \forall e \in \EnI, \\
\int_{e} (v - \gtilde)  \overline{\we} \, \ds &= 0  \, \quad \forall \we \in \PWc(e), \, \forall e \in \EnD  \},\\
\end{split}
\end{equation}
where, on each internal edge $e \in \mathcal{E}_n^I$ with $e \subseteq \partial{\E^-} \cap \partial {\E^+}$ for some $\E^-$, $\E^+ \in \taun$, the functions $v^-$ and $v^+$ are the Dirichlet traces of $v$ from $\E^-$ and $\E^+$, respectively. 

The \emph{global} nonconforming Trefftz-VE trial and test spaces are given by
\begin{equation} \label{global trial Trefftz space}
\VhgD\:=\{ \vn \in H^{1,nc}_{\gD} (\taun): \, v_{\h|_\E} \in \VE \quad \forall \E \in \taun \}
\end{equation}
and
\begin{equation} \label{global test Trefftz space}
\Vhz\:=\{ \vn \in H^{1,nc}_{0} (\taun): \, v_{\h|_\E} \in \VE \quad \forall \E \in \taun \},
\end{equation}
respectively. In both cases, the set of global degrees of freedom is obtained by coupling the local degrees of freedom on the interfaces between elements.
\begin{remark} \label{remark on D boundary conditions}
Owing to the definition \eqref{non conforming Sobolev}, the Dirichlet boundary conditions are imposed weakly, via the definition of moments with respect to plane waves. At the computational level, one can approximate $\gD$ by taking a sufficiently high-order Gau\ss-Lobatto interpolant.
\end{remark}
With these ingredients at hand, we recall the construction of the method from \cite{ncTVEM_theory}. To this purpose, we first fix the notation for the local sesquilinear forms over $K \in \taun$:
\begin{equation*}
\aE(u,v) := \int_{\E} \nabla u \cdot \overline{\nabla v} \, \dx - \k^2 \int_\E u \overline v \, \dx \quad \forall u,v \,\in H^1(\E).
\end{equation*}
Then, for a given $\E \in \taun$,  we define the local projector
\begin{equation} \label{projector aK}
\begin{split}
\Pip: \, &\VE \rightarrow \PWE \\
&\aE(\Pip \un, \wE) = \aE(\un, \wE) \quad \forall \un \in \VE,\, \forall \wE \in \PWE.
\end{split}
\end{equation}
Using an integration by parts, one can observe that $\Pip$ is indeed computable without explicit knowledge of the Trefftz-VE functions in the bulk of $\E$, thanks to the choice of the degrees of freedom in \eqref{dofs}.

\begin{remark}\label{remark eig}
In~\cite[Proposition 3.2]{ncTVEM_theory}, it was proven that, whenever $\k^2$ is not a Neumann-Laplace eigenvalue in~$\E$, the projector $\Pip$ in~\eqref{projector aK} is well-defined and continuous. In order to numerically investigate this condition,
we plot the minimal (absolute) eigenvalues of the matrix $\boldsymbol{A}^{\widehat{\E}}:=\{ a^{\widehat{\E}}(w_\ell^{\widehat{\E}},w_j^{\widehat{\E}}) \}_{\ell,j=1,\dots,p}$ in terms of the wave number $k$ on the reference element $\widehat{\E}=(0,1)^2$, see Figure \ref{fig:min_eig}. On this domain, the Neumann-Laplace eigenvalues $\nu_{m,n}$ are known explicitly:
\begin{equation*}
\nu_{m,n}=\pi^2 (m^2+n^2), \quad m,n \in \N_0.
\end{equation*}
\begin{figure}[h]
\begin{center}
\includegraphics[width=0.5\textwidth]{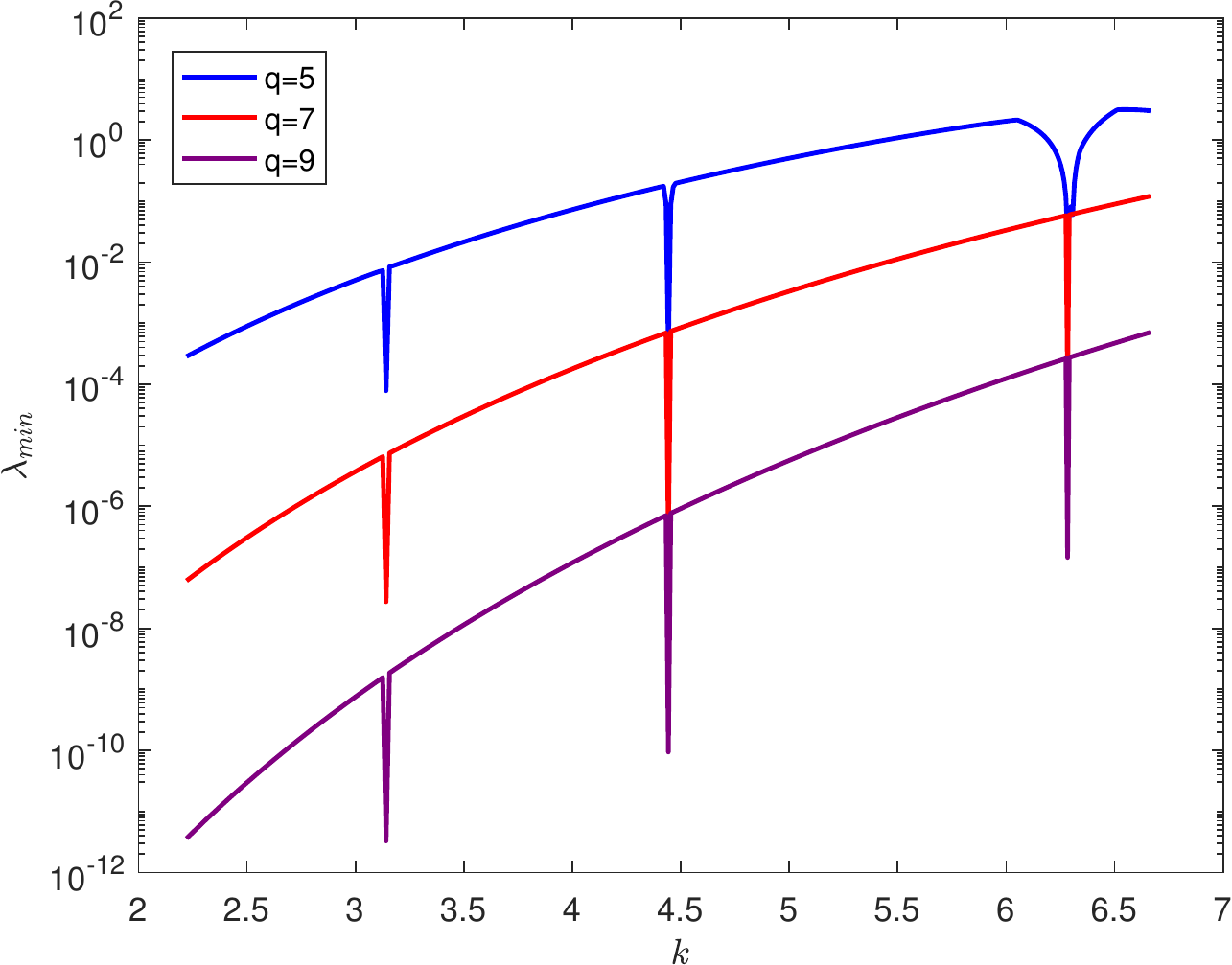}
\end{center}
\caption{Minimal (absolute) eigenvalues of the matrix $\boldsymbol{A}^{\widehat{\E}}$, see Remark \ref{remark eig}.}
\label{fig:min_eig}
\end{figure}
We observe that, for wave numbers $k$ close to the square roots of the eigenvalues $\nu_{m,n}$, the minimal (absolute) eigenvalue of $\boldsymbol{A}^{\widehat{\E}}$ is actually some orders of magnitude lower than outside the neighborhoods of $\sqrt{\nu_{m,n}}$.
Therefore, when~$\k^2$ is close to a Neumann-Laplace eigenvalue, the continuity constant of $\Pip$ may deteriorate.

\end{remark}

On any boundary edge $\e \in \Enb$, denoting by $\Ee \in \taun$ the adjacent element of $e$, we further set the $L^2(\e)$ projector 
\begin{equation} \label{projector edge L2}
\begin{split}
\Pie: \, & \VEe_{|_e} \rightarrow \PWc(\e) \\
&\int_e (\Pie \un) \overline{\we} \, \ds = \int_e \un \overline{\we} \, \ds \quad \forall \un\in\VEe,\,\forall \we \in \PWc(\e).
\end{split}
\end{equation}
Again by~\eqref{dofs}, this projector is computable as well. In the sequel, we will use the notation $\Pi_p^{0,\omega}$ to denote the $L^2$ projector onto the space $\prod_{e \in \omega} \PWc(\e)$
defined edgewise by~\eqref{projector edge L2}, where $\omega$ is either $\GammaR$ or $\GammaN$.

We highlight that the method is not obtained by simply substituting the spaces $\VgD$ and $\Vz$ in~\eqref{weak continuous problem} by the discrete spaces $\VhgD$ and $\Vhz$.
In fact, on the one hand, an explicit representation of Trefftz-VE functions is not elementwise available in closed form, and hence $\a(\un,\vn)$ is not computable by means of the degrees of freedom \eqref{dofs} for all $\un\in \VhgD$ and $\vn \in \Vhz$.
On the other, Dirichlet traces of Trefftz-VE functions are unknown on  $\GammaR$, see \eqref{decomposition boundary}, and therefore $\langle F ,\vn \rangle$ and the term $ \im \k \theta \int_{\Gamma_R} \un \overline \vn \, \ds$ cannot be computed for all $\vn \in \Vhz$.

Following the standard VEM gospel \cite{VEMvolley}, we replace the original sesquilinear forms and right-hand sides with some computable counterparts. More precisely:
\begin{enumerate}[label=(\roman*)]
\item In order to find a suitable computable substitute for the sesquilinear form in~\eqref{weak continuous problem}
\begin{equation*}
b(u,v)=\sum_{\E \in \taun} \left[ \int_{\E} \nabla u \cdot \overline{\nabla v} \, \dx - \k^2 \int_{\E} u \overline{v} \, \dx \right] + \im\k \int_{\GammaR} u \overline{v} \, \ds,
\end{equation*}

we first make use of the definition of the projector $\Pip$ in \eqref{projector edge L2}, obtaining, for the bulk term,
\begin{equation*} 
\aE(\un,\vn) = \aE( \Pip \un, \Pip \vn) + \aE( (I-\Pip) \un, (I-\Pip) \vn) \quad \forall \un,\,\vn \in \VE.
\end{equation*}
The first term on the right-hand side is computable, but the second one is not. Hence, the latter is substituted by a proper computable sesquilinear form $\SE(\cdot,\cdot)$ mimicking $\aE(\cdot, \cdot)$, and referred to in the following as \textit{stabilization}. Therefore, we are able to introduce local discrete sesquilinear forms
\begin{equation} \label{local discrete bf}
\anE(\un,\vn) := \aE( \Pip \un, \Pip \vn) + \SE\left( (I-\Pip) \un, (I-\Pip) \vn \right) \quad \forall \un,\,\vn \in \VE.
\end{equation}
In order to guarantee the well-posedness of the method, some conditions on the choice of $\SE(\cdot,\cdot)$ are needed, see \cite[Proposition 3.4, Theorem 4.3]{ncTVEM_theory}.
We anticipate that, in Section \ref{section hversion num recipe}, we will discuss the effects of the choice of the stabilization on the numerical performance of the method. It is important to mention that the local sesquilinear form is \textit{consistent} in the sense that
\begin{equation} \label{consistency}
\anE(\vn, \wE) = \aE(\vn, \wE),\quad \anE(\wE, \vn) = \aE(\wE, \vn)\quad \forall \vn \in \VE,\, \forall \wE \in \PWE.
\end{equation}

The boundary term is instead discretized by
\begin{equation*}
\im\k \theta \int_{\GammaR} u \overline{v} \, \ds
\quad \mapsto \quad \im\k \theta \int_{\GammaR} (\PiGammaR \un) \overline{(\PiGammaR \vn)} \, \ds \quad \forall \un \in \VhgD,\, \forall \vn \in \Vhz.
\end{equation*}
Altogether, $b(u, v)$ is discretized by
\begin{equation} \label{definition bn}
\bn(\un,\vn):=\an(\un,\vn) + \im\k\theta \int_{\GammaR} (\PiGammaR \un) \overline{(\PiGammaR \vn)} \, \ds  \quad \forall \un \in \VhgD,\, \forall \vn \in \Vhz,
\end{equation}
with
\begin{equation} \label{definition ah}
\an(\un,\vn):=\sum_{\E \in \taun} \anE(\un,\vn) \quad \forall \un \in \VhgD,\, \forall \vn \in \Vhz.
\end{equation}
\item The functional
\begin{equation*}
\langle F,v \rangle=\int_{\GammaN} \gN \overline v \, \ds + \int_{\GammaR} \gR \overline v \, \ds
\end{equation*}
on the right-hand side of~\eqref{weak continuous problem} is discretized by
\begin{equation} \label{discrete rhs}
\langle \fn,\vn \rangle:= \int_{\GammaN} \gN \overline{(\PiGammaN \vn)}  + \int_{\GammaR} \gR \overline{(\PiGammaR \vn)} \, \ds \quad \vn \in \Vhz.
\end{equation}
\end{enumerate}
With these definitions, the nonconforming Trefftz-VEM reads as follows:
\begin{equation} \label{complete method}
\begin{cases}
\text{find } \un \in \VhgD \text{ such that}\\
\bn(\un,\vn) = \langle \fn,\vn \rangle \quad  \forall \vn \in  \Vhz,\\
\end{cases}
\end{equation} 
where $\bn(\cdot,\cdot)$ and $\langle \fn,\cdot \rangle$ are given in \eqref{definition bn} and \eqref{discrete rhs}, respectively.

In \cite{ncTVEM_theory}, an abstract error analysis of the method \eqref{complete method}, along with $\h$-convergence results, was proven for the case that $\GammaR = \partial \Omega$.

\section{Details on the implementation} \label{section implementational details}
In this section, we give some details concerning the implementation of the method \eqref{complete method}, involving in particular the computation of the two projectors $\Pip$ and $\Pie$ introduced in \eqref{projector aK} and \eqref{projector edge L2}, respectively.
We point out that, despite the setting of the method~\eqref{complete method} is rather different from that of standard VEM, the implementation follows the same lines; hence, we will employ the same ideas and notation as in~\cite{hitchhikersguideVEM}.

\subsection{Assembly of the global system of linear equations}
The global system of linear equations corresponding to the method \eqref{complete method} is assembled as in the standard nonconforming VEM \cite{nonconformingVEMbasic, ncHVEM} and FEM \cite{CrouzeixRaviart}.
For the sake of clarity, we first consider the case that $\GammaD=\emptyset$. The general case will be addressed in Section \ref{subsection gammaD} below.

Given $\Nedg$ the total number of edges of the mesh $\taun$, let $\{\varphi_{\tilde{s},\tilde{\ell}}\}_{\tilde{s}=1,\dots,\Nedg,\, \tilde{\ell}\in {\Je}_{\tilde{s}}}$ be the set of canonical basis functions given by~\eqref{definition canonical basis}. In this section, we use the convention that the indices hooded by a tilde denote global indices, whereas those without stand for local ones. 

Expanding $\un$ as $\sum_{\tilde{s}=1}^{\Nedg} \sum_{\tilde{\ell}=1}^{{\pe}_{\tilde{s}}}  u_{\tilde{s},\tilde{\ell}} \varphi_{\tilde{s},\tilde{\ell}}$ and plugging this ansatz into~\eqref{complete method} lead to 
\begin{equation} \label{set of equations}
\begin{split}
\sum_{\tilde{s}=1}^{\Nedg} \sum_{\tilde{\ell}=1}^{{\pe}_{\tilde{s}}} &u_{\tilde{s},\tilde{\ell}} \, \left[ \an(\varphi_{\tilde{s},\tilde{\ell}},\varphi_{\tilde{r},\tilde{j}}) + \im\k\theta \int_{\GammaR} (\PiGammaR \varphi_{\tilde{s},\tilde{\ell}}) \overline{(\PiGammaR \varphi_{\tilde{r},\tilde{j}})} \, \ds \right] \\
&= \int_{\GammaN} \gN \overline{(\PiGammaN \varphi_{\tilde{r},\tilde{j}})}  + \int_{\GammaR} \gR \overline{(\PiGammaR \varphi_{\tilde{r},\tilde{j}})} \, \ds \quad \forall \tilde{r}=1,\dots,\Nedg, \, \forall \tilde{j}=1,\dots,{\pe}_{\tilde{r}},
\end{split}
\end{equation} 
where, with a slight abuse of notation, we relabelled by $1,\dots,{\pe}_{\tilde{s}}$ the indices in ${\Je}_{\tilde{s}}$ that remain after the filtering process similarly for the ones in ${\Je}_{\tilde{r}}$. 

We observe that~\eqref{set of equations} can be represented as  the linear system
\begin{equation} \label{linear system}
(\boldsymbol{A}+\boldsymbol{R}) \boldsymbol{u}=\boldsymbol{f},
\end{equation}
where $\boldsymbol{A}, \boldsymbol{R} \in \C^{\Nd \times \Nd}$, $\boldsymbol{u} \in \C^{\Nd}$, and $\boldsymbol{f} \in \C^{\Nd}$, $\Nd$ being the total number of global degrees of freedom, are matrices and vectors with entries defined by
\begin{equation*}
\begin{alignedat}{3}
\boldsymbol{A}_{(\tilde{r},\tilde{j}),(\tilde{s},\tilde{\ell})} &=\an(\varphi_{\tilde{s},\tilde{\ell}},\varphi_{\tilde{r},\tilde{j}}), \quad \quad &\boldsymbol{R}_{(\tilde{r},\tilde{j}),(\tilde{s},\tilde{\ell})}&=\im\k\theta \int_{\GammaR} (\PiGammaR \varphi_{\tilde{s},\tilde{\ell}}) \overline{(\PiGammaR \varphi_{\tilde{r},\tilde{j}})} \, \ds, \\
\boldsymbol{u}_{(\tilde{s},\tilde{\ell})}&=u_{\tilde{s},\tilde{\ell}},
&\boldsymbol{f}_{(\tilde{r},\tilde{j})}&=\int_{\GammaN} \gN \overline{(\PiGammaN \varphi_{\tilde{r},\tilde{j}})}  + \int_{\GammaR} \gR \overline{(\PiGammaR \varphi_{\tilde{r},\tilde{j}})} \, \ds.
\end{alignedat}
\end{equation*}
Note that here the subindex $(\tilde{r},\tilde{j})$ is associated with the index $\sum_{\tilde{t}=1}^{\tilde{r}-1} p_{e_{\tilde{t}}}+\tilde{j}$. The computation of $\boldsymbol{A}$, $\boldsymbol{R}$, and $\boldsymbol{f}$ are described in the forthcoming Sections \ref{subsection A}, \ref{subsection Robin}, and \ref{subsection g}, respectively.
\subsection{Computation of the matrix $\boldsymbol{A}$} \label{subsection A}
Using the definition of $\an(\cdot,\cdot)$ in \eqref{definition ah}, together with \eqref{local discrete bf}, we have 
\begin{equation} \label{matrix AK}
\boldsymbol {A}_{(\tilde r, \tilde j), (\tilde s, \tilde \ell)}
= \an(\varphi_{\tilde{s},\tilde{\ell}},\varphi_{\tilde{r},\tilde{j}})=\sum_{\E \in \taun} \left[\aE( \Pip \varphi_{\tilde{s},\tilde{\ell}}, \Pip \varphi_{\tilde{r},\tilde{j}}) + \SE\left( (I-\Pip) \varphi_{\tilde{s},\tilde{\ell}}, (I-\Pip) \varphi_{\tilde{r},\tilde{j}} \right)\right].
\end{equation}
The global matrix $\boldsymbol{A}$ is then assembled by means of the local matrices $\boldsymbol{A}^\E\in \C^{\pE \times \pE}$ that are defined as
\[
\boldsymbol{A}^\E_{(r, j), (s,  \ell)} = \aE( \Pip \varphi_{s,\ell}, \Pip \varphi_{r, j}) + \SE\left( (I-\Pip) \varphi_{ s , \ell}, (I-\Pip) \varphi_{ r , j } \right),
\]
where $\{\varphi_{s,\ell}\}_{s \in \ME,\, \ell \in \Jes}$ denotes the local basis of $\VE$.

Following~\cite{hitchhikersguideVEM}, the computation of such local matrices is performed in various steps.

\paragraph*{Computation of the bulk projector $\Pip$ in~\eqref{projector aK}.}
Let $\varphi_{s,\ell} \in \VE$, $s \in \ME$, $\ell \in \Jes$, be the canonical basis function.
As a first step, we write $\Pip \varphi_{s,\ell} \in \PWE$ as a linear combination of the plane waves $\wzE \in \PWE$, $\zeta=1,\dots,p$,
\begin{equation*}
\Pip \varphi_{s,\ell}=\sum_{\zeta=1}^{p} \gamma_\zeta^{\E (s,\ell)} \wzE.
\end{equation*}
Plugging this ansatz into \eqref{projector aK} and testing with plane waves lead to the system of linear equations
\begin{equation*}
\boldsymbol{G}^\E \boldsymbol{\gamma}^{\E(s,\ell)} = \boldsymbol{b}^{\E(s,\ell)},
\end{equation*}
where $\boldsymbol{G}^\E \in \C^{p \times p}$, $\boldsymbol{\gamma}^{\E(s,\ell)} \in \C^p$, $\boldsymbol{b}^{\E(s,\ell)} \in \C^p$, for all $s \in \ME$ and $\ell \in \Jes$, are defined as
\begin{align*} 
\boldsymbol{G}^\E := \begin{bmatrix}
\aE(w_1^\E,w_1^\E) & \cdots & \aE(w_p^\E,w_1^\E) \\
\vdots & \ddots & \vdots \\
\aE(w_1^\E,w_p^\E) & \cdots & \aE(w_p^\E,w_p^\E)
\end{bmatrix}, \quad 
\boldsymbol{\gamma}^{\E(s,\ell)} := \begin{bmatrix}
\gamma_1^{\E(s,\ell)} \\
\vdots \\
\gamma_{p}^{\E(s,\ell)}
\end{bmatrix}, \quad
\boldsymbol{b}^{\E(s,\ell)} := 
\begin{bmatrix}
\aE(\varphi_{s,\ell},w_1^\E) \\
\vdots \\
\aE(\varphi_{s,\ell},w_p^\E) 
\end{bmatrix}.
\end{align*}
Collecting columnwise the $\boldsymbol{b}^{\E(s,\ell)}$ leads to a matrix $\boldsymbol{B}^\E:= \left[ \boldsymbol{b}^{\E(1,1)},\dots,\boldsymbol{b}^{\E(\Ne,{\pe}_{\Ne})} \right] \in \C^{p \times \pE}$.

The matrix $\boldsymbol{\Pi}^\E_\star$ representing the action of $\Pip$ from $\VE$ into $\PWE$ is then given by
\begin{equation} \label{matrix representation Pi_star}
\boldsymbol{\Pi}^\E_\star = (\boldsymbol{G}^\E)^{-1} \boldsymbol{B}^\E \in \C^{p \times \pE}.
\end{equation}
We introduce next the matrix
\begin{equation*}
\boldsymbol{D}^\E:=
\begin{bmatrix}
\dof_{1,1}(w_1^\E) & \cdots & \dof_{1,1}(w_p^\E) \\
\vdots & \ddots & \vdots \\
\dof_{\Ne,{\pe}_{\Ne}}(w_1^\E) & \cdots & \dof_{\Ne,{\pe}_{\Ne}}(w_p^\E)
\end{bmatrix}
\in \C^{\pE \times p}.
\end{equation*}
Then, as in \cite{hitchhikersguideVEM}, the matrix $\boldsymbol{\Pi}^\E$ representing the composition of the embedding of $\PWE$ into $\VE$ after $\Pip$ can be expressed as
\begin{equation} \label{definition Pi 2}
\boldsymbol{\Pi}^\E = \boldsymbol{D}^\E (\boldsymbol{G}^\E)^{-1} \boldsymbol{B}^\E \in \C^{\pE \times \pE}.
\end{equation}

\paragraph*{Matrix representation of $\boldsymbol{A}^\E \in \C^{\pE \times \pE}$.}
The local VE stiffness matrix $\boldsymbol{A}^\E$ is given by
\begin{equation} \label{matrix A}
\boldsymbol{A}^\E = \overline{(\boldsymbol{\Pi}^\E_\star)}^T \boldsymbol{G}^\E \boldsymbol{\Pi}^\E_\star + \overline{(\boldsymbol{I}^\E-\boldsymbol{\Pi}^\E)}^T \boldsymbol{S}^\E (\boldsymbol{I}^\E-\boldsymbol{\Pi}^\E),
\end{equation}
where $\boldsymbol{I}^\E$ denotes the identity matrix of size $\pE \times \pE$, and $\boldsymbol{S}^\E$ is the matrix representation of the local stabilization forms $\SE(\cdot,\cdot)$; for a specific choice of the stabilization, we refer to Section \ref{section numerical results} below. Further, note that by using \eqref{matrix representation Pi_star}, it holds 
\begin{equation*}
\overline{(\boldsymbol{\Pi}^\E_\star)}^T \boldsymbol{G}^\E \boldsymbol{\Pi}^\E_\star = \overline{(\boldsymbol{B}^\E)}^T \overline{(\boldsymbol{G}^\E)}^{-T} \boldsymbol{B}^\E.
\end{equation*}

\subsubsection{Computation of the local matrices $\boldsymbol{G}^\E$, $\boldsymbol{B}^\E$, and $\boldsymbol{D}^\E$}
The matrices $\boldsymbol{G}^\E$, $\boldsymbol{B}^\E$, and $\boldsymbol{D}^\E$ can actually be computed exactly without numerical integration, but rather by using the definition of the degrees of freedom in \eqref{dofs} and the formula 
\begin{equation} \label{clever formula}
\Phi(z):=\int_{0}^{1} e^{zt} \textup{d}t
=\begin{cases}
\frac{e^z-1}{z} & \text{if } z \neq 0 \\
1 & \text{if } z=0
\end{cases}\quad \forall  z \in \C.
\end{equation}
This has been already investigated in~\cite{Helmholtz-VEM,gittelson}.

\paragraph*{Computation of $\boldsymbol{G}^\E \in \C^{p \times p}$.} 
Given $j,\ell \in \J$, we compute, by using an integration by parts and taking into account the definition of the bulk plane waves $\wjE$ and $\wlE$, respectively,
\begin{equation*}
\begin{split}
\boldsymbol{G}^\E_{j,\ell}&=\aE(\wlE,\wjE)=\sum_{r=1}^{\Ne} \int_{e_r} (\nabla \wlE \cdot {\nE}_{|_{e_r}}) \overline{\wjE} \, \ds \\
&=\im\k \sum_{r=1}^{\Ne} e^{\im\k(\djj-\dl) \cdot \xE} (\dl \cdot {\nE}_{|_{e_r}}) \int_{e_r} e^{\im\k(\dl-\djj) \cdot \x} \, \ds.
\end{split}
\end{equation*}
The integral over the edges $e_r$, $r \in \ME$, on the right-hand side can be computed by application of the transformation rule. In fact, denoting by $\boldsymbol{a}_r$ and $\boldsymbol{b}_r$ the endpoints of the edge $e_r$, we obtain
\begin{equation} \label{formula pw edge integral}
\begin{split}
\int_{e_r} e^{\im\k(\dl-\djj) \cdot \x} \, \ds
&=h_{e_r} e^{\im\k(\dl-\djj) \cdot \boldsymbol{a}_r} \int_0^1 e^{\im\k(\dl-\djj) \cdot (\boldsymbol{b}_r-\boldsymbol{a}_r)t} \, \textup{d}t \\
&=h_{e_r} e^{\im\k(\dl-\djj) \cdot \boldsymbol{a}_r} \Phi\left(\im\k(\dl-\djj) \cdot (\boldsymbol{b}_r-\boldsymbol{a}_r)\right),
\end{split}
\end{equation}
where $\Phi$ is defined in \eqref{clever formula}.

\paragraph*{Computation of $\boldsymbol{B}^\E  \in \C^{p \times \pE}$.}
Given $s \in \ME$, $\ell \in \Jes$, $j \in \J$, an integration by parts, the definitions of the local canonical basis functions in~\eqref{definition canonical basis}, and the definition of the degrees of freedom in~\eqref{dofs} yield
\begin{equation*} 
\begin{split}
\boldsymbol{B}^\E_{j,(s,\ell)} = \aE(\varphi_{s,\ell},\wjE) 
&= \sum_{r=1}^{\Ne} \int_{e_r} \varphi_{s,\ell} \, \overline{(\nabla \wjE \cdot {\nE}_{|_{e_r}})} \, \ds
= -\im\k \sum_{r=1}^{\Ne} (\djj \cdot {\nE}_{|_{e_r}})  \int_{e_r} \varphi_{s,\ell} \, \overline{\wjE} \, \ds \\
&= -\im\k (\djj \cdot {\nE}_{|_{e_s}}) e^{-\im\k \djj \cdot (\x_{e_s}-\xE)} \int_{e_s} \varphi_{s,\ell} \, \overline{\underbrace{e^{\im\k \djj \cdot (\x-\x_{e_s})}}_{=w_t^{e_s}}} \, \ds \\
&= -\im\k (\djj \cdot {\nE}_{|_{e_s}}) e^{-\im\k \djj \cdot (\x_{e_s}-\xE)} h_{e_s} \delta_{t,\ell}.
\end{split}
\end{equation*}
where $t \in \Jes$ is the local index such that $w_t^{e_s}=e^{\im\k \djj \cdot (\x-\x_{e_s})}$ on $e_s$.

\paragraph*{Computation of $\boldsymbol{D}^\E \in \C^{\pE \times p}$.} 
Given $r \in \ME$, $j \in \mathcal{J}_{e_r}$, $\ell \in \J$, a direct computation gives
\begin{equation*}
\dof_{r,j}(\wlE)=\frac{1}{h_{e_r}} \int_{e_r} \wlE \overline{w_j^{e_r}} \, \ds 
=\frac{1}{h_{e_r}} e^{\im\k (\djj \cdot \x_{e_r} - \dl \cdot \xE)} \int_{e_r} e^{\im\k (\dl-\djj) \cdot \x} \, \ds.
\end{equation*}
The last term on the right-hand side can be computed as in \eqref{formula pw edge integral}.

\subsection{Computation of the Robin boundary matrix $\boldsymbol{R}$} \label{subsection Robin}
Recall that the Robin boundary matrix $\boldsymbol{R}$ is given by 
\begin{equation} \label{matrix Re}
\boldsymbol{R}_{(\tilde{r},\tilde{j}),(\tilde{s},\tilde{\ell})}=\im\k\theta \sum_{\e \in \EnR} \int_{\e} (\Pie \varphi_{\tilde{s},\tilde{\ell}}) \overline{(\Pie \varphi_{\tilde{r},\tilde{j}})} \, \ds.
\end{equation}
Similarly as above, the global matrix $\boldsymbol{R}$ is assembled by means of the local matrices $\boldsymbol{R}^\e \in \C^{\pe \times \pe}$ that are defined as
\[
\boldsymbol{R}^\e_{(r, j), (s,  \ell)} = \im\k\theta  \int_{\e} (\Pie \varphi_{s,\ell}) \overline{(\Pie \varphi_{r, j})} \, \ds,
\]
where $\{\varphi_{s,\ell}\}_{s \in \ME,\, \ell \in \Jes}$ denotes the local basis of $\VE$, with $\E$ such that $\e \subset \partial \E\cap  \GammaR$.

Let $e \in \EnR$ be a fixed boundary edge in $\EnR$ with local index $z \in \ME$, where $\E \in \taun$ is the unique polygon with $\e = \partial \E \cap \GammaR$.

\paragraph*{Computation of the edge projector $\Pie$ in \eqref{projector edge L2}.}
Let $\varphi_{z,\ell} \in \VE$, $\ell \in \Je$, be a fixed function of the local canonical basis. We first expand $\Pie \varphi_{z,\ell} \in \PWc(\e)$ in terms of $\wetae \in \PWc(\e)$, $\eta=1,\dots,\pe$, 
\begin{equation*}
\Pie \varphi_{z,\ell} = \sum_{\eta=1}^{\pe} \beta_\eta^{\e(\ell)} \wetae.
\end{equation*}
Inserting this ansatz into \eqref{projector edge L2} and testing with edge plane waves lead to the linear system
\begin{equation*}
\boldsymbol{G}_0^e \boldsymbol{\beta}^{\e(\ell)} = \boldsymbol{b}_0^{e(\ell)}.
\end{equation*}
Here, $\boldsymbol{G}_0^e \in \C^{\pe \times \pe}$, $\boldsymbol{\beta}^{\e(\ell)} \in \C^{\pe}$, $\boldsymbol{b}_0^{\e(\ell)} \in \C^{\pe}$ for all $\ell\in \Je$, are defined as
\begin{align} \label{VEM edge matrices}
\boldsymbol{G_0^e} := \begin{bmatrix}
(w_1^e,w_1^e)_{0,e} & \cdots & (w_{\pe}^e,w_1^e)_{0,e} \\
\vdots & \ddots & \vdots \\
(w_1^e,w_{\pe}^e)_{0,e} & \cdots & (w_{\pe}^e,w_{\pe}^e)_{0,e}
\end{bmatrix}, \quad 
\boldsymbol{\beta}^{\e(\ell)} := \begin{bmatrix}
\beta_1^{\e(\ell)} \\
\vdots \\
\beta_{\pe}^{\e(\ell)}
\end{bmatrix}, \quad
\boldsymbol{b}_0^{\e(\ell)} := 
\begin{bmatrix}
(\varphi_{z,\ell},w_1^e)_{0,e} \\
\vdots \\
(\varphi_{z,\ell},w_{\pe}^e)_{0,e} 
\end{bmatrix},
\end{align}
where $(\cdot,\cdot)_e$ denotes the complex $L^2$ inner product over $\e$. Note that in fact $\boldsymbol{G_0^e} \in \R^{\pe \times \pe}$, see \eqref{G0e} below.
Moreover, such matrix is positive definite for all $\E \in \taun$, and thus also invertible. Nevertheless, it is worth to underline  that in presence of small-sized elements and of a large number of plane waves, such matrix may become singular in machine precision.
This problem will be analyzed in Section~\ref{section numerical results standard} and addressed in Section~\ref{section cure illconditioning}.

Consequently, collecting the $\boldsymbol{b}_0^{e(\ell)}$ columnwise into a matrix $\boldsymbol{B}_0^e \in \C^{\pe \times \pE}$, the matrix representation of $\Pie$ is given by
\begin{equation*}
\boldsymbol{\Pi}_\star^{0,e}=(\boldsymbol{G}_0^e)^{-1} \boldsymbol{B}_0^e.
\end{equation*}
\paragraph*{Matrix representation of $\boldsymbol{R}^e$.}
The local edge VE boundary mass matrix $\boldsymbol{R}^e$ has the form 
\begin{equation} \label{computation R}
\boldsymbol{R}^e=\overline{\boldsymbol{\Pi}_\star^{0,e}}^T \boldsymbol{G}_0^e \boldsymbol{\Pi}_\star^{0,e}
=\overline{\boldsymbol{B}_0^e}^T (\overline{\boldsymbol{G}_0^e})^{-T} \boldsymbol{B}_0^e.
\end{equation}

\subsubsection{Computation of the local matrices $\boldsymbol{G}_0^e$ and $\boldsymbol{B}_0^e$}
The matrices $\boldsymbol{G}_0^e$ and $\boldsymbol{B}_0^e$ can be computed exactly using the formula~\eqref{clever formula}.

\paragraph*{Computation of $\boldsymbol{G}_0^e \in \R^{\pe \times \pe}$.} \label{subsection G0}
Given $j,\ell \in \Je$ and denoting by $\boldsymbol{a}$ and $\boldsymbol{b}$ the endpoints of the edge $e$, it holds $(\boldsymbol{G}_0^e)_{j,j}=\he$ and, if $j \neq \ell$,
\begin{equation} \label{G0e}
(\boldsymbol{G}_0^e)_{j,\ell}=(\wle,\wje)_{0,e}
= e^{\im\k (\djj-\dl) \cdot \xe} \int_{e} e^{\im\k (\dl-\djj) \cdot \x} \, \ds
=2\he \frac{\sin\left(\k(\dl-\djj) \cdot \frac{\boldsymbol{b}-\boldsymbol{a}}{2} \right)}{\k (\dl-\djj) \cdot (\boldsymbol{b}-\boldsymbol{a})} \in \R,
\end{equation}
where we used \eqref{formula pw edge integral} and the property $\sin(z)=\frac{1}{2 \im}(e^{\im z}-e^{-\im z})$, $z \in \C$, in the last equality.

\paragraph*{Computation of $\boldsymbol{B}_0^e \in \C^{\pe \times \pE}$.} \label{subsection B0}
For all $j, \ell \in \Je$, the definition of the degrees of freedom in \eqref{dofs} implies
\begin{equation*}
(\boldsymbol{B}_0^e)_{j,\ell}=(\varphi_{z,\ell},w_j^{e})_{0,e}
= \int_{e} \varphi_{z,\ell} \, \overline{w_j^{e}} \, \ds 
= h_{e} \delta_{j,\ell}.
\end{equation*}

\subsection{Computation of the right-hand side vector $\boldsymbol{f}$} \label{subsection g}
Recall that $\boldsymbol{f}$ is given by
\begin{equation*}
\boldsymbol{f}_{(\tilde{r},\tilde{j})}=
\sum_{e \in \EnN} \int_{e} \gN \overline{(\Pie \varphi_{\tilde{r},\tilde{j}})}\,\ds  + \sum_{e \in \EnR} \int_{e} \gR \overline{(\Pie \varphi_{\tilde{r},\tilde{j}})} \, \ds
:= \boldsymbol{f}^N_{(\tilde{r},\tilde{j})} + \boldsymbol{f}^R_{(\tilde{r},\tilde{j})}.
\end{equation*}
Once again, the global right-hand side $\boldsymbol{f}$ is assembled by means of the local vectors $\boldsymbol{f}^{N,e} \in \C^{\pe}$ and $\boldsymbol{f}^{R,e} \in \C^{\pe}$ that are defined as
\[
\boldsymbol{f}^{N,\e}_{(r,j)} = \int_{e} \gN \overline{(\Pie \varphi_{r, j})} \, \ds,\quad \boldsymbol{f}^{R,\e}_{(r,j)} = \int_{e} \gR \overline{(\Pie \varphi_{r,  j})} \, \ds,
\]
where $\{\varphi_{s,\ell}\}_{s \in \ME,\, \ell \in \Jes}$ denotes the local basis of $\VE$, with $\E$ such that either $\e \subset \partial \E\cap  \GammaN$ or $\e \subset \partial \E\cap  \GammaR$.

We only show the details concerning the computation of $\boldsymbol{f}^{N,e}$. The assembly of $\boldsymbol{f}^{R,e}$ is analogous.

Let therefore $e \in \EnN$ be a fixed Neumann boundary edge with local index $z \in \ME$, where $\E \in \taun$ is the unique polygon with $\e = \partial \E \cap \GammaN$. Then, for every $\ell \in \Je$, denoting by $\abold_z$ and $\bbold_z$ the endpoints of edge $e$, we have
\begin{equation} \label{computation F}
\begin{split}
\boldsymbol{f}^{N,e}_{j} &= \int_{e} \gN \overline{(\Pie\varphi_{z,j})} \, \ds
=\sum_{\eta=1}^{\pe} \overline{\beta_\eta^{e(j)}} \int_e \gN \overline{\wetae} \, \ds \\
&=\sum_{\eta=1}^{\pe} \overline{\beta_\eta^{e(j)}} \he \int_0^1 \gN(\abold_z+t(\bbold_z-\abold_z)) e^{-\im\k \djj \cdot (\abold_z+t(\bbold-\abold_z)-\xe)} \, \textup{d}t.
\end{split}
\end{equation}
The last integral can be approximated employing a Gau\ss-Lobatto quadrature formula. We remark that the computation of the right-hand side $\boldsymbol{f}$ is the only one where numerical quadrature may be required.

\subsection{General case ($\GammaD \neq \emptyset$)} \label{subsection gammaD}
The general case with $\GammaD \neq \emptyset$ can be dealt with in a similar fashion. First of all, we implement the global matrices $\boldsymbol{A}, \boldsymbol{R}$ and the right-hand side vector $\boldsymbol{f}$ as above.
Then, in order to incorporate the Dirichlet boundary conditions, we additionally impose that the numerical solution $\un$ satisfies
\begin{equation*}
\int_{e_\zeta} (\un - \gD)  \overline{\we_j} \, \ds = 0  \, \quad \forall j=1,\dots,{\pe}_\zeta, \, \forall e_\zeta \in \EnD,
\end{equation*}
which, using the expansion of $\un$ in terms of the canonical basis functions, leads to
\begin{equation*}
\sum_{\tilde{s}=1}^{\Nedg} \sum_{\tilde{\ell}=1}^{{\pe}_{\tilde{s}}}   u_{\tilde{s},\tilde{\ell}} \int_{e_\zeta} \varphi_{\tilde{s},\tilde{\ell}} \overline{w^{e_\zeta}_j} \, \ds = \int_{e_\zeta} \gD \overline{w^{e_\zeta}_j} \, \ds  \, \quad \forall j=1,\dots,{\pe}_\zeta, \, \forall e_\zeta \in \EnD.
\end{equation*}
Employing the definition of the canonical basis functions in \eqref{definition canonical basis} and the degrees of freedom in~\eqref{dofs} results in 
\begin{equation} \label{general case bc}
u_{\zeta,j} = \frac{1}{h_{e_\zeta}} \int_{e_\zeta} \gD \overline{w^{e_\zeta}_j} \, \ds  \, \quad \forall j=1,\dots,p_{e_\zeta}, \, \forall e_\zeta \in \EnD.
\end{equation}
This information is inserted in the linear system \eqref{linear system} by setting to zero all the entries in the rows of $\boldsymbol{A}$ corresponding to test functions associated with Dirichlet boundary edges, apart from the diagonal entry, which is set to one,
and replacing the corresponding values of the vector $\boldsymbol{f}$ with the right-hand sides of \eqref{general case bc}.


\section{The curse of ill-conditioning} \label{section numerical results standard}
In this section, we investigate the numerical performance of the method \eqref{complete method}.
We anticipate that the present construction of the method does not deliver accurate results due to the strong ill-conditioning related to the plane wave bases. Therefore, we will propose a numerical recipe apt to remove such instabilities, see Section \ref{subsection orthogonalized version} below.
All the tests were performed with \texttt{Matlab} \texttt{R2016b}.

We consider here boundary value problems of the form~\eqref{weak continuous problem} with $\theta=1$, $\GammaR=\partial \Omega$ on the square domain~$\Omega:=(0,1)^2$ with analytical solutions
\begin{equation} \label{exact solutions u0 u1}
\begin{split}
u_0(x,y)&:=\exp\left(\im\k x \right),\\
u_1(x,y)&:=\exp\left(\im\k\left(\cos\left(\frac{\pi}{4}\right)x+\sin\left(\frac{\pi}{4}\right)y\right)\right).
\end{split}
\end{equation}
The functions $u_0$ and $u_1$ are plane waves travelling in the directions $(1,0)$ and $(\frac{\pi}{4},\frac{\pi}{4})$, respectively, see also Figure~\ref{fig:real parts exact solutions} for contour plots of the real parts of $u_0$ and $u_1$ for $k=20$.
\begin{figure}[h]
\centering
\begin{subfigure}[b]{0.475\textwidth}
\centering
\includegraphics[width=0.9\textwidth]{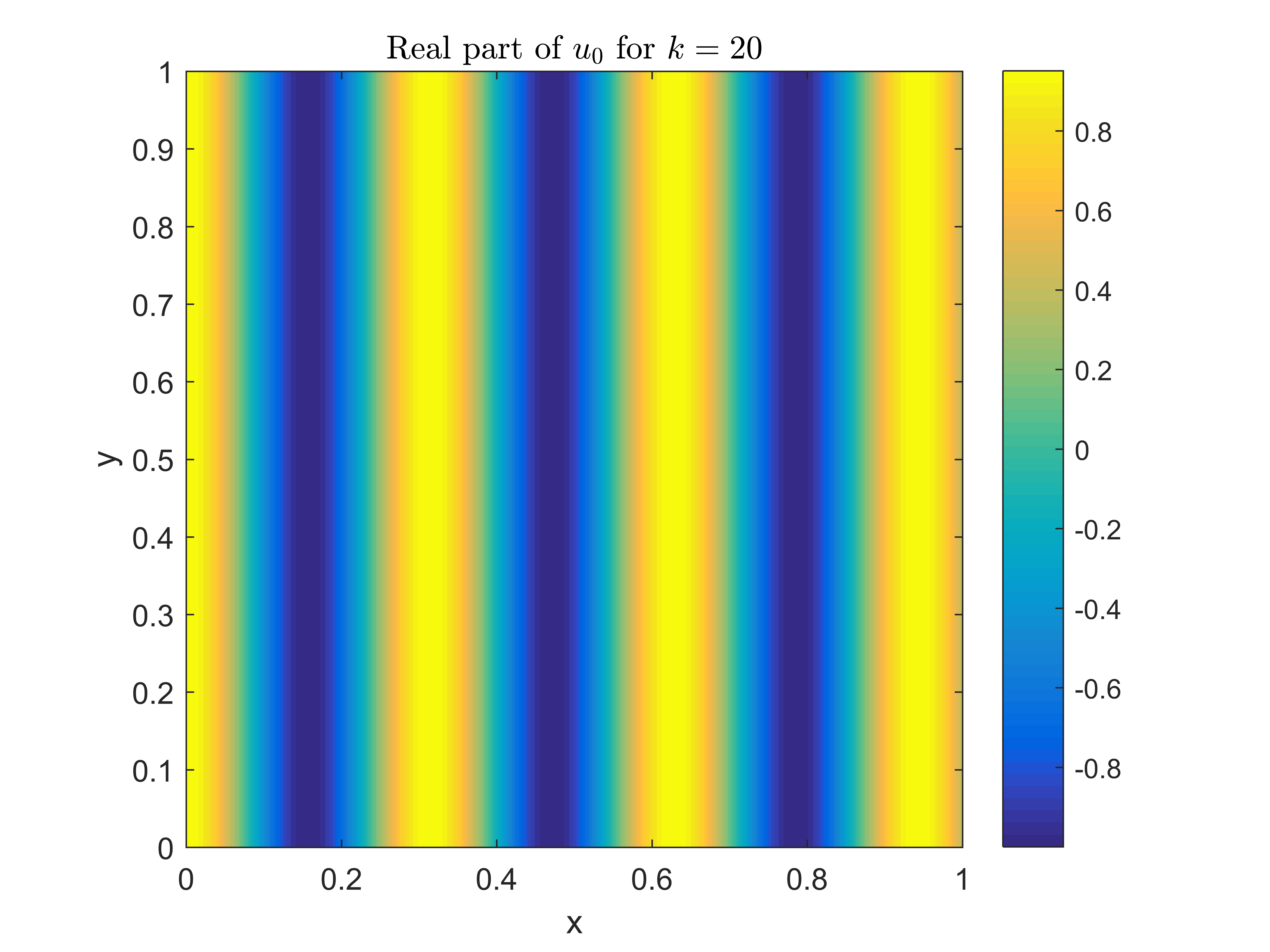}
\end{subfigure}
\hfill
\begin{subfigure}[b]{0.475\textwidth}  
\centering 
\includegraphics[width=0.9\textwidth]{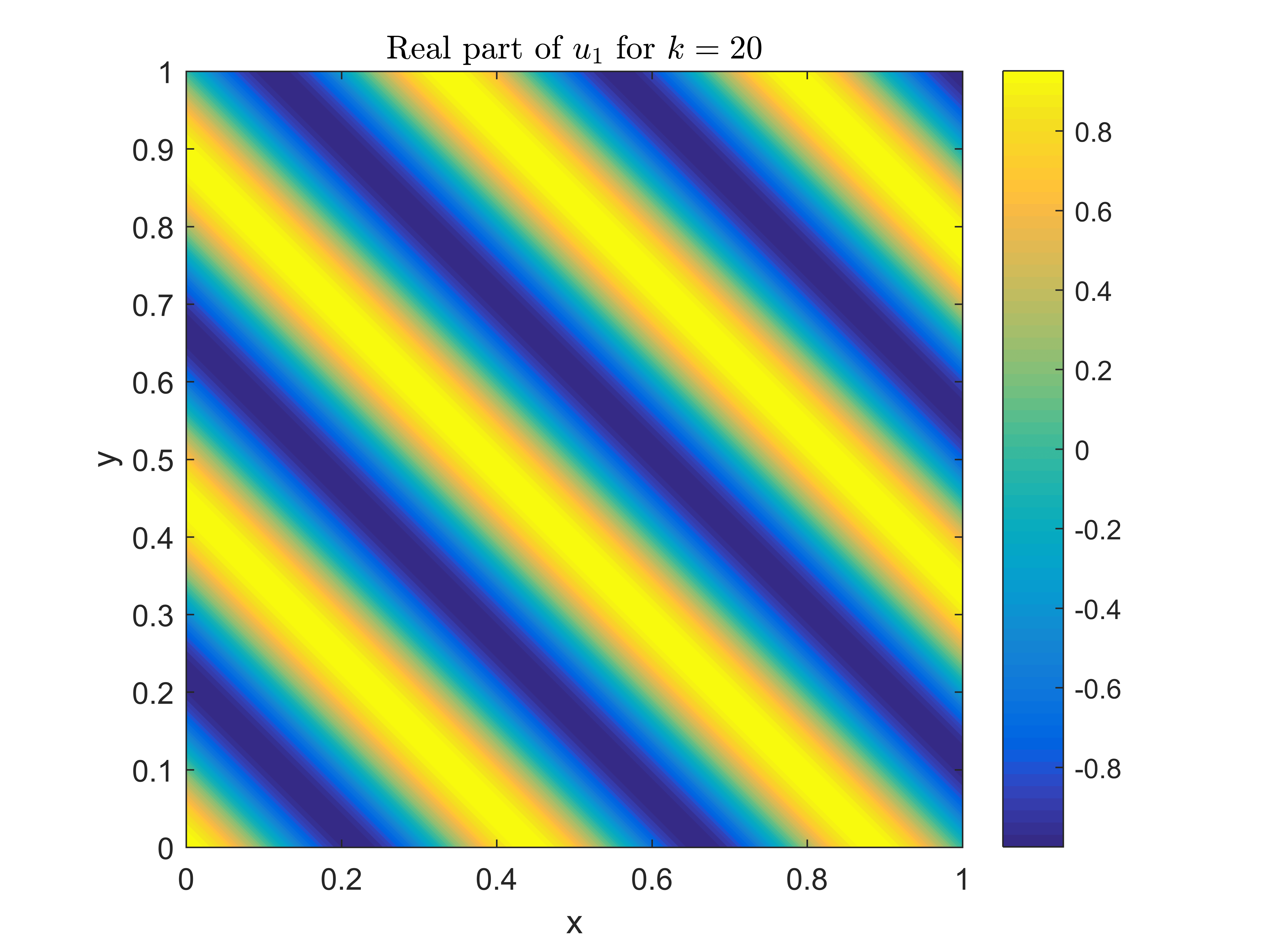}
\end{subfigure}
\caption{Real parts of the functions $u_0$ (\textit{left}) and $u_1$ (\textit{right}) defined in \eqref{exact solutions u0 u1} for $k=20$.} 
\label{fig:real parts exact solutions}
\end{figure}

Since an exact representation of the numerical solution $\un$ is not available in closed form inside each element, it is not possible to compute the exact $H^1$ and $L^2$ discretization errors directly. Instead, we compute the approximate relative errors
\begin{equation} \label{rel_errors}
\frac{\lVert u-\Pi_p \un \rVert_{1,k,\taun}}{\lVert u \rVert_{1,k,\Omega}}, \quad
\frac{\lVert u-\Pi_p \un \rVert_{0,\taun}}{\lVert u \rVert_{0,\Omega}}, 
\end{equation} 
where ${\Pi_p}_{|\E}=\Pip$, $\E \in \taun$, is the local projector defined in \eqref{projector aK}.
Mimicking what done in~\cite{ncHVEM}, it is possible to show that these relative errors converge with the same rate as the exact relative $H^1$ and $L^2$ discretization errors.

Furthermore, we employ two different local stabilizations, which in matrix form read as follows:
\begin{itemize}
\item the identity stabilization
\begin{equation} \label{identity stabilization}
\begin{split}
\boldsymbol{S}^\E=\boldsymbol{I}^\E,
\end{split}
\end{equation}
where $\boldsymbol{I}^\E \in \C^{\pE \times \pE}$ denotes the identity matrix;
\item the \emph{modified} $D$-recipe stabilization
\begin{equation}\label{modified D-recipe}
\boldsymbol{S}^\E_{(s,\ell),(r,j)}=\max\{\Re(\aE( \Pip \varphi_{r,j}, \Pip \varphi_{s,\ell})),1 \} \delta_{r,s} \delta_{\ell,j},
\end{equation}
where $\delta$ denotes the Kronecker delta.
\end{itemize}
The former choice is the original VEM stabilization proposed in \cite{VEMvolley,hitchhikersguideVEM}, whereas the latter is a modification of the {\it diagonal recipe} (D-recipe), which was introduced in \cite{VEM3Dbasic},
and whose performance was investigated for high-order VEM and in presence of badly-shaped elements in~\cite{fetishVEM, fetishVEM3D}.

In order to build a basis of $\PWE$, see \eqref{plane wave bulk space}, we employ a set $\{\dir^{(0)}_\ell\}_{\ell=1}^p$ of $p=2q+1$, $q \in \N$, equidistributed plane wave directions given by
\begin{equation} \label{pw directions}
\dir^{(0)}_\ell=
\begin{pmatrix}
\cos\left(\frac{2\pi}{p}(\ell-1)\right),\, \sin\left(\frac{2\pi}{p}(\ell-1)\right)
\end{pmatrix}.
\end{equation}

We discretize the boundary value problem on sequences of quasi-uniform Cartesian meshes and Voronoi-Lloyd meshes~\cite{paulinotestnumericipolygonalmeshes}, see Figure \ref{fig:meshes}, and investigate the $\h$-version of the method for a fixed wave number $\k=10$ and different values of $\q=2$, $3$, and $4$. Note that in the case of $u_0$, since $u_0 \in \lin\{\wlE\}_{\ell=1}^p$ and owing to the consistency property \eqref{consistency} of the discrete bilinear form, the method should reproduce, up to machine precision, the exact solution. 
The approximate relative $L^2$ bulk errors defined in~\eqref{rel_errors} are plotted in Figures \ref{fig:TEST1} and \ref{fig:TEST2}.
\begin{figure}[h]
\begin{center}
\begin{minipage}{0.46\textwidth} 
\includegraphics[width=\textwidth]{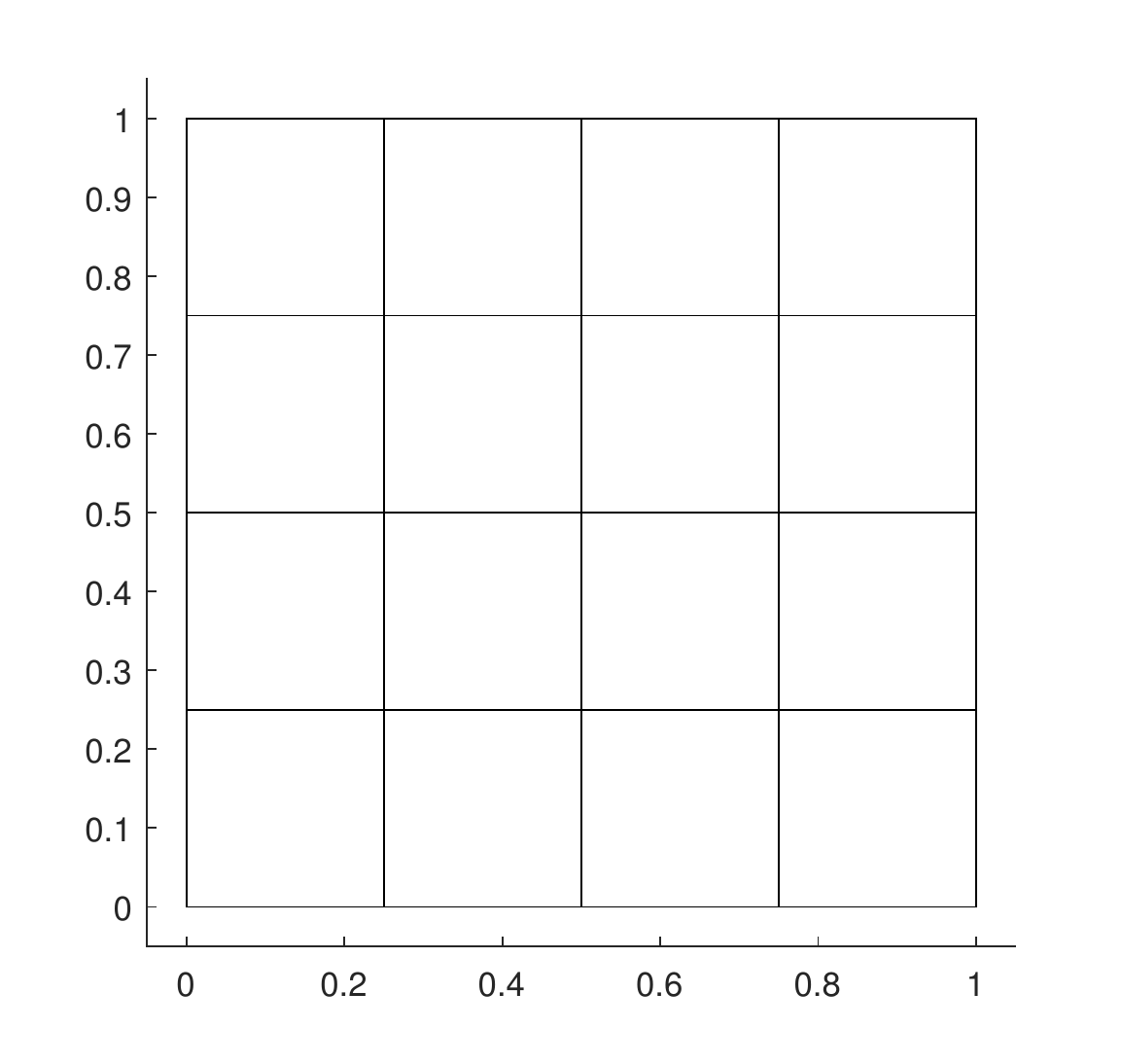}
\end{minipage}
\hspace{0.5cm}
\begin{minipage}{0.4\textwidth}
\includegraphics[width=\textwidth]{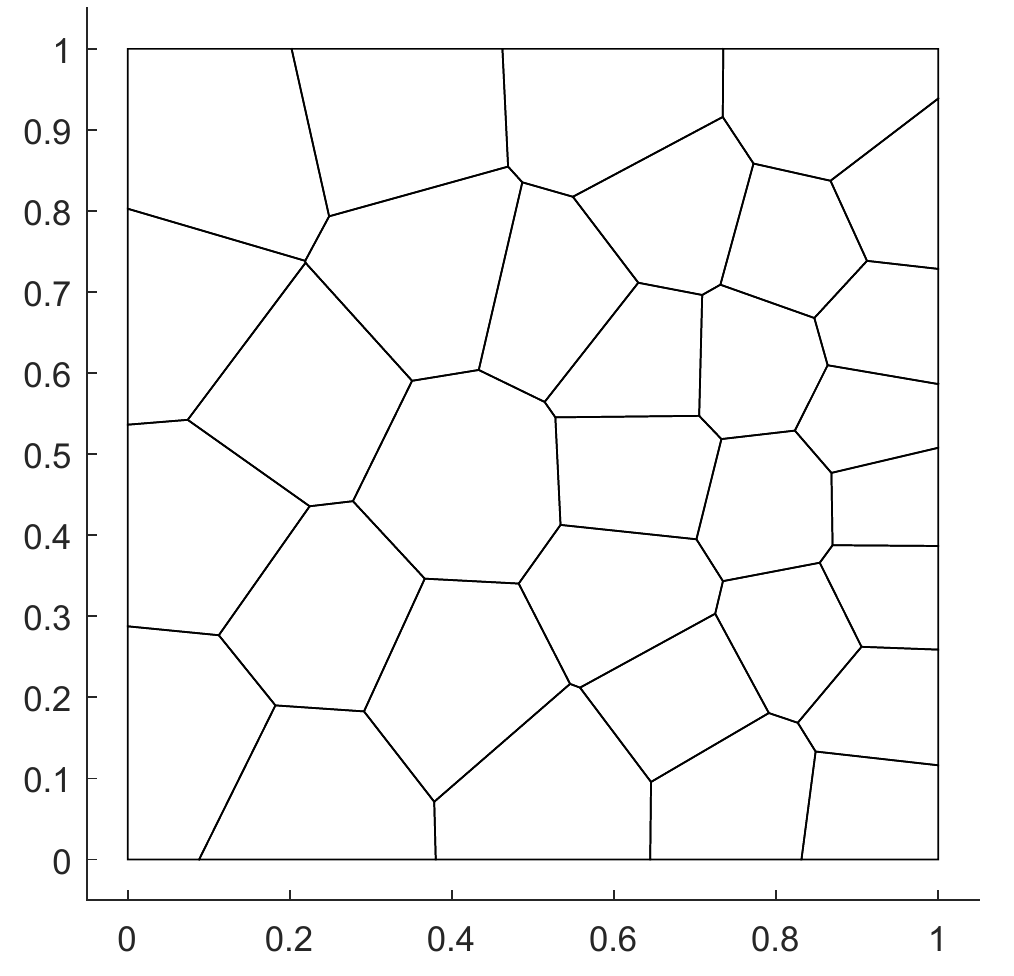}
\end{minipage}
\end{center}
\caption{\textit{Left}: Cartesian mesh. \textit{Right}: Voronoi-Lloyd mesh.}
\label{fig:meshes} 
\end{figure}

\begin{figure}[h]
\begin{center}
\begin{minipage}{0.45\textwidth} 
\includegraphics[width=\textwidth]{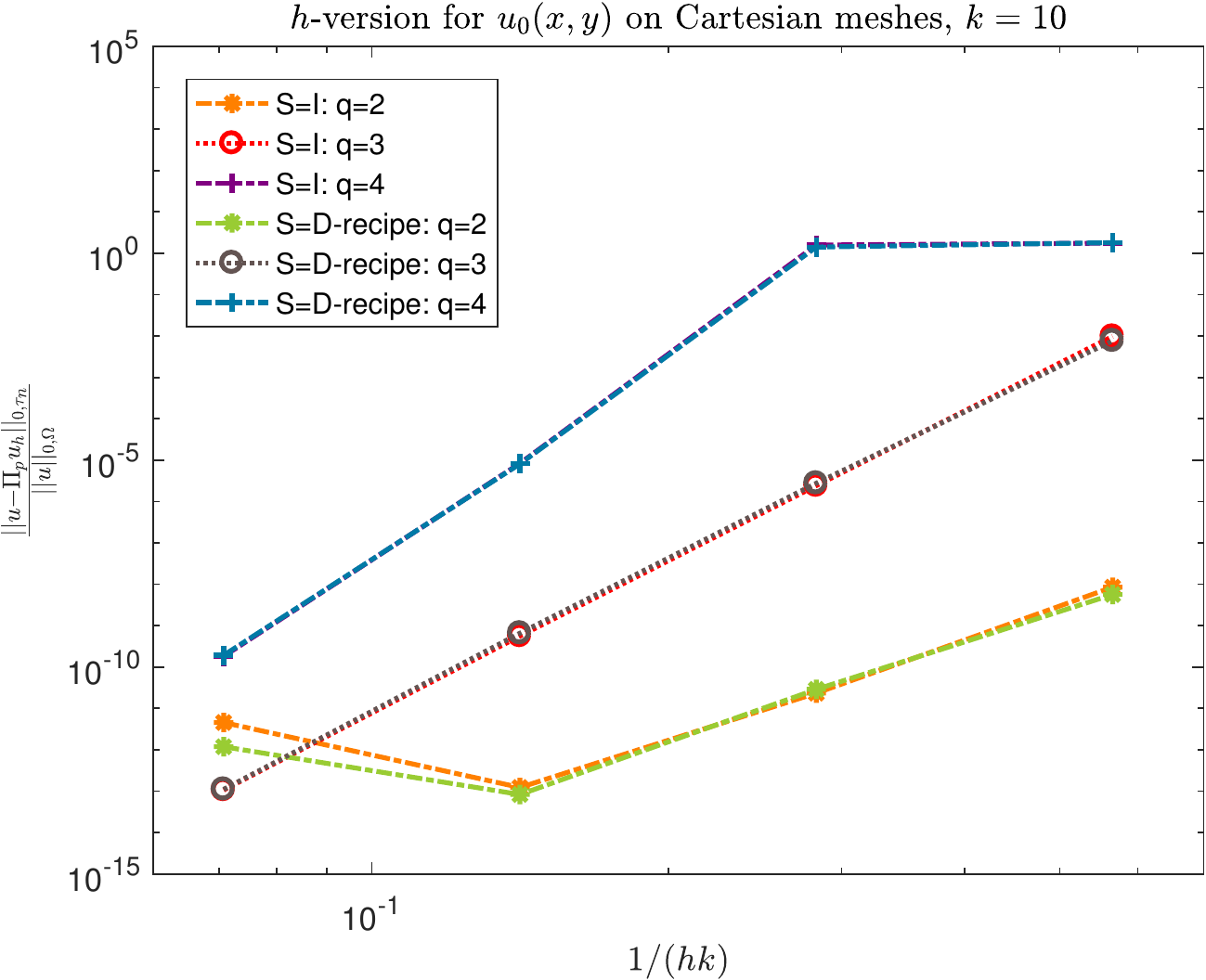}
\end{minipage}
\hspace{0.5cm}
\begin{minipage}{0.45\textwidth}
\includegraphics[width=\textwidth]{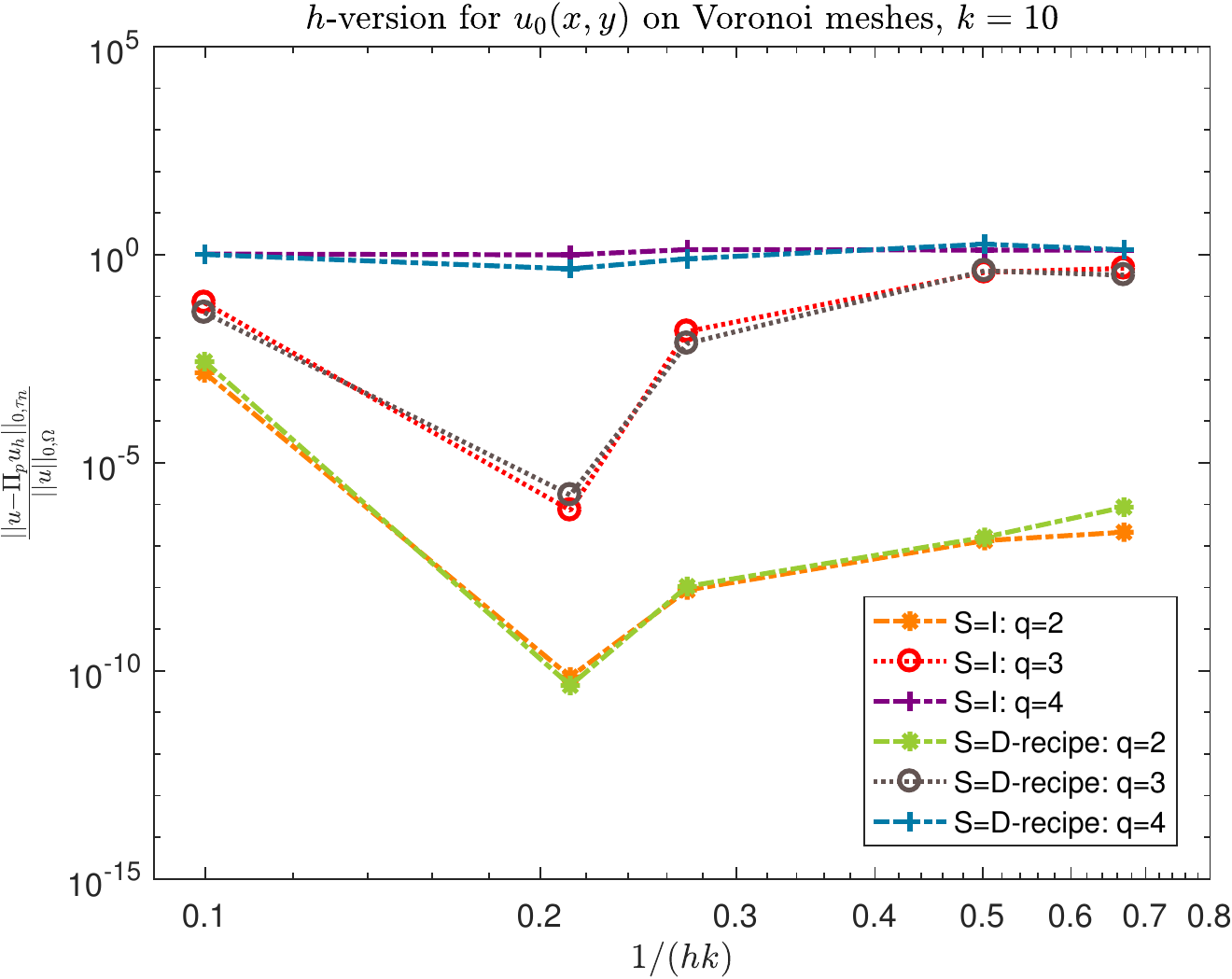}
\end{minipage}
\end{center}
\caption{Approximate relative $L^2$ bulk errors for the $h$-version of the method for $u_0$ in \eqref{exact solutions u0 u1} with $k=10$, $\q=2$, $3$, and $4$, on Cartesian meshes (\textit{left}) and Voronoi meshes (\textit{right}) with directions $\{\dir^{(0)}_\ell\}_{\ell=1}^p$ as in \eqref{pw directions},
and the identity and modified D-recipe stabilizations in \eqref{identity stabilization} and \eqref{modified D-recipe}, respectively.}
\label{fig:TEST1} 
\end{figure}

\begin{figure}[h]
\begin{center}
\begin{minipage}{0.45\textwidth} 
\includegraphics[width=\textwidth]{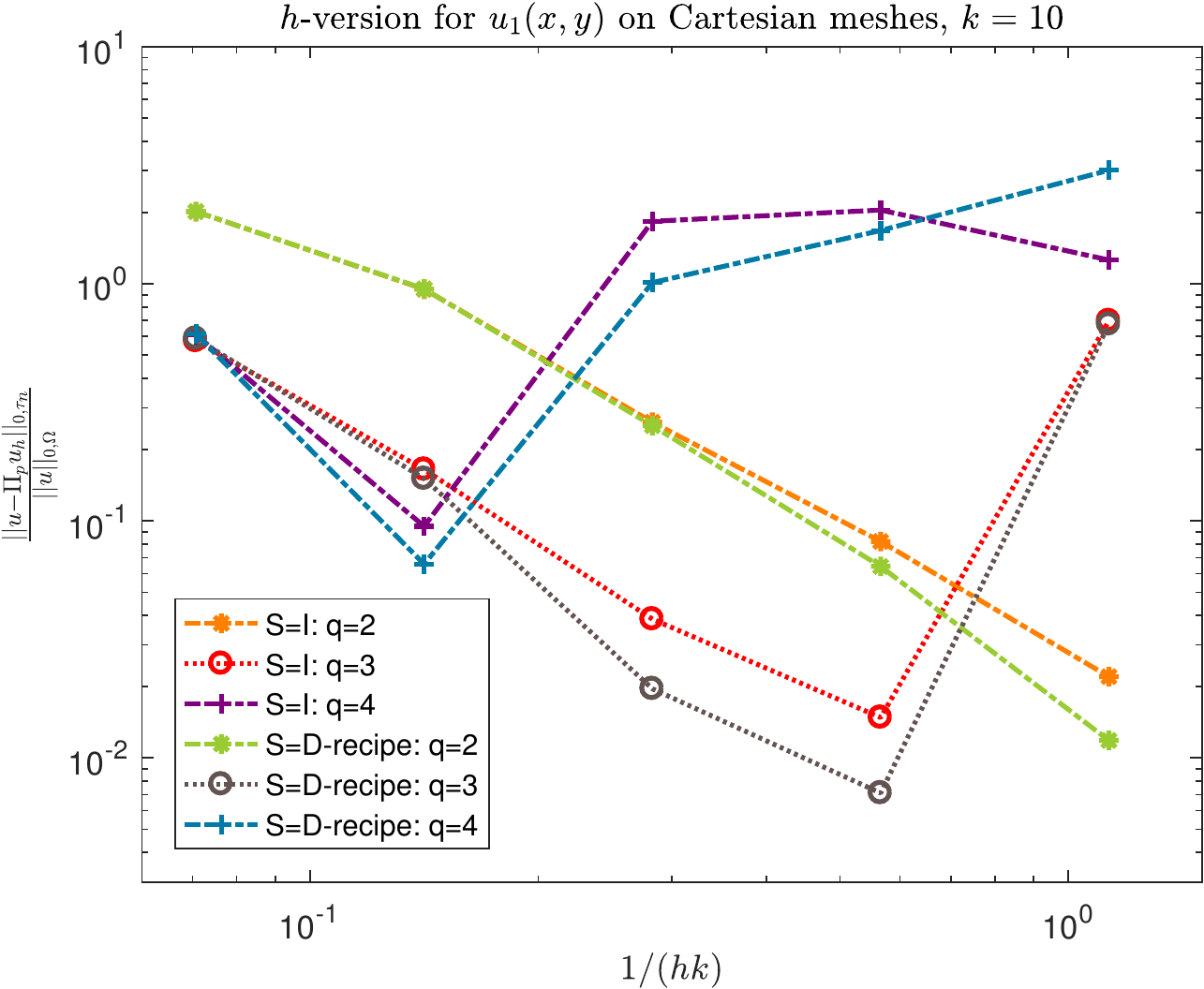}
\end{minipage}
\hspace{0.5cm}
\begin{minipage}{0.45\textwidth}
\includegraphics[width=\textwidth]{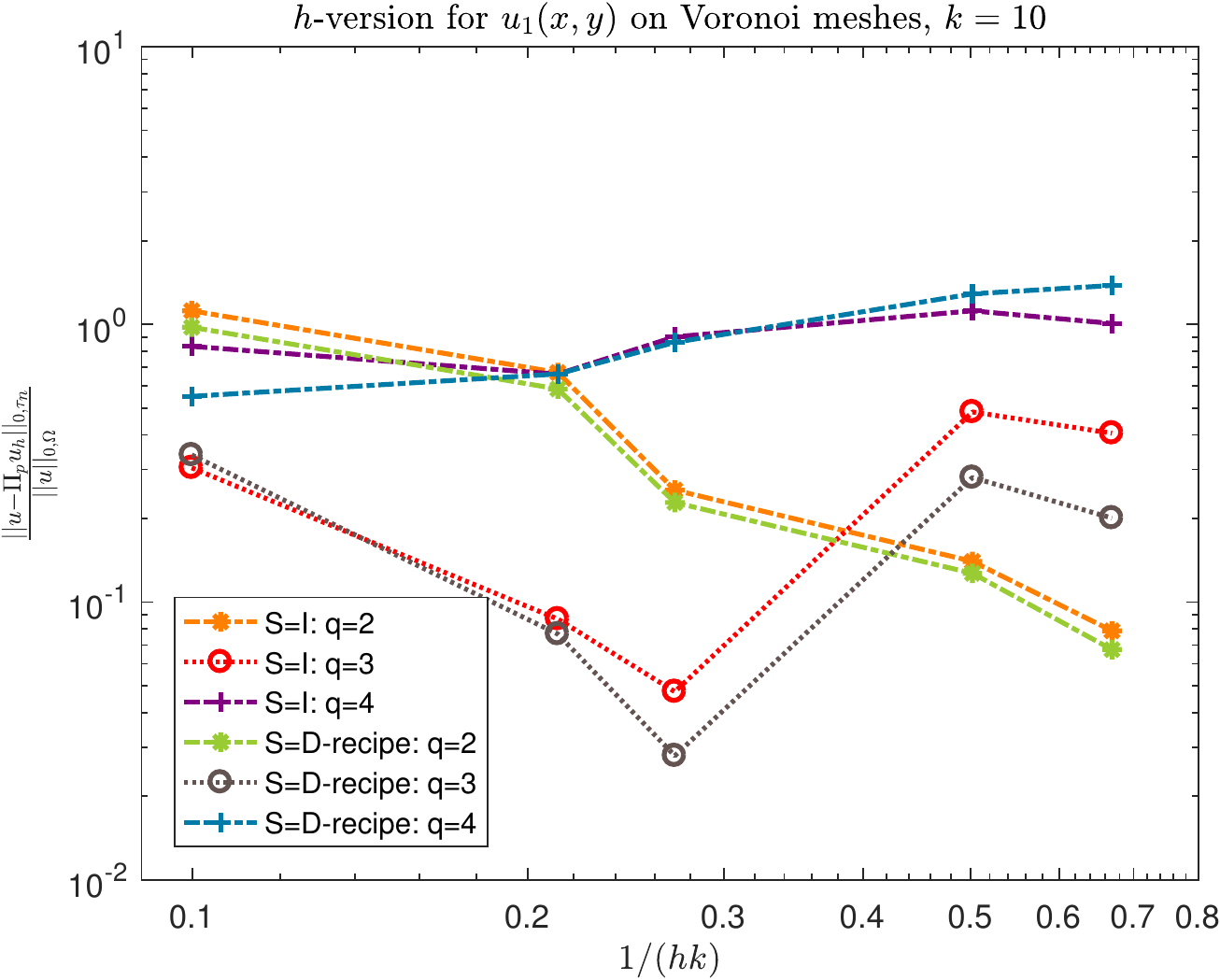}
\end{minipage}
\end{center}
\caption{Approximate relative $L^2$ bulk errors for the $h$-version of the method for $u_1$ in \eqref{exact solutions u0 u1} with $k=10$, $\q=2$, $3$, and $4$, on Cartesian meshes (\textit{left}) and Voronoi meshes (\textit{right}) with directions $\{\dir^{(0)}_\ell\}_{\ell=1}^p$ as in \eqref{pw directions},
and the identity and modified D-recipe stabilizations in \eqref{identity stabilization} and \eqref{modified D-recipe}, respectively.}
\label{fig:TEST2} 
\end{figure}

In all the cases, we notice that the method becomes unstable after very few mesh refinements.
This fact can be traced back to the computation of the Robin matrix $\boldsymbol{R}$ in \eqref{computation R} and of the right-hand side vector $\boldsymbol{f}$ in \eqref{computation F}.
Indeed, in both cases, we locally invert the edge plane wave mass matrices $\boldsymbol{G}_0^e$ in \eqref{G0e} on all boundary edges $e \in \Enb$. Such matrices are highly ill-conditioned; see Figure~\ref{fig:h cond nr G0e},
where the condition number of the matrix $\boldsymbol{G}_0^\e$ for the edge $\e$ with endpoints in $\boldsymbol{a}=[0,0]$ and $\boldsymbol{b}=[0,h]$
is depicted in dependence of~$h$ for the set of directions $\{\dir^{(0)}_\ell\}_{\ell=1}^p$ in~\eqref{pw directions} and for different values of $\q=2$, $3$, and $4$. 
In particular, one can also observe that the ill-conditioning grows together with the effective plane wave degree~$\q$.

\emph{Rebus sic stantibus}, the present version of the method is not reliable.
For this reason, we propose in Section~\ref{section cure illconditioning} a numerical recipe to mitigate this ill-conditioning.

\begin{figure}[h]
\begin{center}
\includegraphics[width=0.5\textwidth]{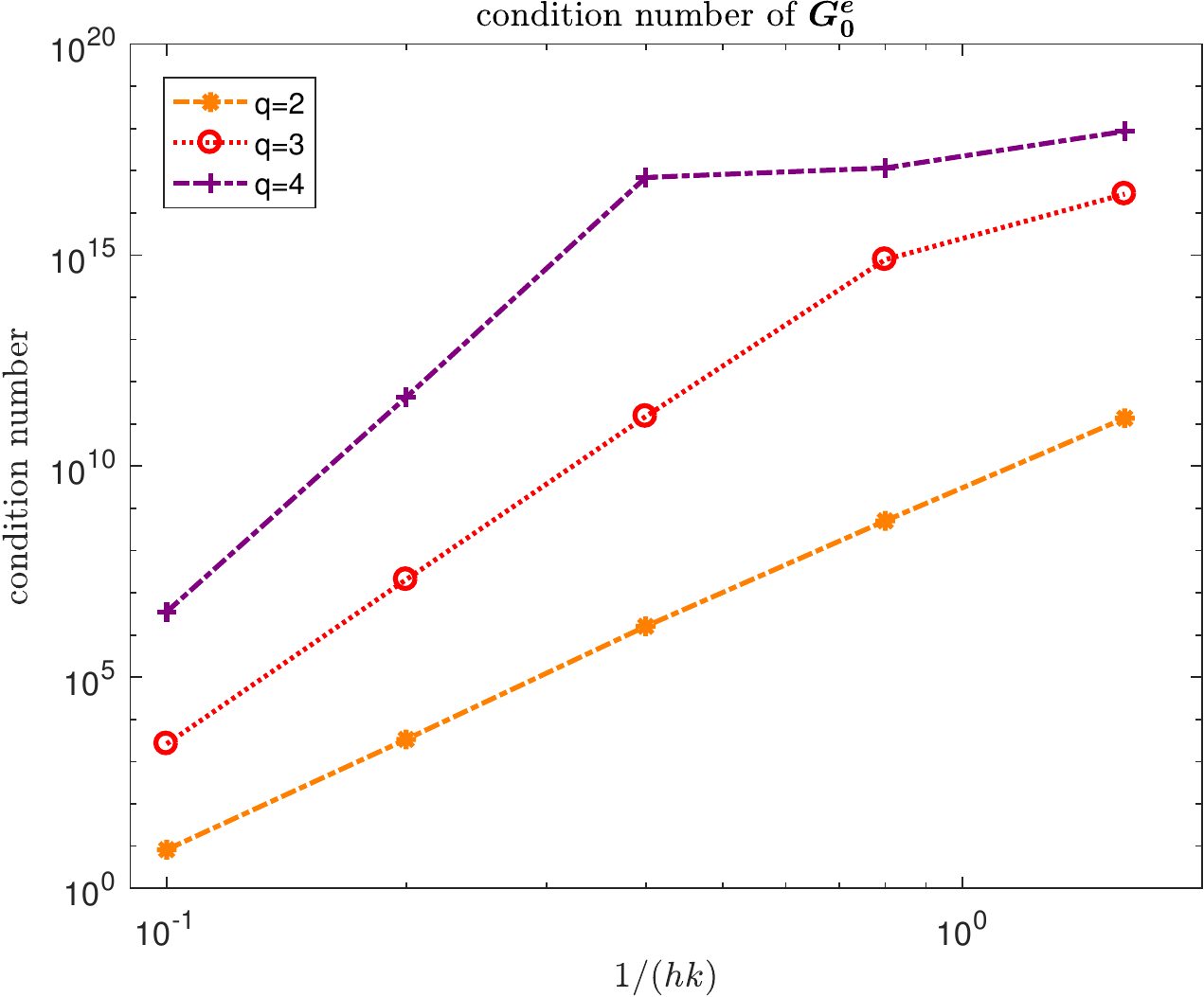}
\end{center}
\caption{Condition number of $\boldsymbol{G}_0^\e$ defined in~\eqref{G0e} for the edge~$\e$ with endpoints in $\boldsymbol{a}=[0,0]$ and $\boldsymbol{b}=[0,h]$
in terms of $\h\k$ for the set of directions $\{\dir^{(0)}_\ell\}_{\ell=1}^p$ in \eqref{pw directions} and different values of $\q=2$, $3$, and $4$.}
\label{fig:h cond nr G0e} 
\end{figure}

\section{The modified nonconforming Trefftz-VEM} \label{section cure illconditioning}
As discussed in Section~\ref{section numerical results standard}, the method~\eqref{complete method} as constructed in Section \ref{section implementational details} does not provide robust numerical performance.
The aim of this section is to describe a recipe to damp the condition number of the local Trefftz-VE matrices and make the method reliable.
In particular, in Section~\ref{subsection orthogonalized version}, we present a modification to the original method,
whose implementation aspects are described in Section \ref{subsection modified version}, and which is tested in Section~\ref{section numerical results}.
We deem that such a modification can be employed in other nonconforming settings.

\subsection{A cure for the ill-conditioning} \label{subsection orthogonalized version}
The main idea of the modification of the method is that, instead of applying the filtering process of Algorithm \ref{algorithm filtering process},
we first compute, on each edge $\e \in \En$, an eigendecomposition of the edge plane wave mass matrix $\boldsymbol{G_0^e}$:
\begin{equation} \label{eigendecomposition}
\boldsymbol {G}_0^\e \boldsymbol{Q}^\e =   \boldsymbol{Q}^\e \boldsymbol {\Lambda}^\e.
\end{equation}
Here, $\boldsymbol{G}_0^\e \in \R^{p \times p}$ is defined similarly as in \eqref{VEM edge matrices}, but using the traces of all bulk plane waves in $\PWE$ and not those after filtering as in $\PWc(\e)$, see \eqref{edge space with constant}.
Therefore, $\boldsymbol{G}_0^\e$ can be singular (e.g. when two bulk plane waves have the same trace on $e$). 
Moreover, we do not longer require that the constants belong to the plane wave trace space. Note that the requirement that the constant functions are contained in the plane wave trace spaces
was instrumental in the proof of the abstract error estimate in \cite{ncTVEM_theory}, but seems to be not necessary in practice.

In the decomposition~\eqref{eigendecomposition}, the matrices $\boldsymbol{Q}^\e \in \R^{p \times p}$ and $\boldsymbol{\Lambda}^\e \in \R^{p \times p}$ denote the eigenvector and eigenvalue matrices, respectively.
Equivalently, the $j$-th column of $\boldsymbol{Q}^\e$ contains the coefficients of the expansion of the new orthonormal plane wave $\widehat{w_j^e}$ with respect to the traces of the bulk plane waves $\w_\ell^\E$, $\ell=1,\dots,p$, on~$e$.

Next, we determine the positions of the eigenvalues on the diagonal of the matrix~$\boldsymbol{\Lambda}^\e$ which are zero or ``close'' to zero (up to a given tolerance $\sigma$), and we remove the corresponding columns of $\boldsymbol{Q}^e$. Doing so, we end up with a set of filtered orthonormalized plane waves. 
Having this, all the VE matrices discussed in Section \ref{section implementational details} are computed employing the new filtered basis. 

We highlight that this new filtering process is highly significant in presence of small edges and when employing a large number of initial plane wave basis functions. Moreover, it does not affect the rate of convergence of the method,
as we will see in the numerical experiments. Heuristically, this is not surprising since the traces of the removed plane waves ``almost'' (depending on the choice of~$\sigma$) belong to the span of the traces of the remaining ones.
A pseudo-code of this procedure is given in Algorithm \ref{algorithm orthog process}. 
\begin{algorithm}[h]
\caption{}
\label{algorithm orthog process}
Let $\sigma>0$ be a given tolerance. 
\begin{enumerate}
\item For all the edges $e \in \En$:
\begin{enumerate}
\item Assemble the real-valued, symmetric, and possibly singular matrix $\boldsymbol{G}_0^e \in \R^{\p \times \p}$ given as in \eqref{G0e} by
\begin{equation} \label{G0e new}
(\boldsymbol{G}_0^e)_{j,\ell}= (\wle,\wje)_{0,e}
\quad \forall j,\ell=1,\dots,\p.
\end{equation}
\item Starting from $\boldsymbol{G}_0^e$, compute the eigendecomposition \eqref{eigendecomposition}:
\begin{equation*}
\boldsymbol {G}_0^\e \boldsymbol{Q}^\e =   \boldsymbol{Q}^\e \boldsymbol {\Lambda}^\e,
\end{equation*}
where $\boldsymbol{Q}^e \in \R^{\p \times \p}$ is a matrix whose columns are right-eigenvectors, and $\boldsymbol{\Lambda}^\e \in\R^{\p \times \p}$ is a diagonal matrix containing the corresponding eigenvalues. 
\item Determine the eigenvalues with (absolute) value smaller than the tolerance $\sigma$ and remove the columns of $\boldsymbol{Q}^e$ corresponding to these eigenvalues.
Denote the number of remaining columns of $\boldsymbol{Q}^e$ by $\pehat \le \p$.
The remaining columns of $\boldsymbol{Q}^e$ are relabelled by $1,\dots,\pehat$.
\item Define the new $L^2(\e)$ orthonormal edge functions $\wlehat$, $\ell=1,\dots,\pehat$, in terms of the old ones $\wre$, $r=1,\dots,\p$, as
\begin{equation} \label{new basis fcts}
\wlehat:=\sum_{r=1}^{\p} \boldsymbol{Q}^e_{r,\ell} \, \wre.
\end{equation}
\end{enumerate}
\item By using \eqref{new basis fcts}, build up the new local matrices $\Ghat$, $\Bhat$, and $\Dhat$ for every element $\E \in \taun$, and assemble the global matrices $\Ahat$, $\Rhat$, and the global right-hand side vector $\Fhat$. 
\end{enumerate}
\end{algorithm}

\begin{remark} \label{remark on the second filtering}
We highlight that the influence of the choice of the parameter~$\sigma$ in Algorithm~\ref{algorithm orthog process} on the convergence of the method will be discussed in Remark~\ref{remark on sigma}.
Further, we note that, from a practical point of view, due to the presence of eigenvalues/singular values close to zero,
the computation of an orthonormal basis in \texttt{Matlab} via the eigendecomposition in step~1(b) in Algorithm~\ref{algorithm orthog process} seems to be more robust than other procedures, such as SVD.
\end{remark}
\begin{remark} \label{second remark on the second filtering}
The strategy presented in Algorithm~\ref{algorithm orthog process} seems to be natural in the nonconforming setting.
In fact, the basis functions are defined implicitly inside each elements by prescribing explicit conditions on the traces on each edge, and thus they can be modified edgewise without affecting their behavior on the other edges.
This is not the case, for instance, in DG methods, where a modification of the basis functions implies a change in the behavior of such functions over {\it all} the edges.
\end{remark}

\subsection{Details on the implementation of the modified method} \label{subsection modified version}
Here, we discuss some aspects of the implementation of the modified nonconforming Trefftz-VEM defined as in Algorithm \ref{algorithm orthog process}.

\paragraph*{Definition of the new degrees of freedom and canonical basis functions.}
Given $\E \in \taun$, let $\VEhat$ be defined similarly as $\VE$ in~\eqref{local Trefftz-VE space}, where the only difference is that the space $\PWc(e)$ in \eqref{local Trefftz-VE space} is replaced by $\PWE_{|_e}$. 
In addition, given $\e \in \En$, let $\{\wlehat \}_{\ell=1}^{\pehat}$ be the set of the new ($L^2$ orthonormal) edge functions determined with the Algorithm \ref{algorithm orthog process}. The definitions of the global nonconforming Trefftz-VE spaces in \eqref{global trial Trefftz space} and \eqref{global test Trefftz space}, and of the $L^2$ projector in \eqref{projector edge L2} are changed accordingly. 

Using~\eqref{new basis fcts}, we modify the degrees of freedom and the definition of the canonical basis functions as follows.
The new local degrees of freedom $\{\widehat{\dof}_{r,j}\}_{r=1,\dots,\Ne, \, j=1,\dots,\widehat{\p}_{e_r}}$ related to an element $\E \in \taun$ are given, for any $\vn \in \VEhat$, as 
\begin{equation} \label{dofs new}
\widehat{\dof}_{r,j}(\vn) := \frac{1}{\h_{e_r}} \int_{\e_r} \vn \overline{\widehat{w}_j^{e_r}} \, \ds \quad \forall j=1,\dots,\widehat{\p}_{e_r}.
\end{equation}
Further, the set of the new local canonical basis functions $\{\widehat{\varphi}_{s,\ell}\}_{s=1,\dots,\Ne,\, \ell=1,\dots,\widehat{\p}_{e_s}}$ associated with the local set of degrees of freedom~\eqref{dofs new} is the set of functions in the space $\VEhat$ with the property that
\begin{equation*} 
\begin{split}
\widehat{\dof}_{r,j}(\widehat{\varphi}_{s,\ell}) = \delta_{r,s} \delta_{j,\ell}, \quad \forall r,s=1,\dots,\Ne, \, \forall j=1,\dots,\widehat{\p}_{e_r}, \, \forall \ell=1,\dots,\widehat{\p}_{e_s}.
\end{split}
\end{equation*}
As usual, the sets of global degrees of freedom and of the canonical basis functions are obtained by coupling the local counterparts in a nonconforming fashion.

Next, we show how the new matrices $\Ghat$, $\Bhat$, $\Dhat$, $\Ahat$, and $\Rhat$, and the new discrete right-hand side $\Fhat$, counterparts of those described in Section \ref{section implementational details}, can be built starting from the original ones.

\paragraph*{Computation of new local matrices.}
\begin{itemize}
\item $\Ghat$: This matrix coincides with $\boldsymbol{G}^\E$ since it is computed via plane waves in the bulk. 
\item $\Bhat$: For all $j=1,\dots,p$, $s=1,\dots,\Ne$, $\ell=1,\dots,\widehat{\p}_{e_s}$, it holds
\begin{equation} \label{B orth no bdry}
\begin{split}
(\Bhat)_{j,(s,\ell)} := \aE(\widehat{\varphi}_{s,\ell},\wjE)
= -\im\k (\djj \cdot {\nE}_{|_{e_s}}) e^{-\im\k \djj \cdot (\x_{e_s}-\xE)} \int_{e_s} \widehat{\varphi}_{s,\ell} \, \overline{e^{\im\k \djj \cdot (\x-\x_{e_s})}} \, \ds.
\end{split}
\end{equation}
Expressing the old edge function $w_j^{e_s}$ in terms of the novel ones
\begin{equation} \label{expansion}
w_j^{e_s}=\sum_{\zeta=1}^{\widehat{\p}_{e_s}} (\boldsymbol{Q}^e)^T_{\zeta,j} \, \widehat{w}_{\zeta}^{e_s},
\end{equation}
and	plugging this into \eqref{B orth no bdry} lead to
\begin{equation*} 
(\Bhat)_{j,(s,\ell)} = -\im\k \overline{(\boldsymbol{Q}^e)^T_{\ell,j}} (\djj \cdot {\nE}_{|_{e_s}}) e^{-\im\k \djj \cdot (\x_{e_s}-\xE)} h_{e_s}.
\end{equation*}
\item $\Dhat$:
Given $r \in \ME$, $j=1,\dots,\widehat{\p}_{e_r}$, $\ell=1,\dots,p$, a direct computation based again on the expansion \eqref{expansion} gives
\begin{equation*} 
(\Dhat)_{(r,j),\ell} := \widehat{\dof}_{r,j}(\wlE)=\frac{1}{h_{e_r}} \int_{e_r} \wlE \overline{\widehat{w}_j^{e_r}} \, \ds
= \sum_{\zeta=1}^{\p} \overline{\boldsymbol{Q}^e_{\zeta,j}} \frac{1}{h_{e_r}} \int_{e_r} \wlE \overline{w_{\zeta}^{e_r}} \, \ds.
\end{equation*}
\item $\Ahat$: Starting from the local matrices 
\begin{equation*}
\Ahat^\E = \overline{\Bhat}^T \overline{\Ghat}^{-T} \Bhat + \overline{(\Ihat-\Pihat)}^T \Shat (\Ihat-\Pihat),
\end{equation*}
see~\eqref{matrix A}, $\Ahat$ is assembled as in~\eqref{matrix AK}, where~$\Pihat$ is defined similarly as in~\eqref{definition Pi 2}.

\item $\Rhat$: We need to compute  
\begin{equation*}
\begin{split}
\widehat{\boldsymbol{R}}_{(\tilde{r},\tilde{j}),(\tilde{s},\tilde{\ell})}=\im\k\theta \sum_{\e \in \EnR} \int_{\e} &(\Pie \widehat{\varphi}_{\tilde{s},\tilde{\ell}}) \overline{(\Pie \widehat{\varphi}_{\tilde{r},\tilde{j}})} \, \ds \\
&\forall \tilde{r},\tilde{s}=1,\dots,N_e, \, \forall \tilde{j}=1,\dots,\widehat{\p}_{e_{\tilde{r}}}, \, \forall \tilde{\ell}=1,\dots,\widehat{\p}_{e_{\tilde{s}}}.
\end{split}
\end{equation*}
Given $e \in \EnR$, we only describe here the assembly of the matrix $\Rhate \in \C^{\pe \times \pe}$, which takes into account the local contributions of the basis functions associated with $e$. Then, $\Rhat$ is assembled as in \eqref{matrix Re}. Given $z$ the local index of $e$, for every $j,\ell=1,\dots,\pehat$, it holds
\begin{equation} \label{Robin orthog}
(\Rhate)_{\ell,j} = \int_e (\Pie \widehat{\varphi}_{z,j}) \overline{(\Pie \widehat{\varphi}_{z,\ell})} \, \ds.
\end{equation}
By writing each $\Pie \widehat{\varphi}_{z,j}$, $j=1,\dots,\pehat$, as a linear combination of the $L^2(e)$ orthonormal plane waves $\widehat{w}_\theta^{e}$, $\theta=1,\dots,\pehat$, and inserting this into \eqref{Robin orthog}, one obtains
\begin{equation*}
\Rhate=\overline{(\widehat{\boldsymbol{B}}_0^e)}^T (\overline{\widehat{\boldsymbol{G}}_0^e)}^{-T} \widehat{\boldsymbol{B}}_0^e,
\end{equation*}
where
\begin{equation*}
(\widehat{\boldsymbol{B}}_0^e)_{j,\ell} = (\widehat{\varphi}_{z,\ell},\widehat{w}_j^{e})_{0,e} = \he \delta_{j,\ell} \quad \forall j, \ell=1,\dots,\pehat,
\end{equation*}
and
\begin{equation*}
(\widehat{\boldsymbol{G}}_0^e)_{j,\ell} = (\widehat{\w}_\ell^e,\widehat{\w}_j^e)_{0,e} = \sum_{\zeta,\eta=1}^{\pehat} \boldsymbol{Q}^e_{\eta,\ell} \overline{\boldsymbol{Q}^e_{\zeta,j}} \int_e w_{\eta}^{e} \overline{w_{\zeta}^{e}} \, \ds,
\end{equation*}
which can be represented as
\begin{equation*}
\widehat{\boldsymbol{G}_0^e} = \overline{(\boldsymbol{Q}^e)}^T \boldsymbol{G}_0^e \, \boldsymbol{Q}^e
\end{equation*}
with $\boldsymbol{G}_0^e$ given in \eqref{G0e new}.

\item $\Fhat:=\widehat{\boldsymbol{f}}^N+\widehat{\boldsymbol{f}}^R$: We restrict here ourselves to the computation of $\widehat{\boldsymbol{f}}^N$, which is given by
\begin{equation*}
(\widehat{\boldsymbol{f}}^N)_{(\tilde{r},\tilde{j})}=
\sum_{e \in \EnN} \int_{e} \gN \overline{(\Pie \widehat{\varphi}_{\tilde{r},\tilde{j}})} \quad \forall \tilde{r}=1,\dots,N_e, \, \forall \tilde{j}=1,\dots,\widehat{\p}_{e_{\tilde{r}}}.
\end{equation*}
The local vector $\widehat{\boldsymbol{f}}^{N,e} \in \C^{\pe}$ for a given $e \in \EnN$, with $z$ denoting again the local index associated with $e$, has the form
\begin{equation*}
\begin{split}
(\widehat{\boldsymbol{f}}^{N,e})_{\ell} 
=\int_e \gN \overline{(\Pie\widehat{\varphi}_{z,\ell})} \, \ds
=\sum_{\eta=1}^{\pehat} \widehat{\beta}_{\eta}^{(\ell)} \int_e \gN \overline{\widehat{w}_{\eta}^{e}} \, \ds = \sum_{\eta=1}^{\pehat} \sum_{\zeta=1}^{\p} \widehat{\beta}_{\eta}^{(\ell)} (\boldsymbol{Q}^e)_{\zeta,\eta} \int_\e \gN \overline{w_{\zeta}^{e}} \, \ds.
\end{split}
\end{equation*}
\item The Dirichlet boundary conditions are incorporated in the global system of linear equations as already shown in Section \ref{subsection gammaD}, by requiring that
\begin{equation*}
\int_{e_\zeta} (\un - \gD)  \overline{\widehat{w}^{e_\zeta}_j} \, \ds = 0  \, \quad \forall j=1,\dots,\widehat{\p}_{e_\zeta}, \, \forall e_\zeta \in \EnD,
\end{equation*}
which leads to
\begin{equation*} 
u_{\zeta,j} = \frac{1}{h_{e_\zeta}} \int_{e_\zeta} \gD \overline{\widehat{w}^{e_\zeta}_j} \, \ds
= \frac{1}{h_{e_\zeta}} \sum_{r=1}^{\p} \overline{(\boldsymbol{Q}^{e_\zeta}_{r,j})} \int_{e_\zeta} \gD \overline{w_r^{e_\zeta}} \, \ds
\quad \forall j=1,\dots,\widehat{p}_{e_{\zeta}}, \, \forall e_\zeta \in \EnD.
\end{equation*}
\end{itemize}

\subsection{Numerical results with the modified method} \label{section numerical results}
In this section, we discuss the $h$-, $p$-, and $hp$-versions of the modified method and assess the improvements in the numerical performance.
We will see that the modified method is not only better conditioned, but also the number of degrees of freedom needed to achieve a given accuracy of the numerical approximation is significantly lower than in the original version in Section \ref{section implementational details}.
Moreover, we compare the modified nonconforming Trefftz-VEM with the PWVEM of~\cite{Helmholtz-VEM} and with the more established PWDG method~\cite{ncTVEM_theory}.

In all the numerical tests throughout this paper, the tolerance~$\sigma$ in Algorithm~\ref{algorithm orthog process} is set to $10^{-13}$. Other choices and their influence on the method are discussed in Remark~\ref{remark on sigma}.

Additionally to the boundary value problems~\eqref{weak continuous problem} on $\Omega:=(0,1)^2$ with known solutions $u_0$ and $u_1$ in \eqref{exact solutions u0 u1}, we consider boundary value problems for $\theta=1$ and $\GammaR = \partial \Omega$ with exact solutions 
\begin{equation} \label{exact solutions u2 u3}
\begin{split}
u_2(x,y)&:=H_0^{(1)}(k|\x-\x_0|), \quad \x_0=(-0.25,0),
\\
u_3(x,y)&:=J_{\xi}(\k r) \cos\left(\xi \theta\right), \quad \xi=\frac{2}{3},
\end{split}
\end{equation}
where $H_0^{(1)}$ is the zeroth-order Hankel functions of the first kind, $J_\xi$ denotes the Bessel function of the first kind, and $r$ and $\theta$ are the polar coordinates of $(x,y-0.5)$, see~\cite[Chapters 9 and 10]{AbramowitzStegun_handbook}.
Note that the function $u_2$ is analytic over $\Omega$, but $u_3$ has a singularity at $(0,0.5)$; more precisely, $u_3 \in H^{\xi+1-\epsilon}(\Omega)$ for all $\epsilon>0$ arbitrarily small, but $u_3 \notin H^{\xi+1}(\Omega)$.
The contour plots of the real parts for the two test cases in~\eqref{exact solutions u2 u3} with $\k=20$ are plotted in Figure~\ref{fig:real parts exact solutions 2}. 

\begin{figure}[h]
\centering
\begin{subfigure}[b]{0.475\textwidth}
\centering
\includegraphics[width=\textwidth]{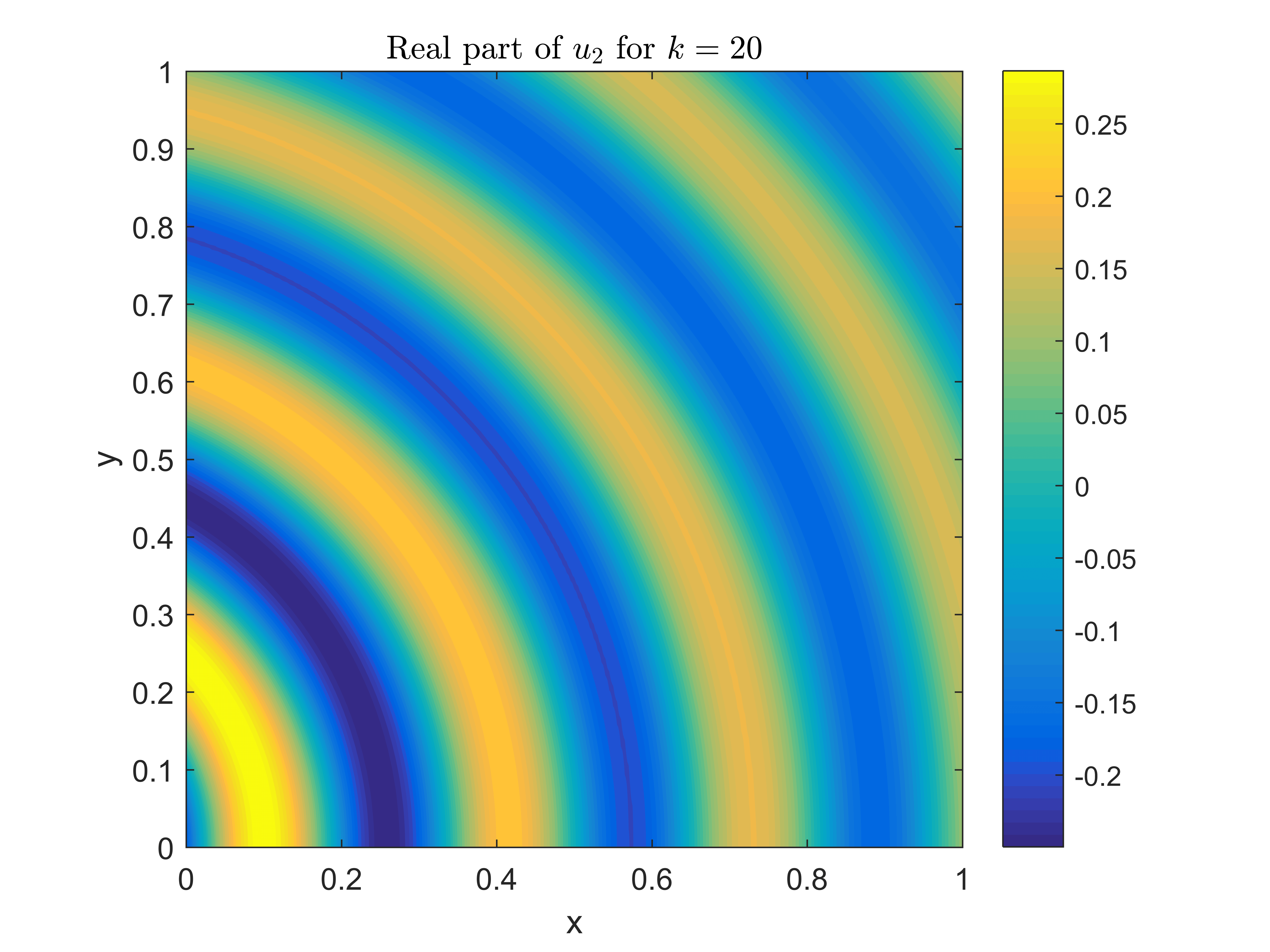}
\end{subfigure}
\hfill
\begin{subfigure}[b]{0.475\textwidth}  
\centering 
\includegraphics[width=\textwidth]{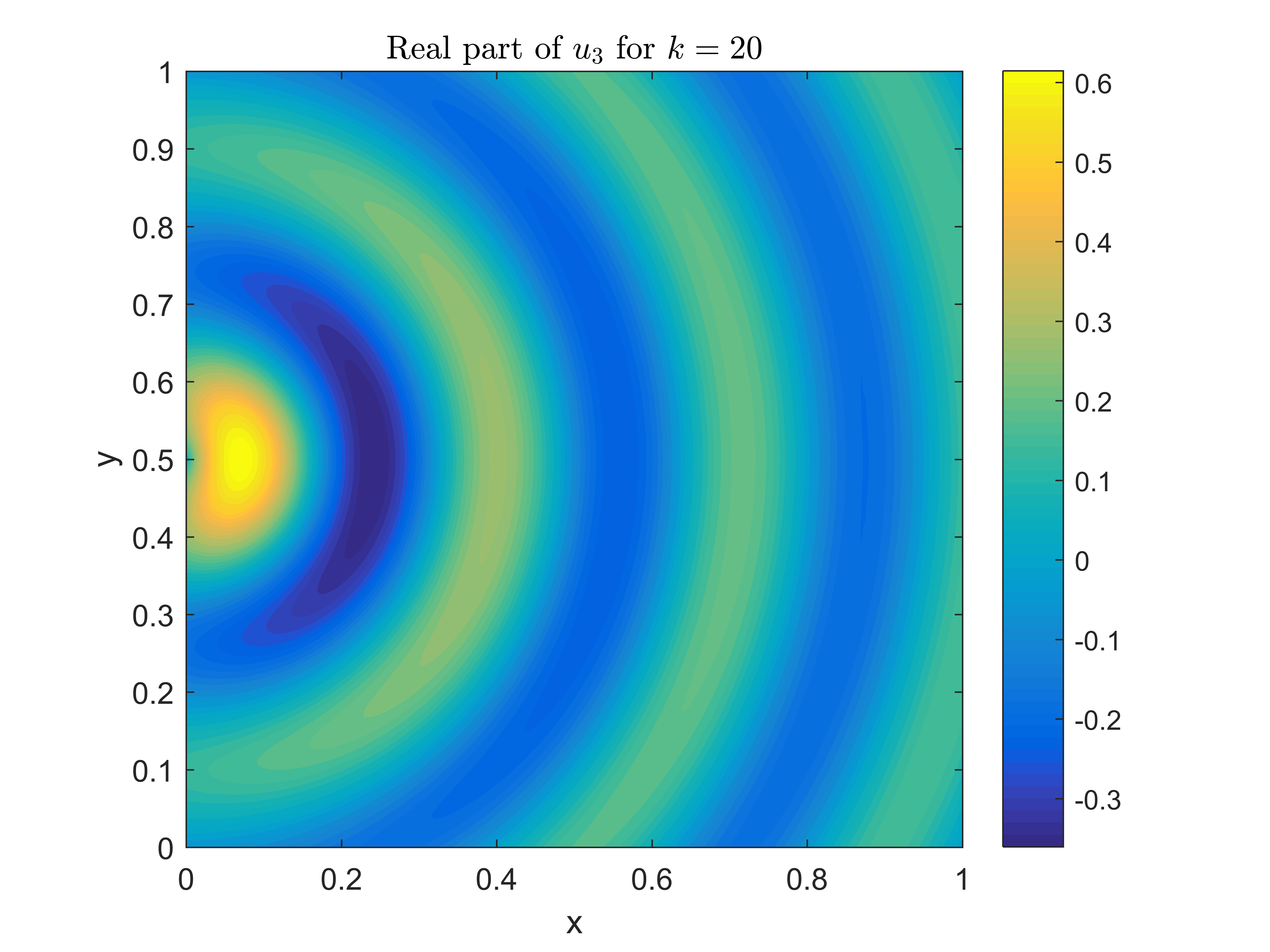}
\end{subfigure}
\caption{Real parts of the functions $u_2$ (\textit{left}) and $u_3$ (\textit{right}) defined in \eqref{exact solutions u2 u3} for $k=20$.} 
\label{fig:real parts exact solutions 2}
\end{figure}

We also consider the test of a scattering problem in Section \ref{subsection scattering} (here, $\GammaR \neq \partial \Omega$).

\subsubsection{$h$-version} \label{section hversion num recipe}
We first show the modified method on the \textit{patch test} $u_0$ defined in \eqref{exact solutions u0 u1} to check the consistency~\eqref{consistency} and to validate the gain in robustness with respect to the original version, cf. Section~\ref{section numerical results standard}.
Let $\{\dir^{(0)}_\ell\}_{\ell=1}^p$ be the set of directions given in \eqref{pw directions}.
The numerical experiments are again performed on sequences of quasi-uniform Cartesian meshes and Voronoi-Lloyd meshes, see Figure \ref{fig:meshes}, for $\k=10$ and $20$, and effective plane wave degree $\q=4$ and~$7$. Recall that the number of used bulk plane waves is $p=2q+1$.
Further, we employ the \textit{modified D-recipe stabilization} in~\eqref{modified D-recipe}.
In Figure~\ref{fig:TEST3_H1}, the approximate relative $H^1$ bulk errors in~\eqref{rel_errors} are plotted.  
\begin{figure}[h]
\begin{center}
\begin{minipage}{0.45\textwidth} 
\includegraphics[width=\textwidth]{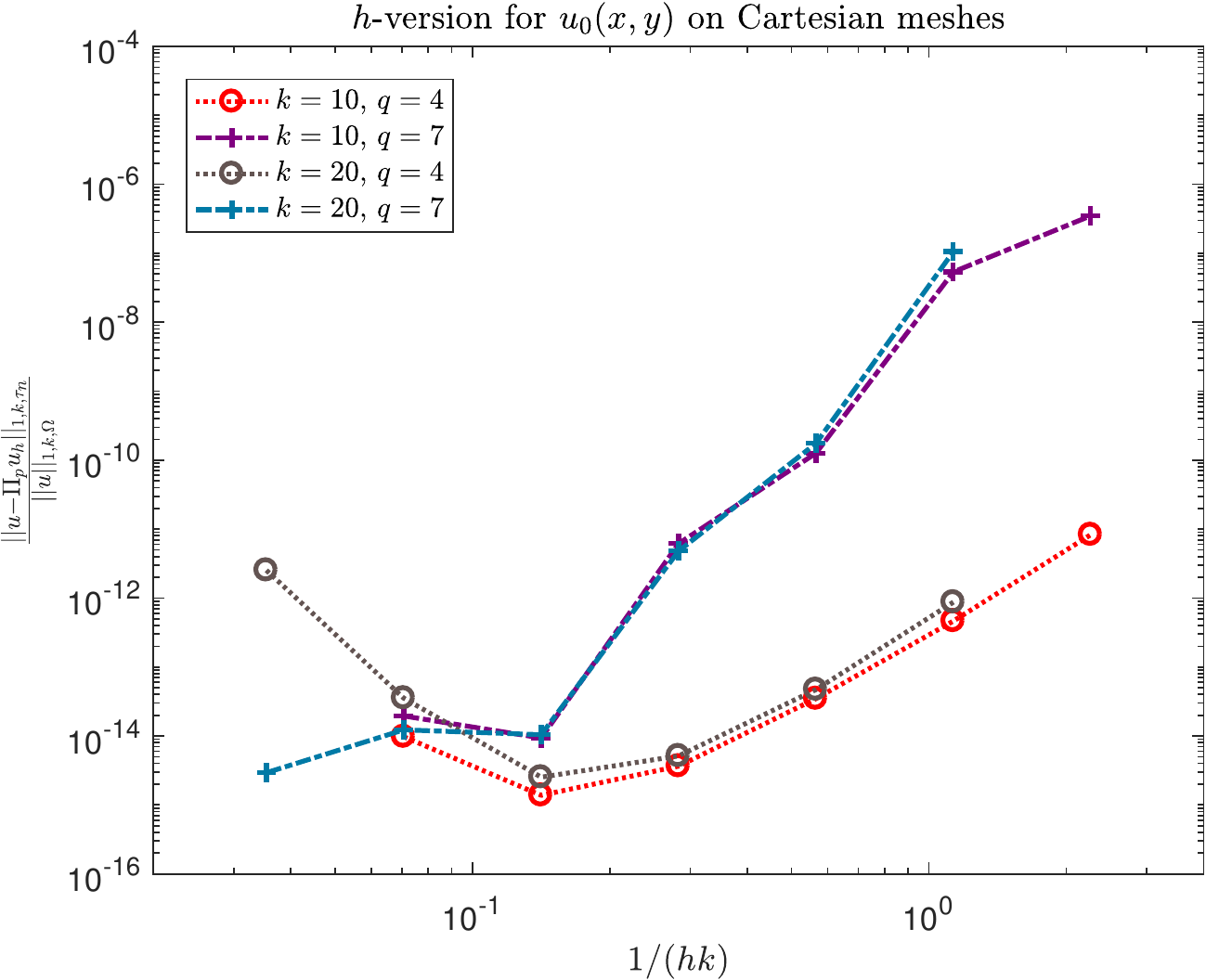}
\end{minipage}
\hspace{0.5cm}
\begin{minipage}{0.45\textwidth}
\includegraphics[width=\textwidth]{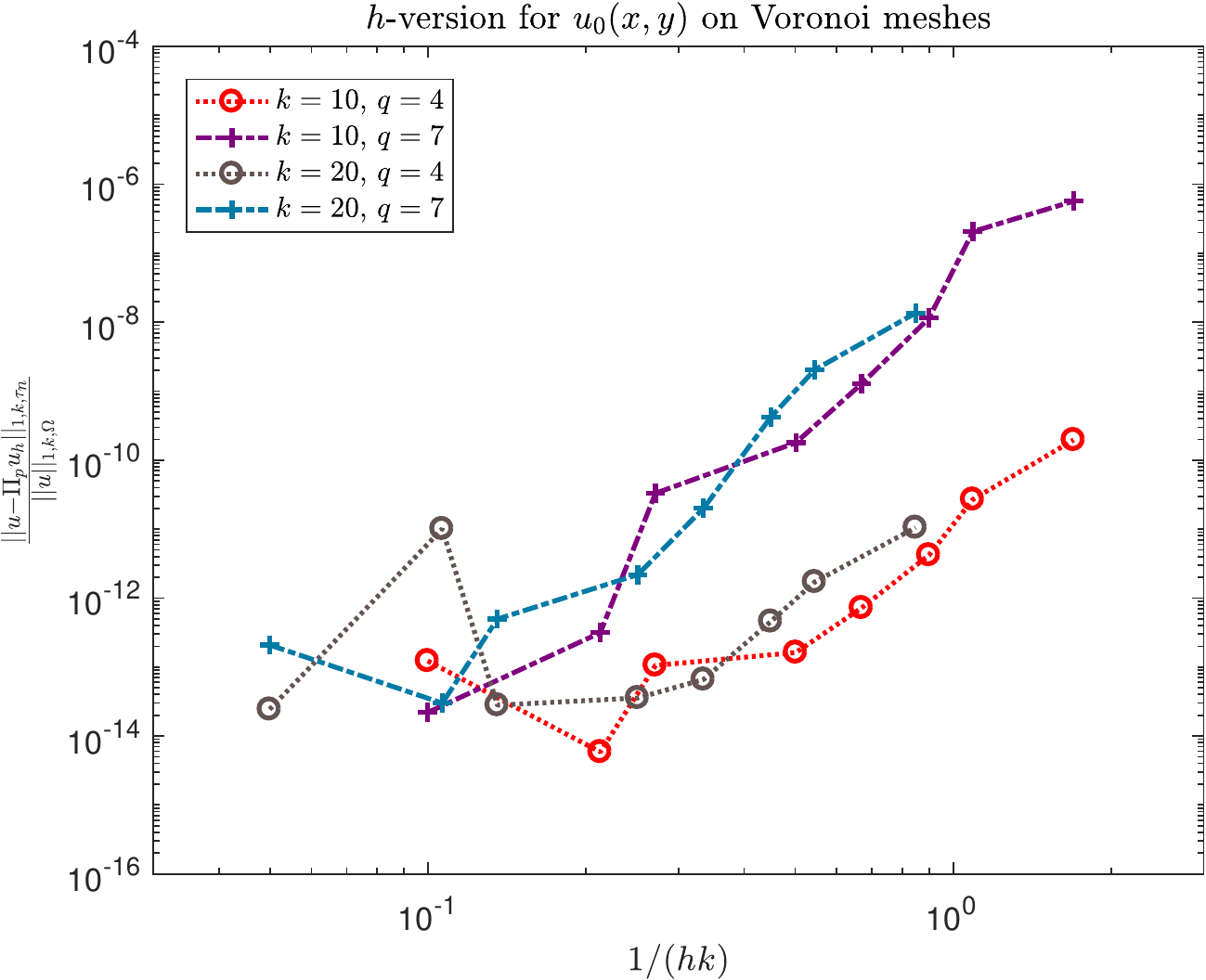}
\end{minipage}
\end{center}
\caption{$h$-version of the method for $u_0$ in \eqref{exact solutions u0 u1} with $\k=10$ and $20$, and $q=4$ and $7$,
with the sets of directions $\{\dir^{(0)}_\ell\}_{\ell=1}^p$ as in \eqref{pw directions} and the modified D-recipe stabilization \eqref{modified D-recipe} on Cartesian meshes (\textit{left}) and Voronoi meshes (\textit{right}).}
\label{fig:TEST3_H1} 
\end{figure}
We observe that the patch test is fulfilled for meshes with a moderately small mesh size. The plots indicate that the modified version is much more stable than the original one, see Figure~\ref{fig:TEST1}.
Nevertheless, also this modified version is affected by ill-conditioning, which results in the increase of the errors for decreasing mesh size $h$, as typical of plane wave-based methods.

As a second test, we investigate the $\h$-version for the exact solution $u_1$ in \eqref{exact solutions u0 u1} with $\k=10$, $20$, and~$40$, and $\q=4$ and~$7$, employing the same choice of directions, meshes, and stabilizations as before.
The numerical results are depicted for the Cartesian meshes in Figure~\ref{fig:TEST4_D_cart} and Table~\ref{tab:TEST4_D_cart} ($k=20$, $q=7$), and for the Voronoi meshes in Figure~\ref{fig:TEST4_D_voro} and Table~\ref{tab:TEST4_voro_D} ($k=20$, $q=7$).
In all cases the errors were computed accordingly with~\eqref{rel_errors}. In Table~\ref{tab:TEST4_D_cart} and~\ref{tab:TEST4_voro_D} we further compare the number of degrees of freedom
using the modified version of the method with the original one. The reduction of degrees of freedom in \% is presented in the last column.

Here, we mention that the tests with exact solution~$u_2$ give similar results to those for the smooth solution $u_1$ and are postponed to Sections \ref{section comp Trefftz PWVEM} and \ref{subsection PWDG},
where the modified nonconforming Trefftz-VEM will be compared with the PWVEM \cite{Helmholtz-VEM} and the PWDG \cite{GHP_PWDGFEM_hversion}, respectively.

\begin{figure}[h]
\begin{center}
\begin{minipage}{0.48\textwidth} 
\includegraphics[width=\textwidth]{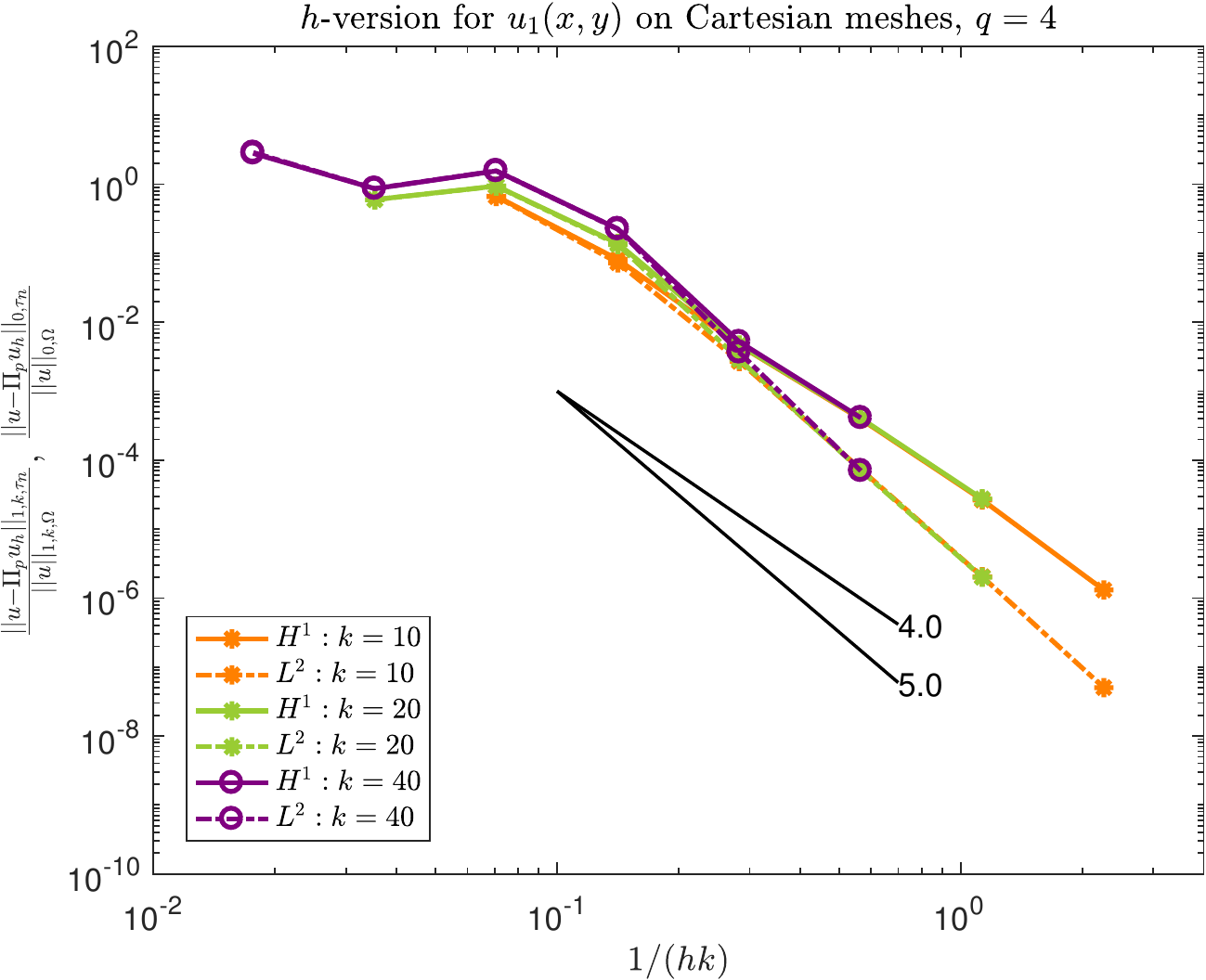}
\end{minipage}
\hfill
\begin{minipage}{0.48\textwidth}
\includegraphics[width=\textwidth]{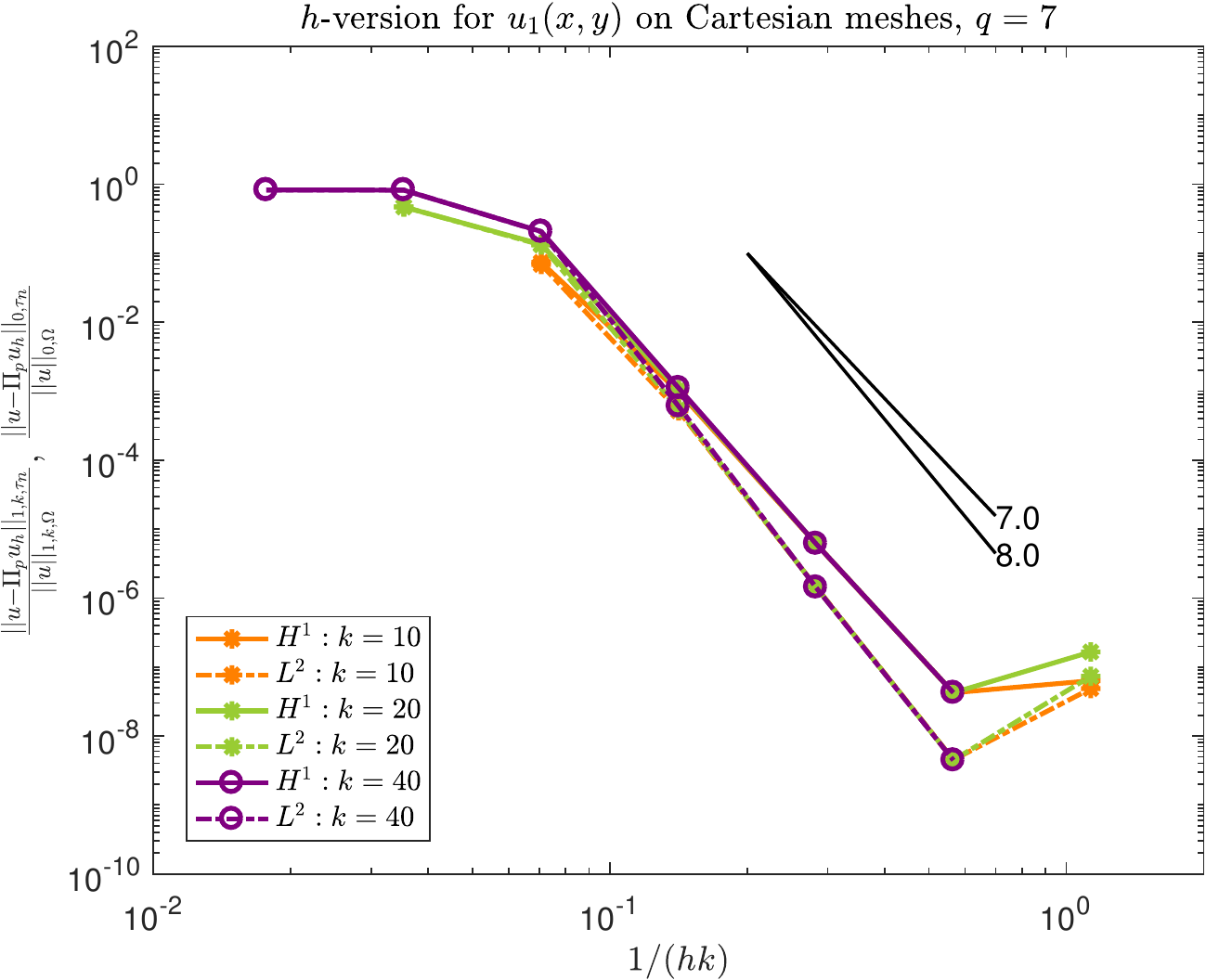}
\end{minipage}
\end{center}
\caption{$h$-version of the modified method for $u_1$ in \eqref{exact solutions u0 u1} with $\k=10$, $20$, and~$40$, and $q=4$ (\textit{left}) and $7$ (\textit{right}), with the sets of directions $\{\dir^{(0)}_\ell\}_{\ell=1}^p$  as in \eqref{pw directions}
and the modified D-recipe stabilization \eqref{modified D-recipe} on Cartesian meshes.}
\label{fig:TEST4_D_cart} 
\end{figure}

\begin{center}
\begin{tabular}{|c||c||c|c||c|c||c|c|}
\hline 
$h$	& $\Nd$ & rel. $H^1$ error & rate & rel. $L^2$ error & rate & $\Nd$ orig. & red. ($\%$) \\ 
\hline \hline 
1.414e+00 	& 46 	 &  4.6885e-01 &  ---   & 4.7153e-01  &	 ---	& 48 & 4.17 \\ 
7.071e-01 	& 120 	 &  1.3527e-01 & 1.793  & 1.3185e-01 &	1.838 	& 144 & 16.67 \\
3.535e-01 	& 340 	 &  1.0540e-03 & 7.004  & 5.4861e-04 &	7.909 	& 480 & 29.17 \\
1.767e-01 	& 1008 	 &  6.1594e-06 & 7.419  & 1.4439e-06 &	8.570 	& 1728 & 41.67 \\
8.838e-02 	& 3264 	 &  4.2394e-08 & 7.183  & 4.4716e-09 &	8.335 	& 6528 & 50.00 \\
4.419e-02 	& 10560  &  1.6544e-07 & -1.964 & 7.3453e-08 &	-4.038  & 25344 & 58.33 \\
\hline 
\hline 
\end{tabular} 
\captionof{table}{Relative errors for $u_1$ in \eqref{exact solutions u0 u1} with $k=20$, $q=7$, and the directions $\{\dir^{(0)}_\ell\}_{\ell=1}^p$ as in \eqref{pw directions} on Cartesian meshes employing the modified method with the modified D-recipe stabilization \eqref{modified D-recipe}.}
\label{tab:TEST4_D_cart}
\end{center}

\begin{figure}[h]
\begin{center}
\begin{minipage}{0.48\textwidth} 
\includegraphics[width=\textwidth]{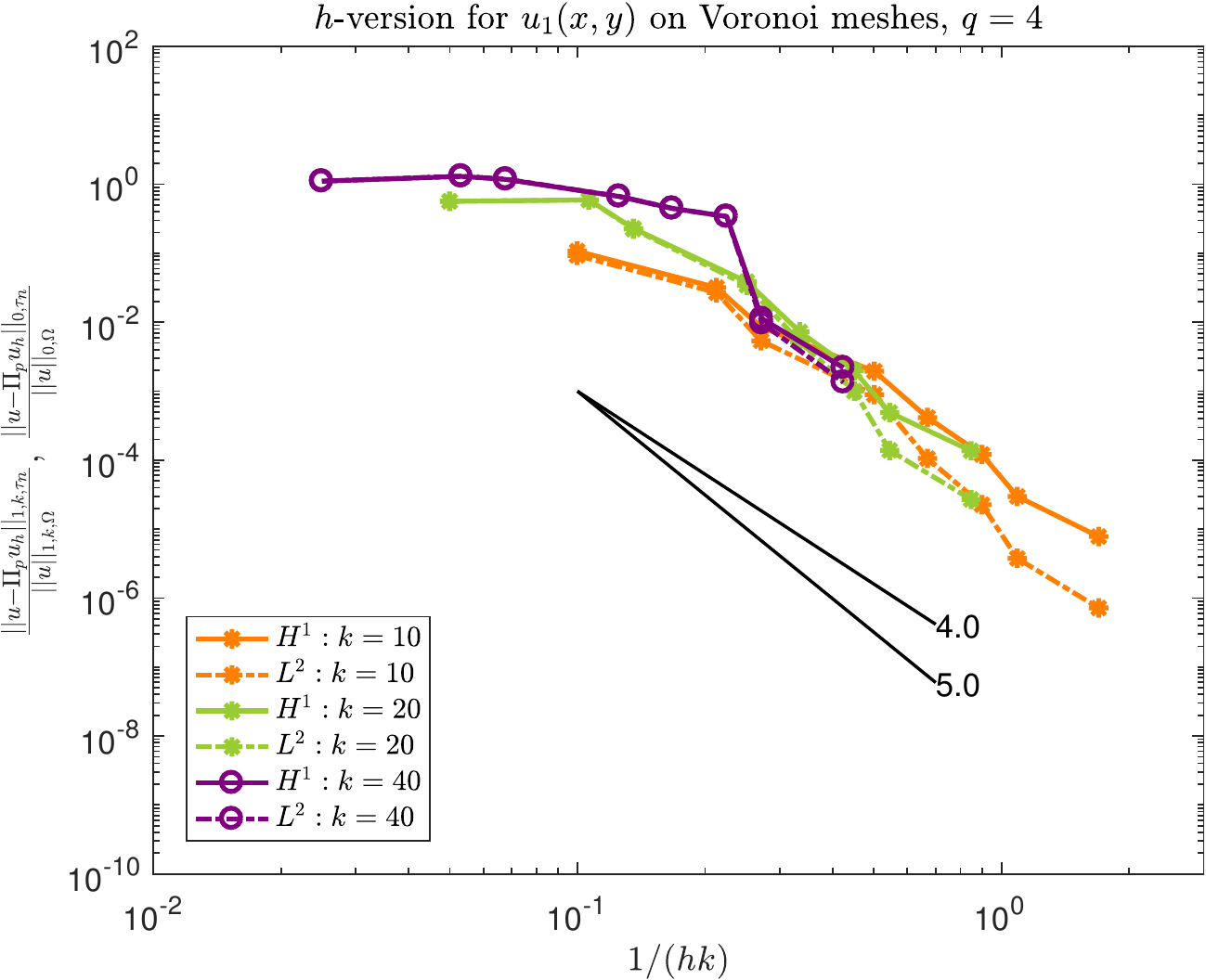}
\end{minipage}
\hfill
\begin{minipage}{0.48\textwidth}
\includegraphics[width=\textwidth]{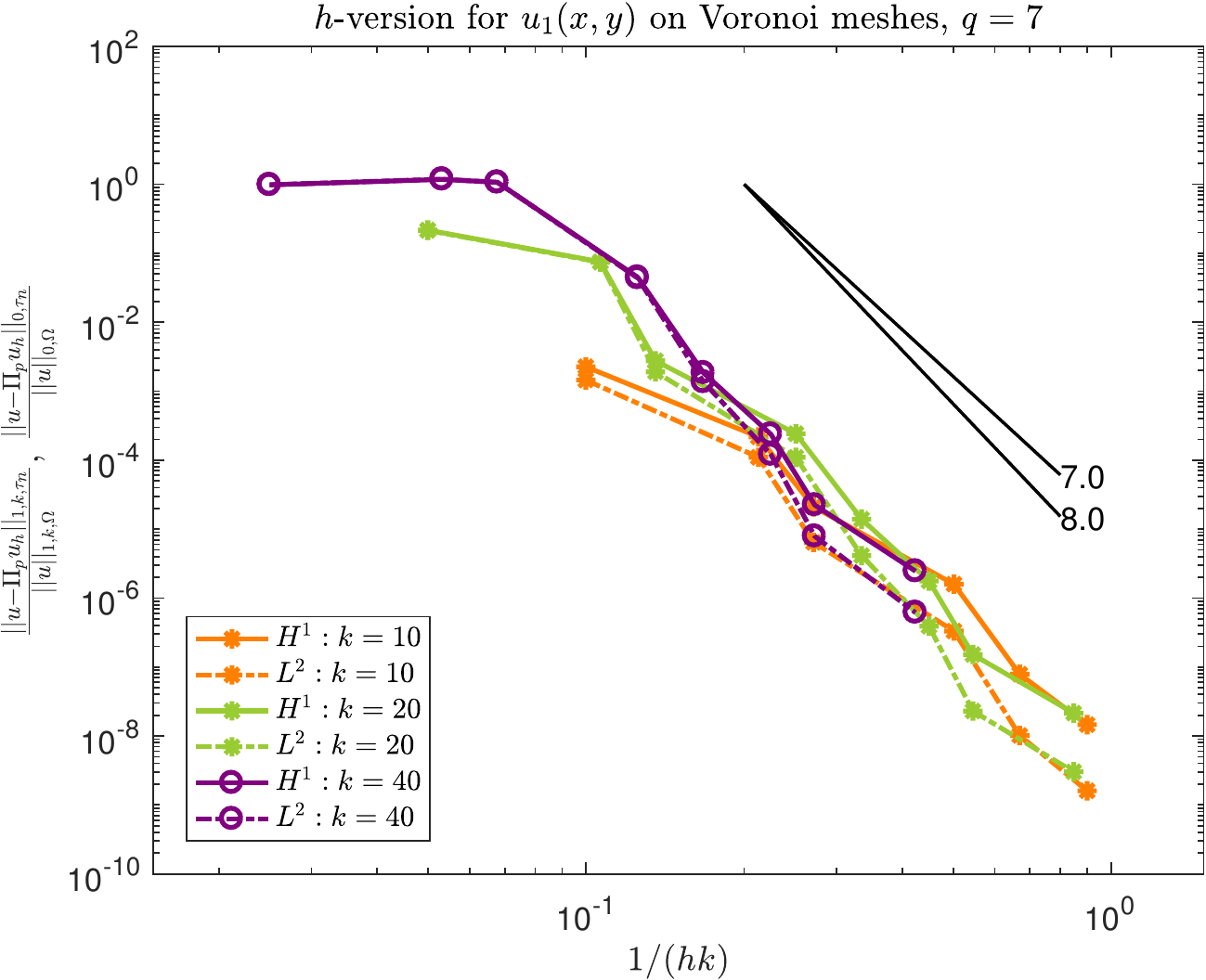}
\end{minipage}
\end{center}
\caption{$h$-version of the modified method for $u_1$ in \eqref{exact solutions u0 u1} with $\k=10$, $20$, and~$40$, and $q=4$ (\textit{left}) and $7$ (\textit{right}), with the sets of directions $\{\dir^{(0)}_\ell\}_{\ell=1}^p$ as in \eqref{pw directions}
and the modified D-recipe stabilization \eqref{modified D-recipe} on Voronoi meshes.}
\label{fig:TEST4_D_voro} 
\end{figure}

\begin{center}
\begin{tabular}{|c||c||c|c||c|c|}
\hline 
$h$	& $\Nd$ & rel. $H^1$ error & rel. $L^2$ error & $\Nd$ orig. & red. ($\%$) \\ 
\hline \hline 
1.001e+00 & 131 & 2.1704e-01 &  2.1440e-01 & 182 & 28.02 \\
4.697e-01 & 224 & 7.5289e-02 &  7.4015e-02 & 359 & 37.60 \\
3.688e-01 & 394 & 2.7605e-03 &  1.9061e-03 & 713 & 44.74 \\
1.993e-01 & 695 & 2.4147e-04 &  1.0970e-04 & 1477 & 52.95 \\
1.493e-01 & 1243& 1.3955e-05 &  4.1303e-06 & 2960 & 58.01 \\
1.111e-01 & 2206& 1.7662e-06 &  3.9013e-07 & 5998 & 63.22 \\
9.171e-02 & 4002& 1.5165e-07 &  2.3002e-08 & 12092 & 66.90 \\
5.896e-02 & 7282& 2.1462e-08 &  3.0271e-09 & 24304 & 70.04 \\
\hline 
\end{tabular} 
\captionof{table}{Relative errors for $u_1$ in \eqref{exact solutions u0 u1} with $k=20$, $q=7$, and the directions $\{\dir^{(0)}_\ell\}_{\ell=1}^p$ as in \eqref{pw directions} on Voronoi meshes employing the modified method with the modified D-recipe stabilization \eqref{modified D-recipe}.}
\label{tab:TEST4_voro_D}
\end{center}

We observe from Figures \ref{fig:TEST4_D_cart} and \ref{fig:TEST4_D_voro}, and Tables \ref{tab:TEST4_D_cart} and \ref{tab:TEST4_voro_D} that the approximate relative $H^1$ and $L^2$ discretization errors in~\eqref{rel_errors} of the method
approximately converge with rate $4$ and $5$ for $q=4$, and $7$ and $8$ for $q=7$, respectively.
This is in agreement with the error estimate derived in \cite{ncTVEM_theory}, which established, for $h \to 0$ and analytic solutions, convergence rates of order $q$ and $q+1$, for the relative $H^1$ and $L^2$ errors, respectively.
Note that due to the fact that the Voronoi meshes are not nested, the slopes indicating the convergence order are not as straight as in the Cartesian case.

In addition, we notice that the number of degrees of freedom was reduced significantly by making use of the orthonormalization process described in Algorithm \ref{algorithm orthog process}
in comparison to the original version of the method, which employs the standard filtering process in Algorithm~\ref{algorithm filtering process}. 

\medskip
Next, we employ the \textit{identity stabilization} \eqref{identity stabilization} and compare the performance with the modified D-recipe stabilization for~$u_1$ using the same meshes and parameters as above.
The results for the relative $H^1$ errors in~\eqref{rel_errors} are shown in Table~\ref{tab:comp_D_recipe_identity}. 
\begin{center}
\begin{minipage}{0.475\textwidth} 
\begin{tabular}{|c|c|c|c|}
\hline 
\multicolumn{4}{|c|}{Cartesian} \\
\hline
$h$	& $\Nd$ & D-recipe & identity   \\ 
\hline \hline 
1.414e+00 & 46 	  & 4.6885e-01  & 4.8651e-01 \\ 
7.071e-01 & 120   & 1.3527e-01  & 2.0525e-01 \\	
3.535e-01 & 340   & 1.0540e-03  & 2.4615e-02 \\	
1.767e-01 & 1008  & 6.1594e-06  & 1.7224e-03 \\	
8.838e-02 & 3264  & 4.2394e-08  & 1.2786e-05 \\	
4.419e-02 & 10560 & 1.6544e-07  & 6.4752e-07 \\
\hline 
\end{tabular} 
\end{minipage}
\hfill
\begin{minipage}{0.475\textwidth}
\begin{tabular}{|c|c|c|c|}
\hline 
\multicolumn{4}{|c|}{Voronoi} \\
\hline 
$h$	& $\Nd$ & D-recipe & identity   \\ 
\hline \hline 
1.001e+00 & 131  & 2.1704e-01 & 2.3510e-01 \\ 
4.697e-01 & 224  & 7.5289e-02 & 9.3167e-02 \\	
3.688e-01 & 394  & 2.7605e-03 & 2.4375e-02 \\	
1.993e-01 & 695  & 2.4147e-04 & 8.5729e-03 \\	
1.493e-01 & 1243 & 1.3955e-05 & 2.4687e-03 \\	
1.111e-01 & 2206 & 1.7662e-06 & 6.0640e-04 \\
\hline 
\end{tabular}
\end{minipage}
\captionof{table}{Relative $H^1$ errors for $u_1$ in \eqref{exact solutions u0 u1} with $k=20$, $q=7$, and the directions $\{\dir^{(0)}_\ell\}_{\ell=1}^p$ as in \eqref{pw directions} on Cartesian (\textit{left}) and Voronoi (\textit{right}) meshes employing the modified method
with the D-recipe stabilization~\eqref{modified D-recipe} and the identity stabilization~\eqref{identity stabilization}.}
\label{tab:comp_D_recipe_identity}
\end{center}
Compared to the modified D-recipe stabilization, the method based on the identity stabilization behaves worse. Similar results are obtained for the relative $L^2$ errors in \eqref{rel_errors}.
This fact highlights that picking a ``good'' stabilization is an important issue in the design of VEM~\cite{VEM3Dbasic,fetishVEM,fetishVEM3D}.

Thus, in the sequel, we will always consider the modified nonconforming Trefftz-VEM endowed with the modified D-recipe stabilization \eqref{modified D-recipe}.

\medskip
As a last test in this section, we study the $h$-version of the method for the non-analytic solution $u_3$ in \eqref{exact solutions u2 u3}. Once again we perform the tests on the Cartesian meshes with $\k=10$, $20$, and~$40$, and $\q=4$ and $7$, in Figure~\ref{fig:TEST5_D_cart}.
We point out that similar results were obtained employing Voronoi meshes.

\begin{figure}[h]
\begin{center}
\begin{minipage}{0.48\textwidth} 
\includegraphics[width=\textwidth]{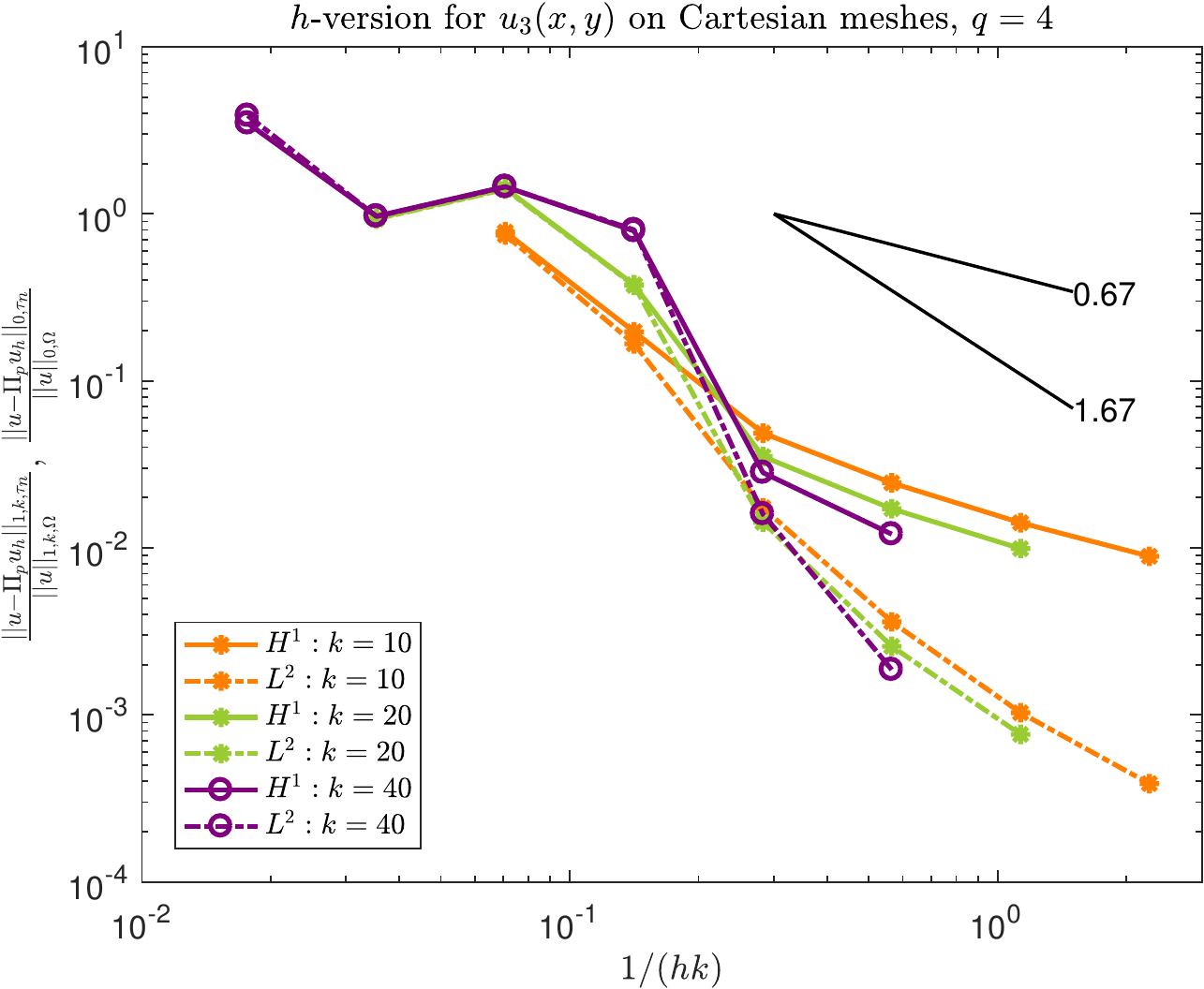}
\end{minipage}
\hfill
\begin{minipage}{0.48\textwidth}
\includegraphics[width=\textwidth]{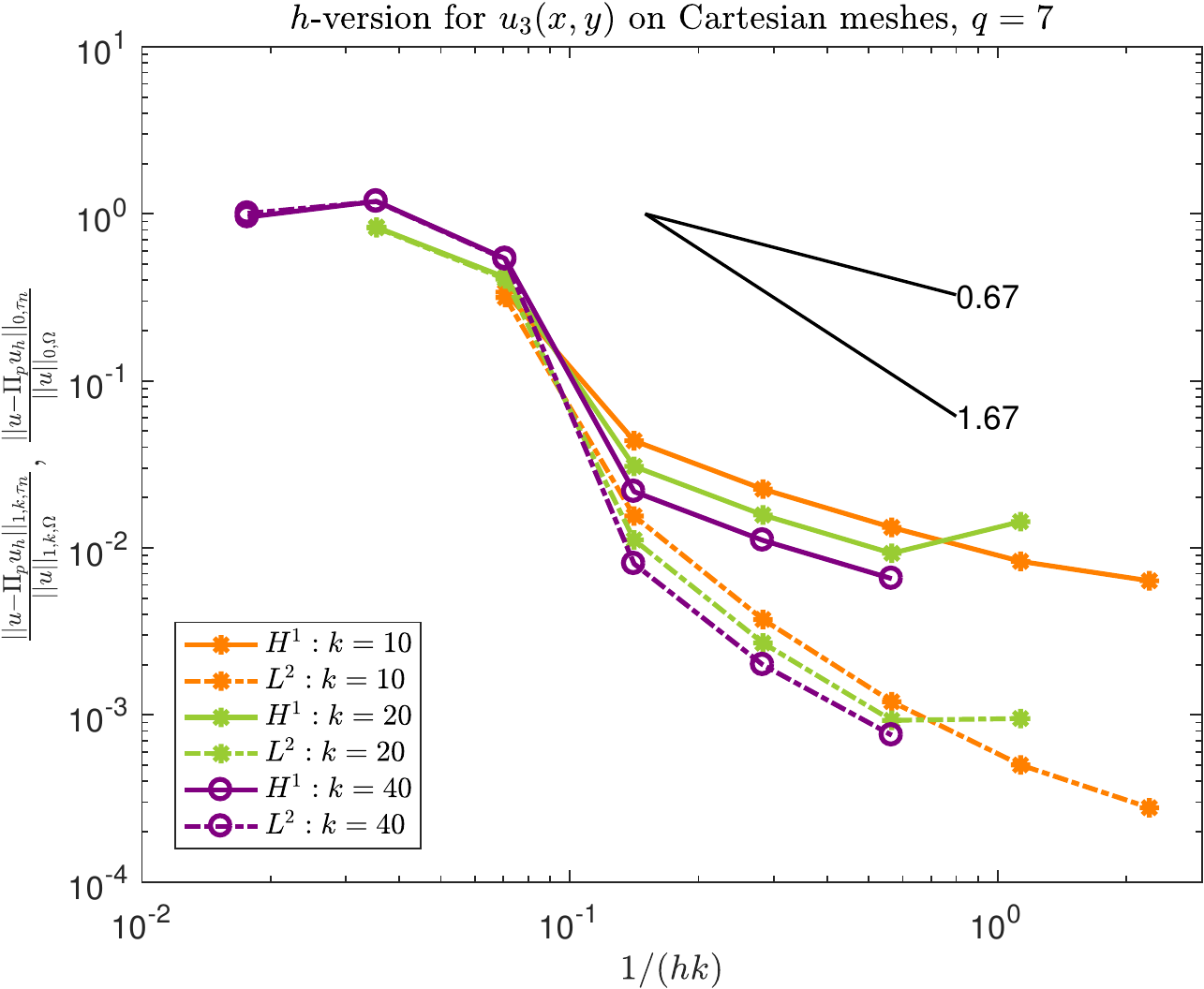}
\end{minipage}
\end{center}
\caption{$h$-version of the method for $u_3$ in \eqref{exact solutions u0 u1} with $\k=10$, $20$, and~$40$, and $q=4$ (\textit{left}) and $7$ (\textit{right}), with the sets of directions $\{\dir^{(0)}_\ell\}_{\ell=1}^p$ as in \eqref{pw directions} and the modified D-recipe stabilization \eqref{modified D-recipe} on Cartesian meshes.}
\label{fig:TEST5_D_cart} 
\end{figure}

The observed convergence rate for the approximate $H^1$ bulk error in~\eqref{rel_errors} is~$\frac{2}{3}$ and that for the approximate $L^2$ bulk error is~$\frac{5}{3}$.
This corresponds to the expected convergence rates $\min\{s,q\}$ and $\min\{s,q\}+1$ for the $H^1$ and $L^2$ errors, respectively, where $s$ is the regularity of the solution and $q$ is the effective plane wave degree, see~\cite{ncTVEM_theory}.

\begin{remark} \label{remark on sigma}
Here, we discuss and motivate the choice for the parameter~$\sigma$ in Algorithm~\ref{algorithm orthog process}, which so far has been set to~$10^{-13}$.
In principle, it would have been more natural to take $\sigma=10 \, \eps$, where eps denotes the machine epsilon.
With this choice, it would be basically guaranteed that the span of the filtered orthonormalized edge plane wave functions coincides with the non-orthonormalized edge plane wave space, up to a negligible difference.
However, 
we could observe from numerical experiments that with smaller choices of $\sigma$, such as $10^{-13}$, it is possible to achieve the same accuracy as when employing $\sigma=10 \,\eps$, but with less degrees of freedom, see Table \ref{tab:comparison sigma},
where we tested the $h$-version of the modified nonconforming Trefftz-VEM with analytical solution $u_2$ in \eqref{exact solutions u2 u3} on a sequence of Voronoi-Llyod meshes of the type in Figure~\ref{fig:meshes} (right) for the two above-mentioned choices of $\sigma$ with $k=10$ and $q=7$.

\begin{center}
\begin{tabular}{|c||c|c||c|c|}
\hline 
& \multicolumn{2}{|c||}{$\sigma=10 \, \eps$} & \multicolumn{2}{|c|}{$\sigma=10^{-13}$}  \\
\hline
$h$	& $\Nd$ & rel. $L^2$ error  & $\Nd$ & rel. $L^2$ error  \\ 
\hline \hline 
1.001346e+00 	& 113 	& 6.174135e-03  &  106 	& 6.147714e-03   \\
4.697545e-01 	& 201 	& 4.285982e-04  &  189 	& 4.337061e-04   \\
3.688297e-01 	& 353 	& 6.529610e-05  &  327 	& 6.250524e-05   \\
1.993180e-01 	& 631 	& 6.754430e-06  &  578 	& 6.625276e-06   \\
1.493758e-01 	& 1139 	& 1.572124e-07  &  1037 & 1.512503e-07   \\
1.111597e-01 	& 2053 	& 6.369678e-08  &  1886 & 6.294611e-08   \\
9.171171e-02 	& 3745 	& 2.514794e-08  &  3445 & 2.441118e-08   \\
\hline 
\end{tabular} 
\captionof{table}{$h$-version of the modified method for the analytical solution $u_2$ in \eqref{exact solutions u2 u3}, $k=10$, $q=7$, on Voronoi-Lloyd meshes of the type in Figure~\ref{fig:meshes} with different choices of $\sigma$ in Algorithm~\ref{algorithm orthog process}. The relative $L^2$ errors are computed accordingly with \eqref{rel_errors}.}
\label{tab:comparison sigma}
\end{center}
\end{remark}

\paragraph{Application to an acoustic scattering problem.} \label{subsection scattering}
In this section, we consider the scattering of acoustic waves at a scatterer $\Omega_{Sc} \subset \R^2$ with polygonal boundary $\Gamma_{Sc}$.
We study the cases of a sound-soft and sound-hard scatterers. The total field $u=u^S+u^I$, $u^S$ and $u^I$ denoting the scattered and the incident fields, respectively, satisfies
\begin{equation*} 
(i)\left\{
\begin{alignedat}{2}
-\Delta u -\k^2 u &= 0 	&&\quad \text{in } \Omega_{Sc}^c \\
u &= 0 &&\quad \text{on } \Gamma_{Sc},
\end{alignedat}
\hspace{2cm}
\right.
(ii) \left\{
\begin{alignedat}{2}
-\Delta u -\k^2 u &= 0 	&&\quad \text{in } \Omega_{Sc}^c\\
\nabla u \cdot \nOmega &= 0 &&\quad \text{on } \Gamma_{Sc},
\end{alignedat}
\right.
\end{equation*}
respectively, where $\Omega_{Sc}^c:=\R^2\backslash \overline{\Omega_{Sc}}$, and both problems are endowed with the Sommerfeld radiation condition at infinity:
\begin{equation} \label{Sommerfeld}
\lim_{|\x| \to \infty} |\x| \left( \frac{\partial u^S(\x)}{\partial |\x|} +\im k u^S(\x) \right) = 0,
\end{equation}
see~\cite[Sect. 2.1]{coltoninverse}.

By truncating the unbounded domain $\Omega_{Sc}^c$ and approximating the Sommerfeld radiation condition~\eqref{Sommerfeld} by a first order absorbing impedance condition on the artificial boundary, one obtains
\begin{equation} \label{scattering_problems}
(iii) \left\{
\begin{alignedat}{2}
-\Delta u -\k^2 u &= 0 	&&\quad \text{in } \Omega\\
u &= 0 &&\quad \text{on } \Gamma_{Sc}\\
\nabla u \cdot \nOmega + \im \k  u  &= \gR &&\quad \text{on } \GammaR,\\
\end{alignedat}
\hspace{2cm}
\right.
(iv) \left\{
\begin{alignedat}{2}
-\Delta u -\k^2 u &= 0 	&&\quad \text{in } \Omega\\
\nabla u \cdot \nOmega &= 0 &&\quad \text{on } \Gamma_{Sc}\\
\nabla u \cdot \nOmega + \im \k u  &= \gR &&\quad \text{on } \GammaR,\\
\end{alignedat}
\right.
\end{equation}
where $\Omega:=\Omega_R \backslash \overline{\Omega_{Sc}}$, with $\Omega_R$ denoting the truncated domain with boundary $\GammaR$, and $\gR=\nabla u^I \cdot \n_\Omega+\im\k \theta u^I$ is the impedance trace of the incoming wave.
Both problems $(iii)$ and $(iv)$ in \eqref{scattering_problems} are well-posed, according to Theorem \ref{thm well-posedness HH}. Note that in the context of acoustic scattering, the unknown function $u$ in \eqref{scattering_problems} represents the acoustic pressure, rather than the displacement.

For the numerical tests, we fix $\Omega=(-1,2) \times (0,3) \, \backslash \, [0,1] \times [1,2]$ and employ uniform Cartesian meshes, see Figure \ref{figure scatterer}.

\begin{figure}[h]
\begin{center}
\begin{minipage}{0.3\textwidth}
\includegraphics[width=\textwidth]{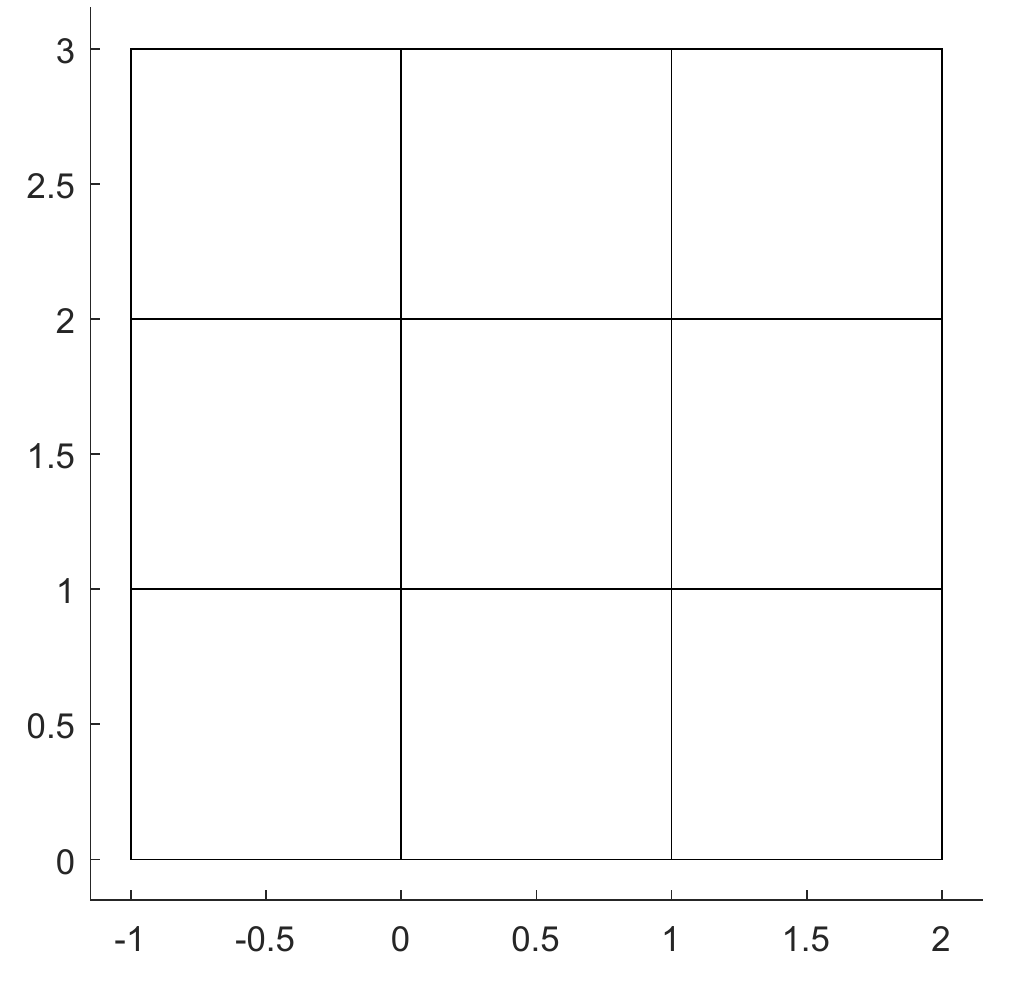}
\end{minipage}
\hfill
\begin{minipage}{0.3\textwidth}
\includegraphics[width=\textwidth]{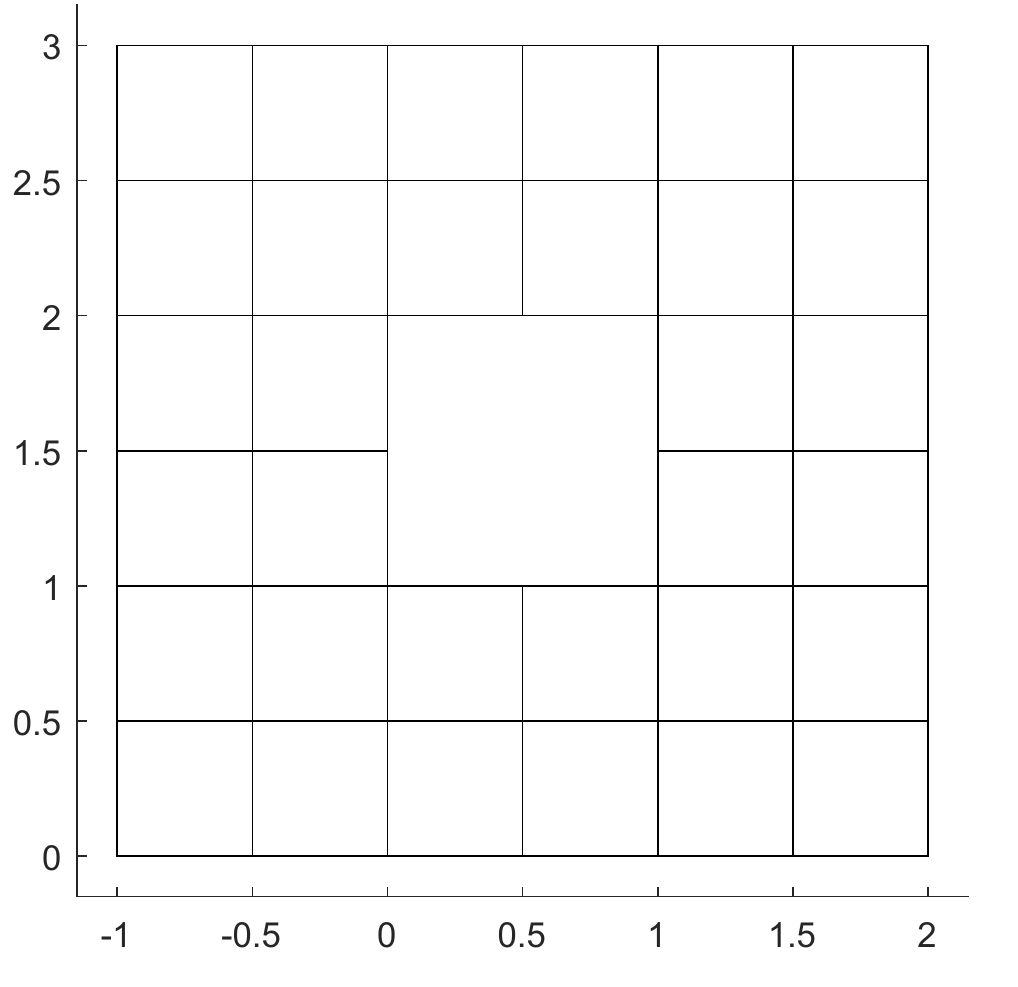}
\end{minipage}
\hfill
\begin{minipage}{0.3\textwidth}
\includegraphics[width=\textwidth]{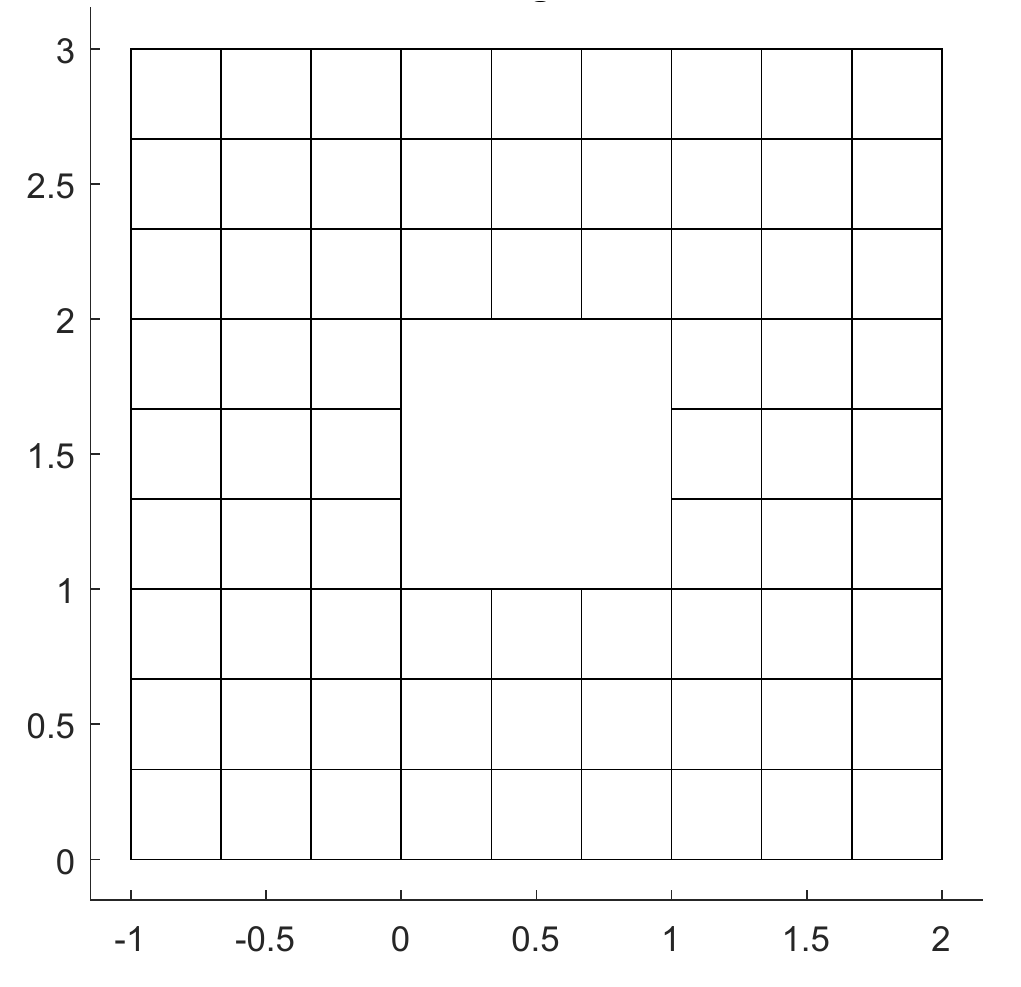}
\end{minipage}
\end{center}
\caption{First three Cartesian meshes in the decomposition over the domain $\Omega=(-1,2) \times (0,3) \, \backslash \, [0,1] \times [1,2]$.}
\label{figure scatterer}
\end{figure}

As incident fields, we consider the plane wave functions $u_0$ and $u_1$ in \eqref{exact solutions u0 u1}, as well as the plane wave given by
\begin{equation} \label{exact solutions u4}
u_4(x,y):=\exp\left(\im\k\left(\cos\left(\frac{2\pi}{17}\right)x+\sin\left(\frac{2\pi}{17}\right)y\right)\right).
\end{equation}
In Figures \ref{fig:scattering_sound_soft} and \ref{fig:scattering_sound_hard}, the real parts of the computed total fields for the sound-hard and sound-soft cases, respectively, are plotted for the different incident fields with $k=15$. As effective plane wave degree we choose $q=10$ (namely $p=21$ bulk plane waves). 

\begin{figure}[H]
\begin{minipage}{0.32\textwidth} 
\includegraphics[width=\textwidth]{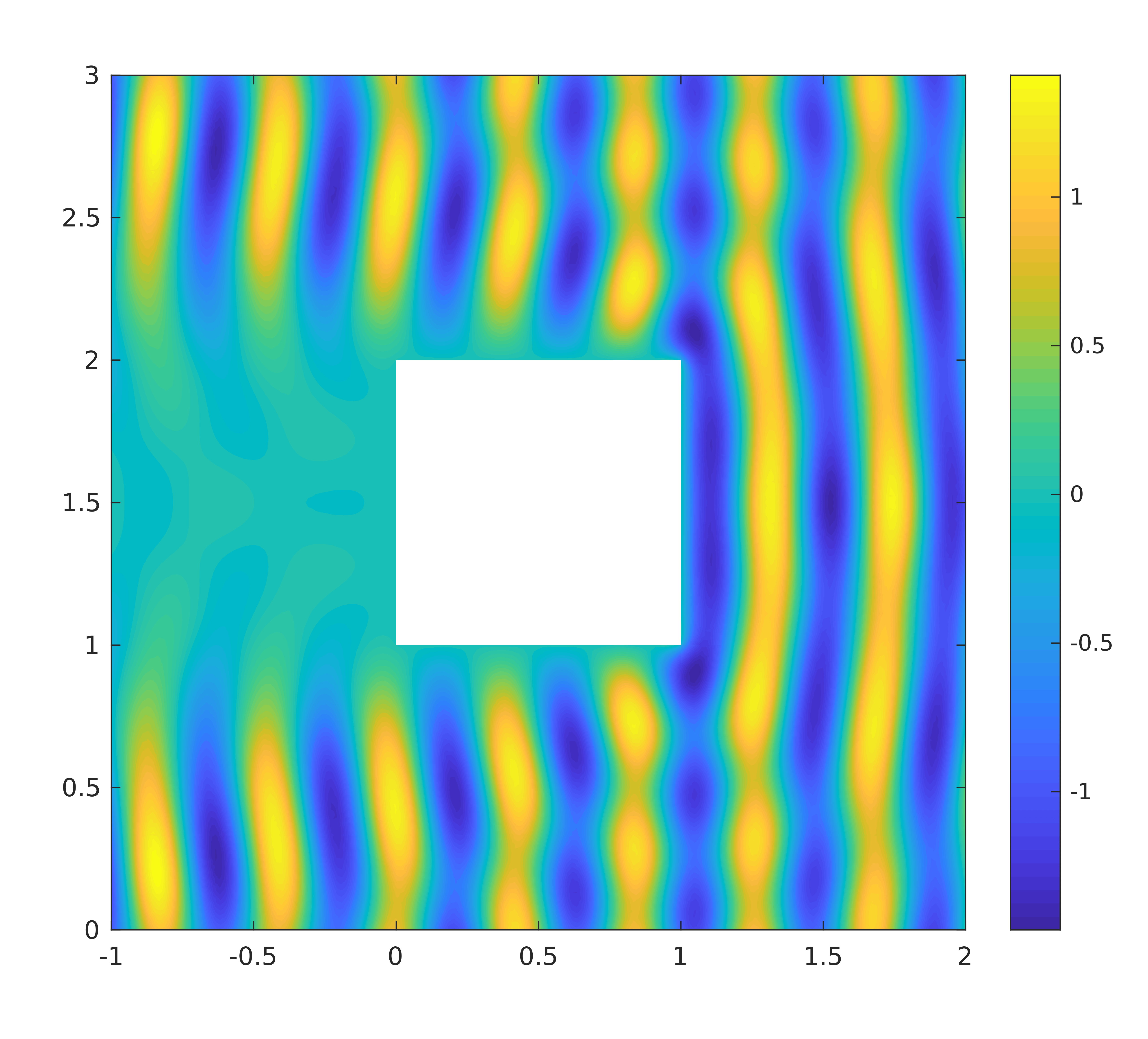}
\end{minipage}
\hfill
\begin{minipage}{0.32\textwidth}
\includegraphics[width=\textwidth]{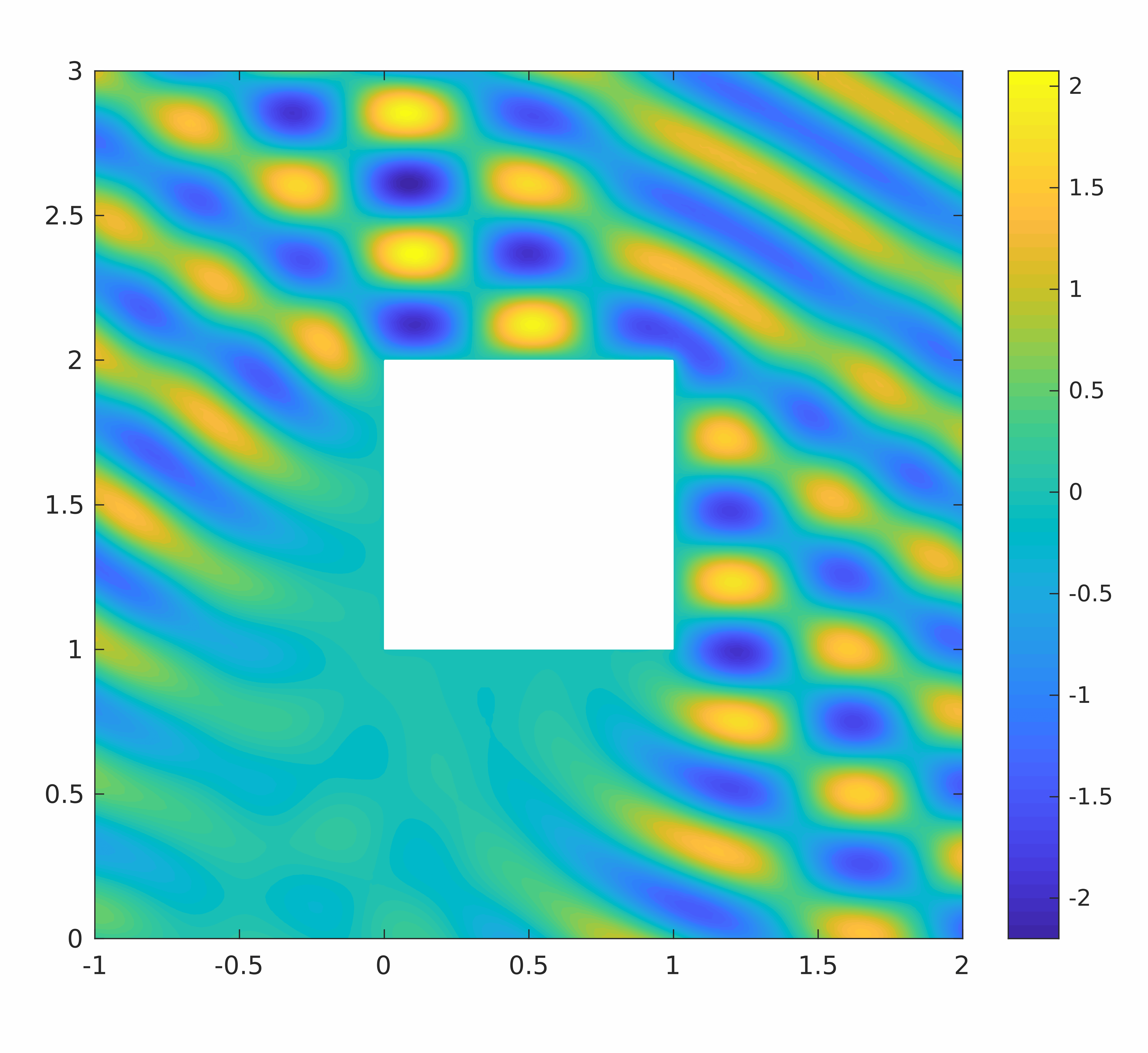}
\end{minipage}
\hfill
\begin{minipage}{0.32\textwidth}
\includegraphics[width=\textwidth]{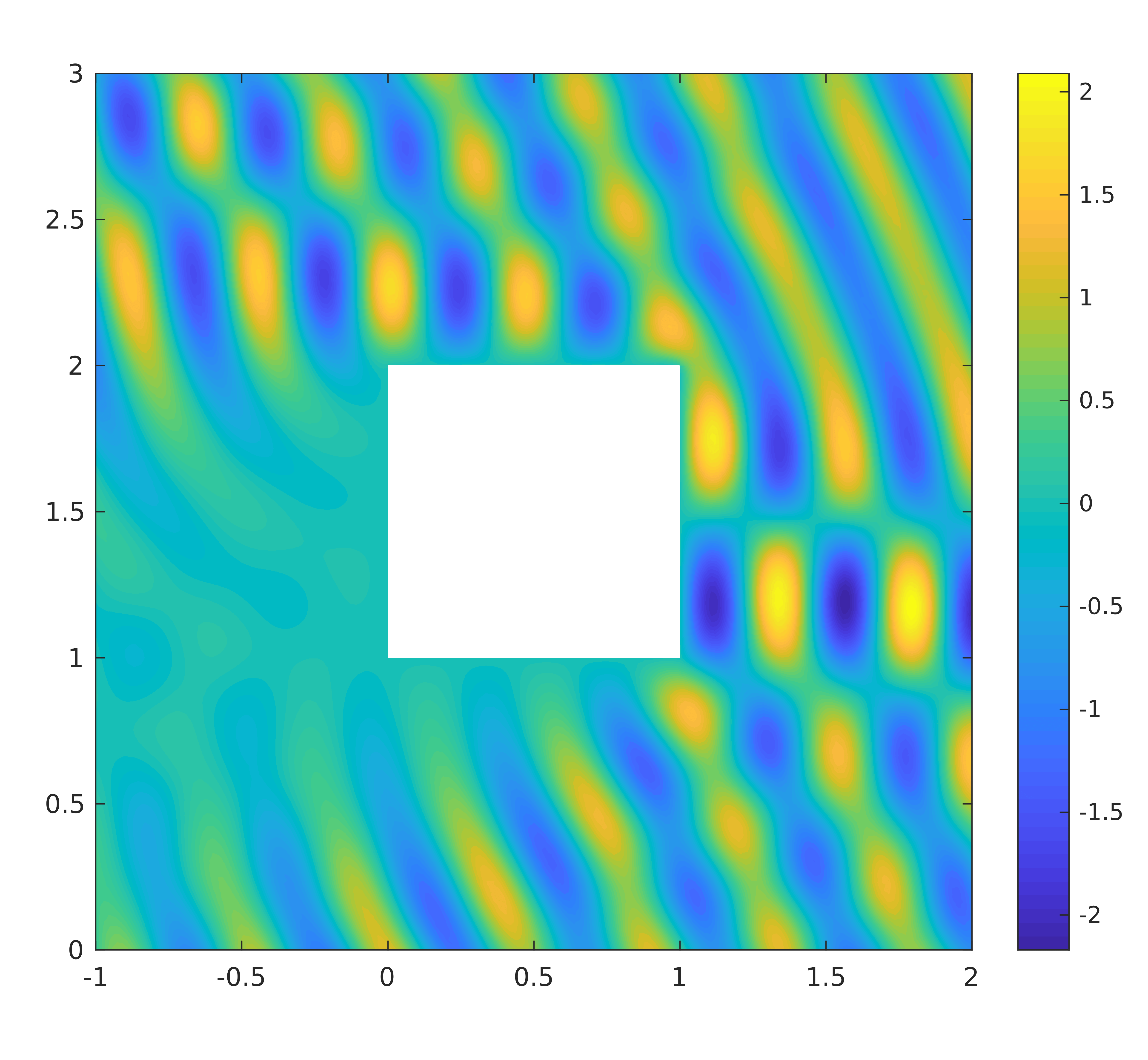}
\end{minipage}
\caption{Real parts of the total fields for the sound-soft scattering employing as incident field the plane waves given by $u_0$ (\textit{left}) and $u_1$ (\textit{center}) in \eqref{exact solutions u0 u1}, and $u_4$ (\textit{right}) in \eqref{exact solutions u4}, with $k=15$.}
\label{fig:scattering_sound_soft} 
\end{figure}

\begin{figure}[H]
\begin{minipage}{0.32\textwidth} 
\includegraphics[width=\textwidth]{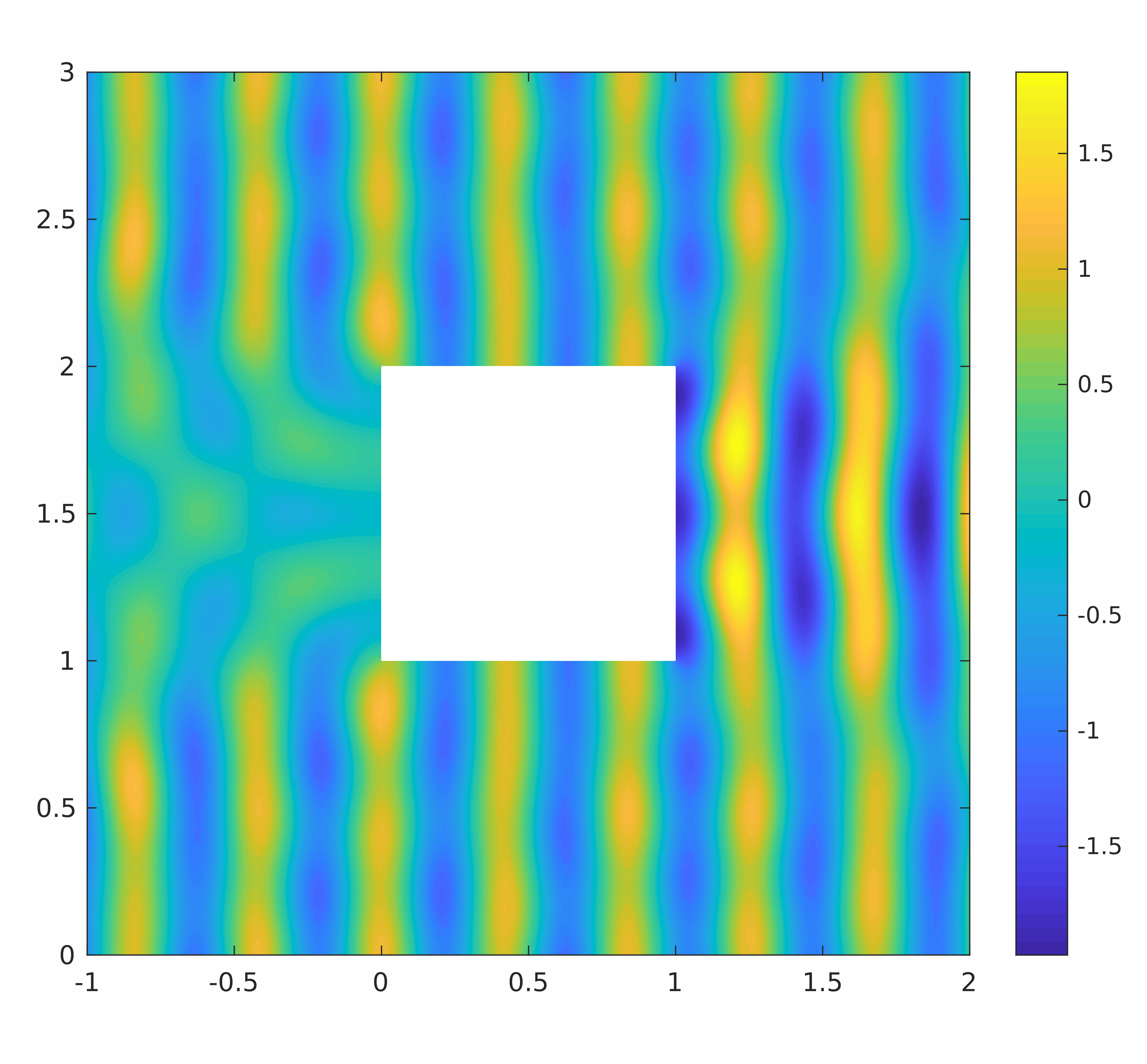}
\end{minipage}
\hfill
\begin{minipage}{0.32\textwidth}
\includegraphics[width=\textwidth]{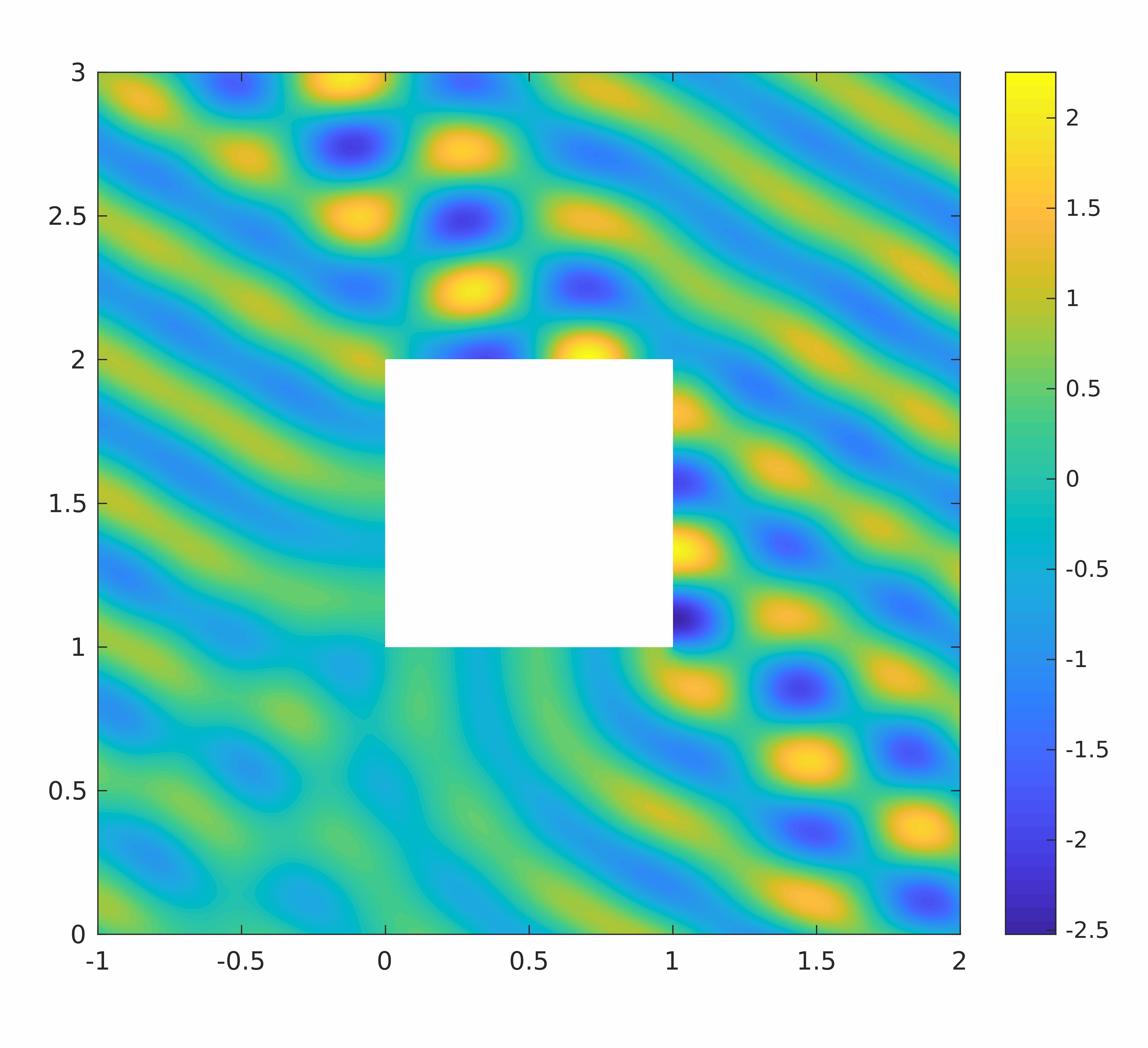}
\end{minipage}
\hfill
\begin{minipage}{0.32\textwidth}
\includegraphics[width=\textwidth]{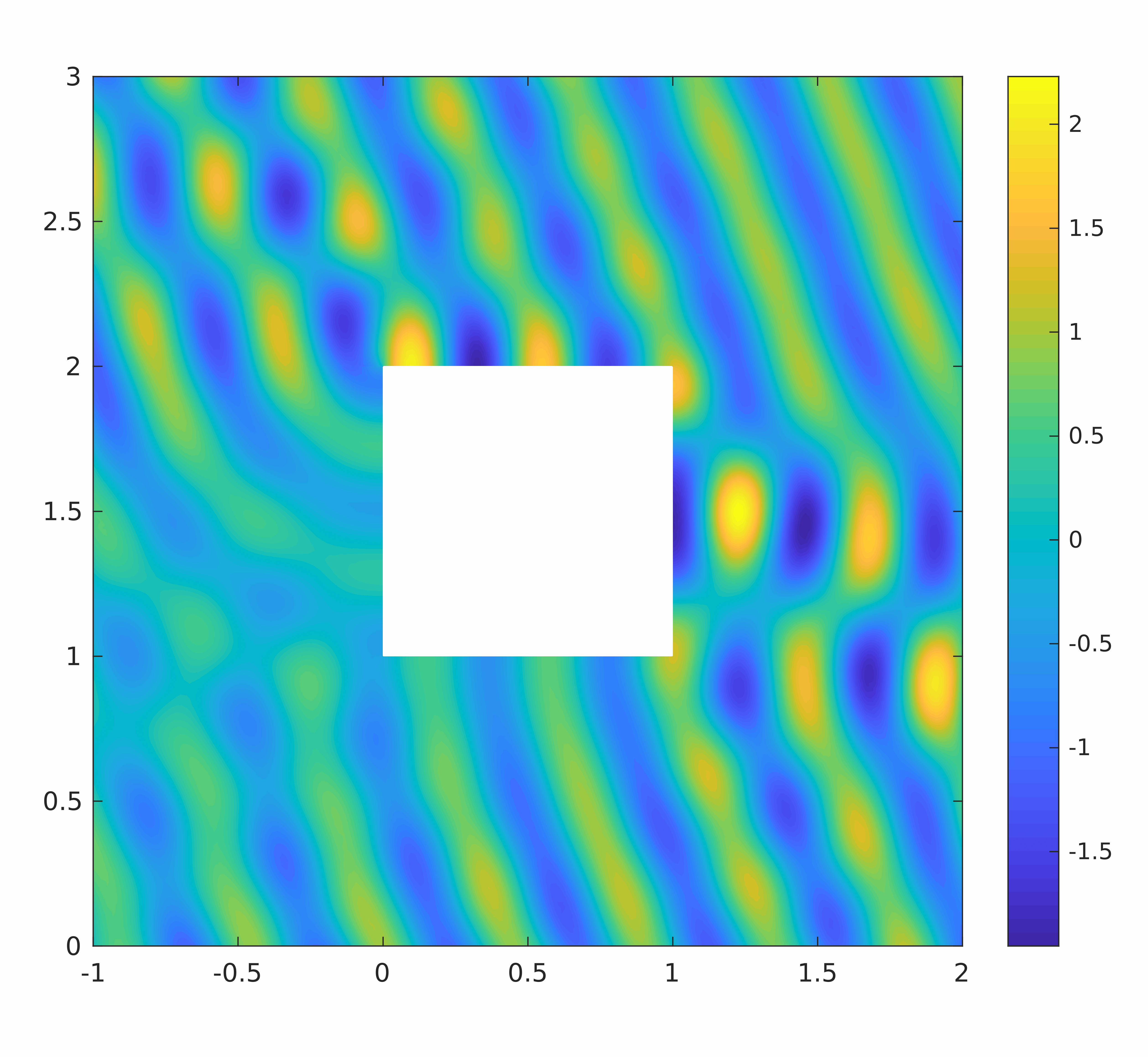}
\end{minipage}
\caption{Real parts of the total fields for the sound-hard scattering employing as incident field the plane waves given by $u_0$ (\textit{left}) and $u_1$ (\textit{center}) in \eqref{exact solutions u0 u1}, and $u_4$ (\textit{right}) in \eqref{exact solutions u4}, with $k=15$.}
\label{fig:scattering_sound_hard} 
\end{figure}

The relative errors are computed accordingly with~\eqref{rel_errors}, where, since an exact solution~$u$ is not known in closed form, $u$ was chosen to be the discrete solution on a very fine mesh. In Figure~\ref{fig:scattering_h}, the obtained results are plotted. 
\begin{figure}[h]
\begin{minipage}{0.475\textwidth} 
\includegraphics[width=\textwidth]{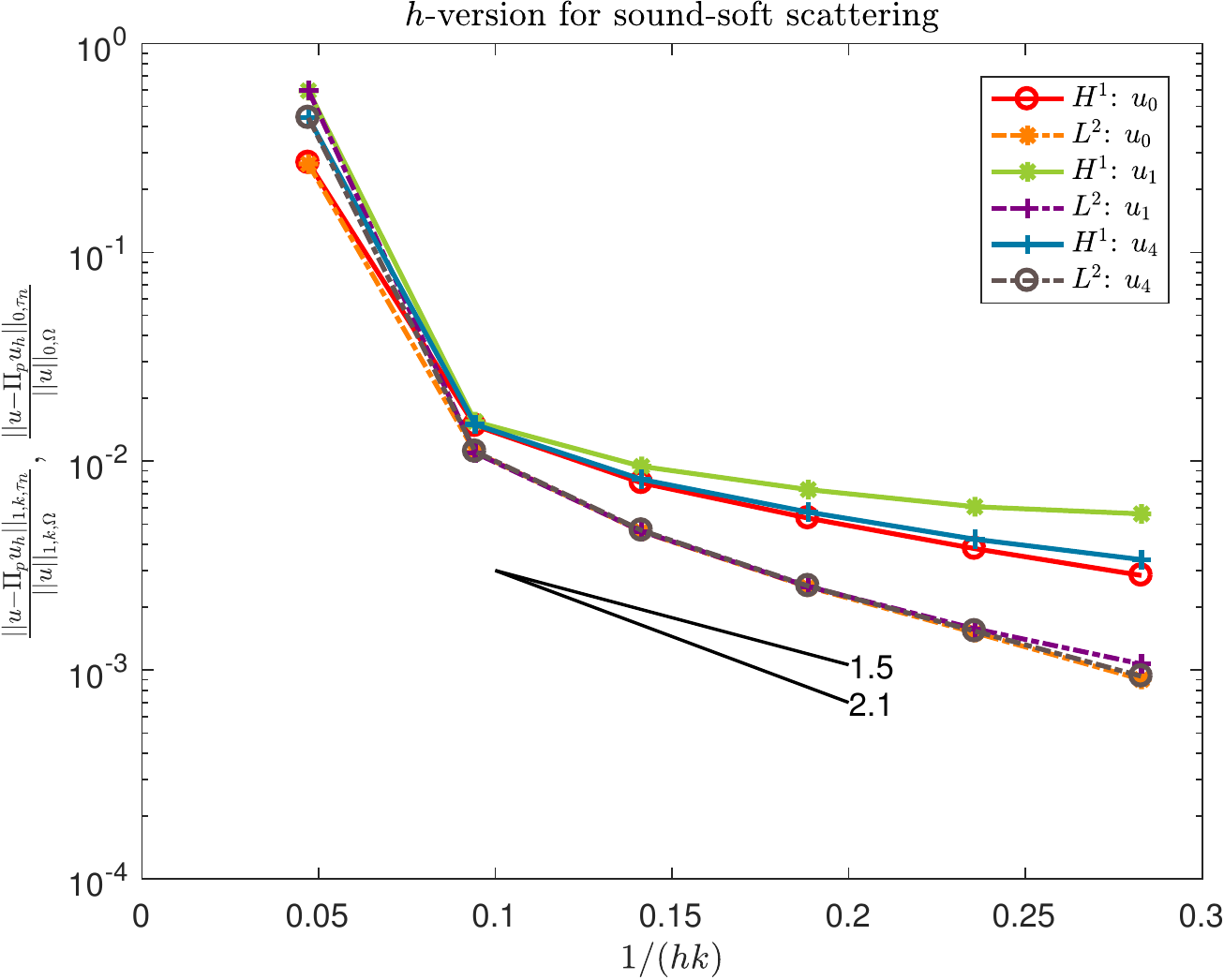}
\end{minipage}
\hfill
\begin{minipage}{0.475\textwidth}
\includegraphics[width=\textwidth]{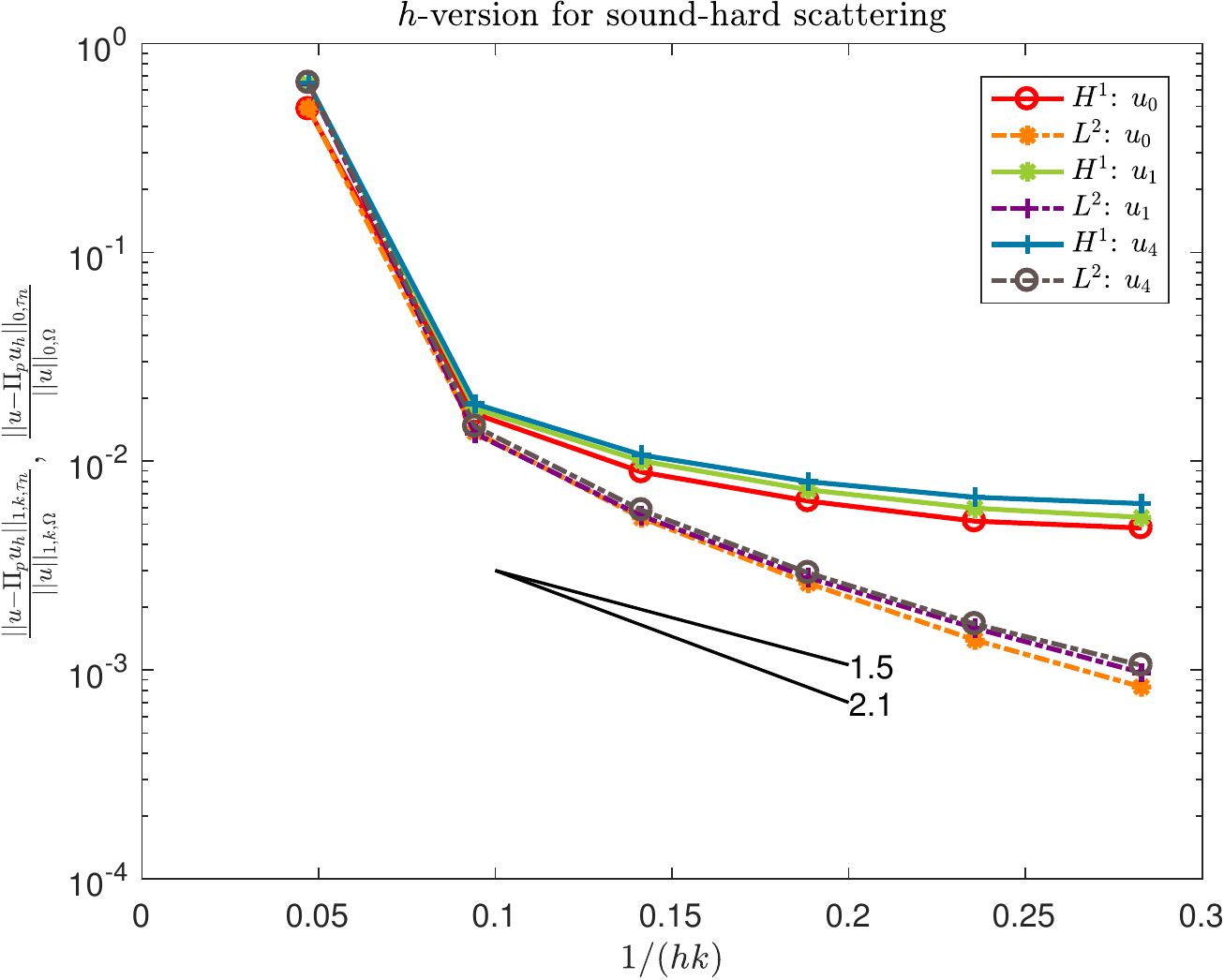}
\end{minipage}
\caption{$h$-version of the modified method for the scattering problems~\eqref{scattering_problems} with $k=15$ and $q=10$. \textit{Left:} sound-soft scattering; \textit{right:} sound-hard scattering. The relative errors are computed accordingly with \eqref{rel_errors}.}
\label{fig:scattering_h} 
\end{figure}
In both cases, the convergence rates are approximately~$1.5$ and~$2.1$ for the relative~$H^1$ and~$L^2$ errors, respectively. 

\subsubsection{$\p$-version}
We test numerically the $\p$-version of the modified nonconforming Trefftz-VEM, that is, we achieve convergence by keeping fixed a mesh and increasing the local effective degree.
To this end, we consider the two meshes shown in Figure \ref{fig:meshes p}. Each of them consists of eight elements. The first one is a Voronoi-Lloyd mesh, and the second is a mesh whose elements are not star-shaped with respect to any ball.
In the sequel, we will refer to these meshes as mesh (a) and mesh (b), respectively.

\begin{figure}[h]
\begin{minipage}{0.49\textwidth}
\centering
\includegraphics[width=0.7\textwidth]{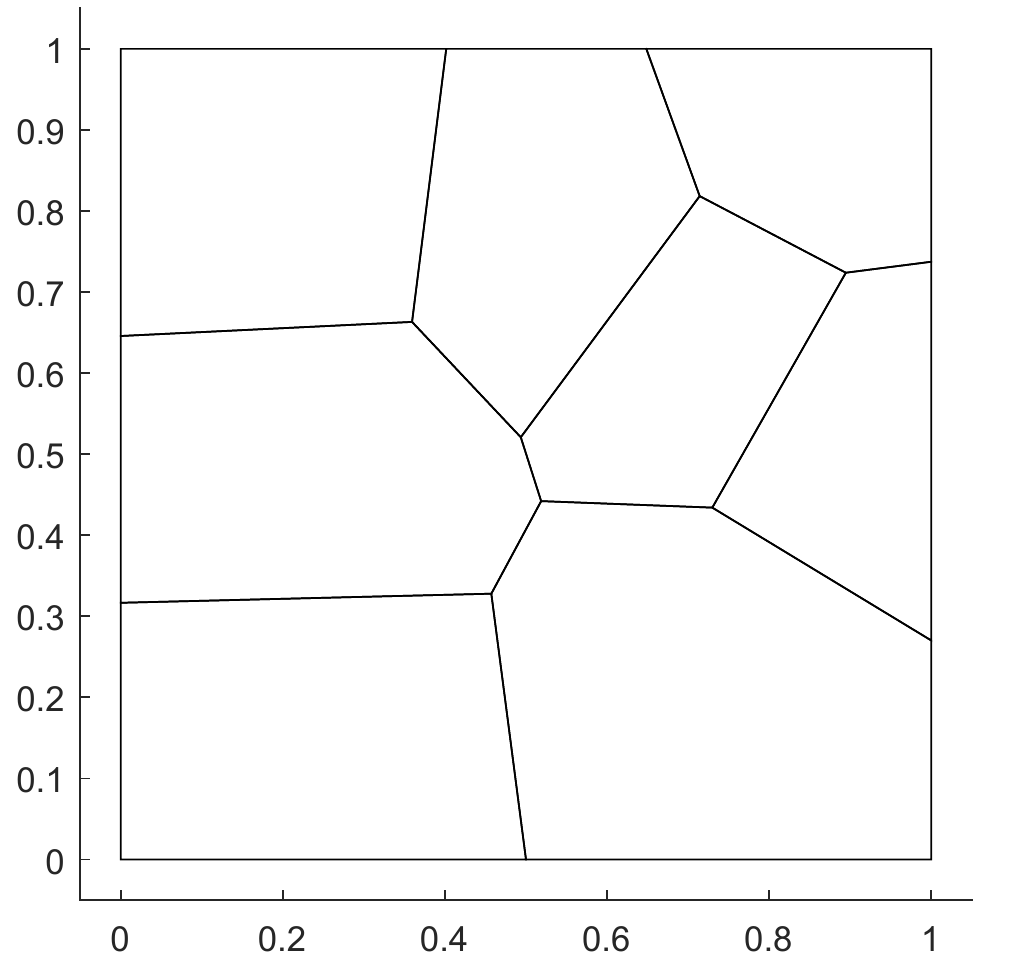}
\end{minipage}
\hfill
\begin{minipage}{0.49\textwidth}
\centering
\includegraphics[width=0.7\textwidth]{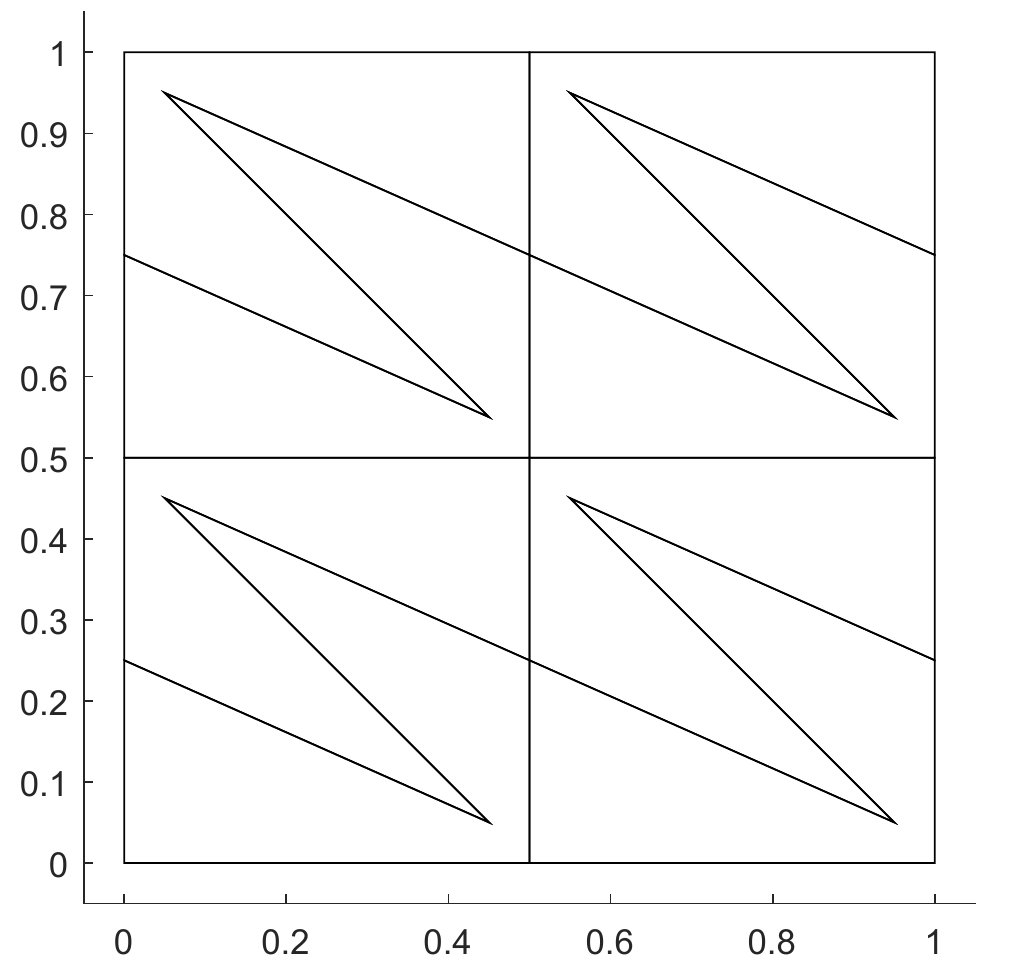}
\end{minipage}
\caption{Different types of meshes made of eight elements; \textit{left}: mesh (a), \textit{right}: mesh (b).}
\label{fig:meshes p} 
\end{figure}

To start with, we first investigate the $\p$-version of the modified method for the test case with analytical solution $u_1$ in \eqref{exact solutions u0 u1}, employing different values of $\k=10$, $20$, and~$40$. The obtained numerical results are shown in Figure \ref{fig:TEST6}.

Also in this case, the tests with analytical solution $u_2$ in \eqref{exact solutions u2 u3} lead to similar results and are postponed to the forthcoming Section~\ref{section comp Trefftz PWVEM}, where the results are compared with the PWVEM and the PWDG.

\begin{figure}[h]
\begin{minipage}{0.48\textwidth}
\includegraphics[width=\textwidth]{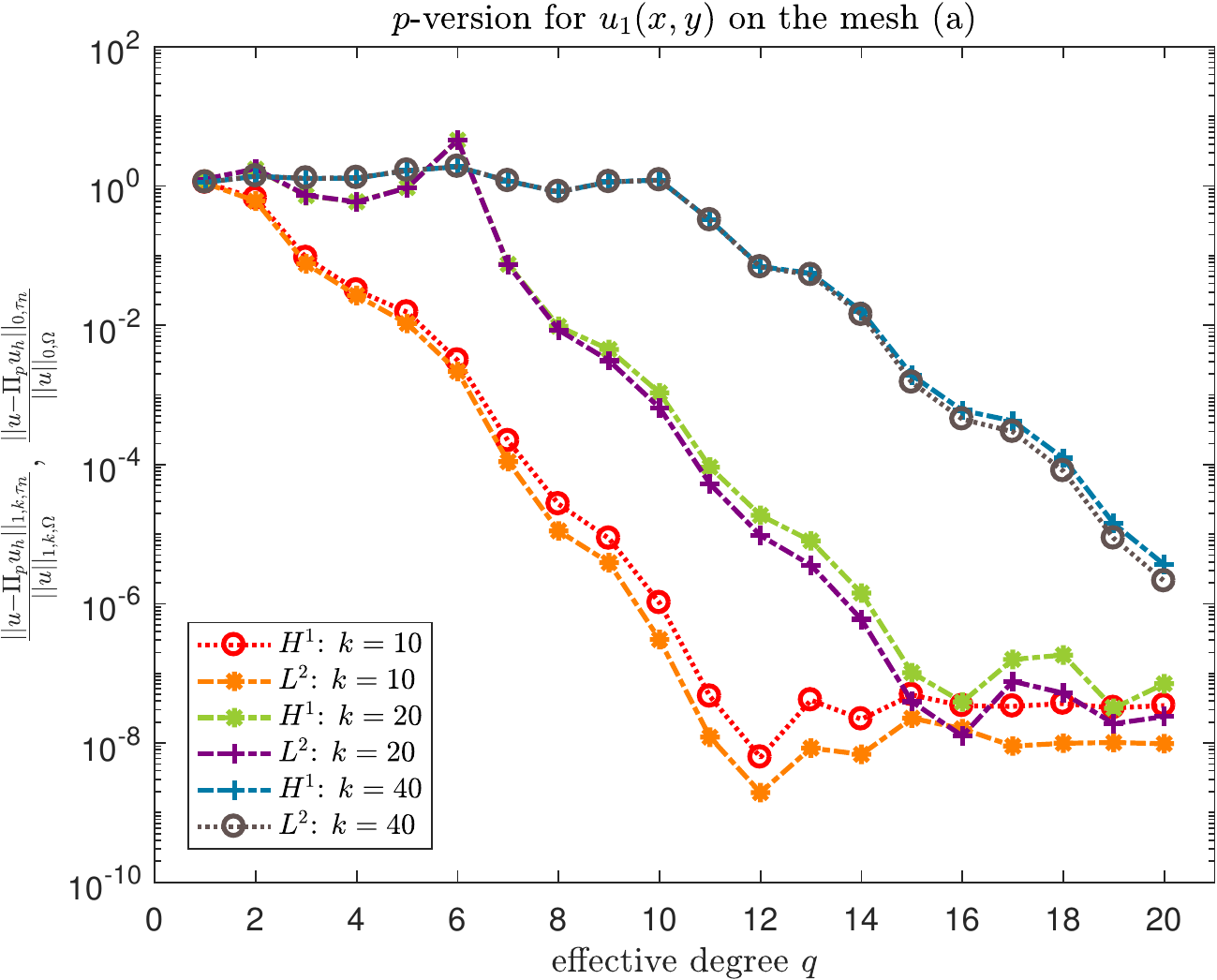}
\end{minipage}
\hspace{0.25cm}
\begin{minipage}{0.48\textwidth}
\includegraphics[width=\textwidth]{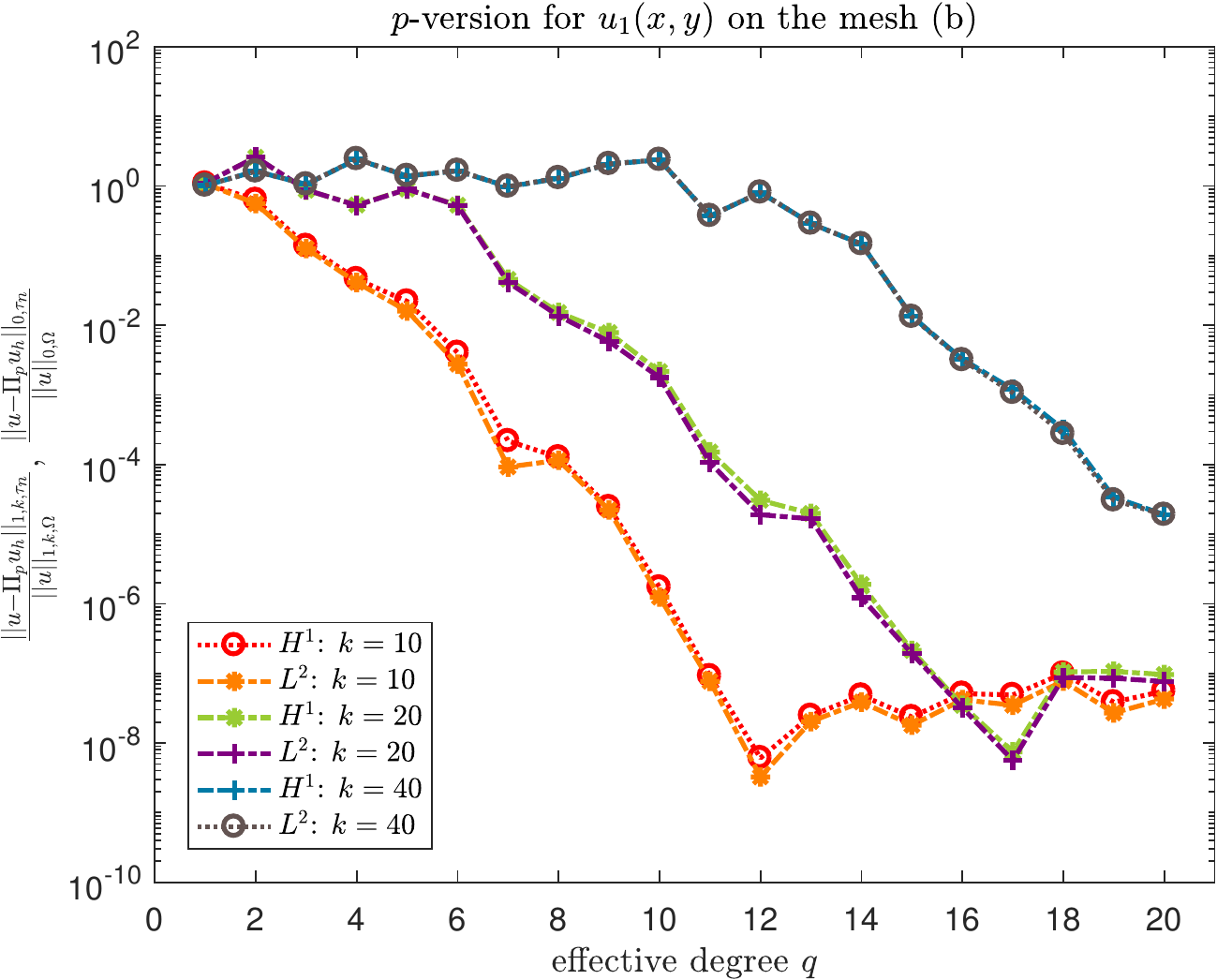}
\end{minipage}
\caption{$p$-version of the modified method for $u_1$ in \eqref{exact solutions u0 u1} on mesh (a) and (b) in Figure \ref{fig:meshes p}, from \textit{left} to \textit{right}.}
\label{fig:TEST6} 
\end{figure}

For both meshes, we observe that after a pre-asymptotic regime, the modified method is able to reach exponential convergence in terms of the effective degree $\q$, before instability takes place, caused by the haunting ill-conditioning of the plane wave basis.
The pre-asymptotic regime is much wider for higher wave numbers, which is typical of plane wave-based methods.
We underline that, despite the $p$-version of the nonconforming Trefftz-VEM has not been investigated theoretically yet, the exponential decay of the error for analytic solutions is not surprising, cf. \cite{SchwabpandhpFEM,hpVEMbasic,hiptmair2011plane}.

Next, we perform the same experiments on the non-analytic exact solution~$u_3$ in~\eqref{exact solutions u2 u3}. The corresponding plots are depicted in Figure~\ref{fig:TEST7}.
We notice that the convergence rate is not exponential any more, but rather algebraic. This is also an expected behavior of the $p$-version \cite{SchwabpandhpFEM,hpVEMbasic,hiptmair2011plane}.

\begin{figure}[h]
\begin{minipage}{0.48\textwidth}
\includegraphics[width=\textwidth]{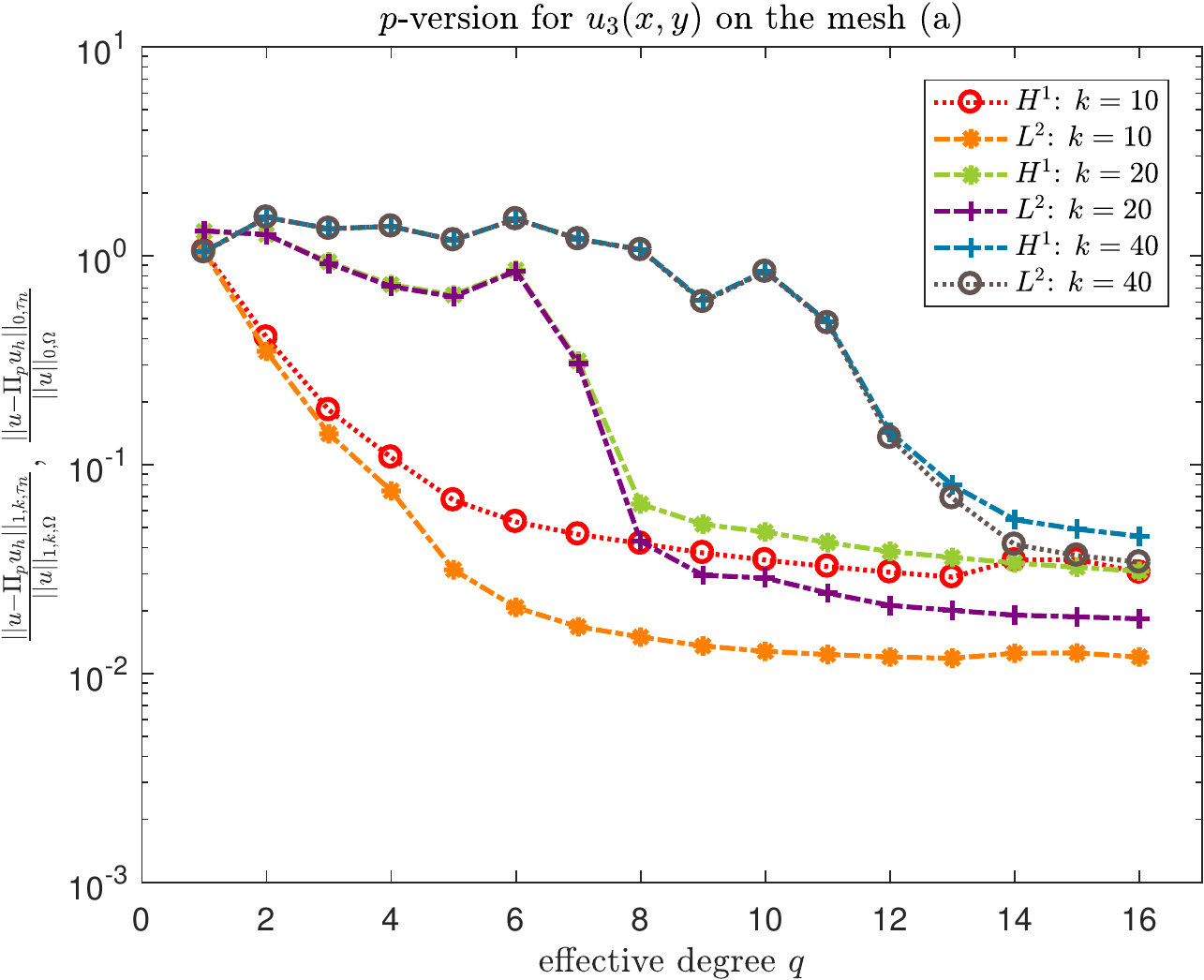}
\end{minipage}
\hspace{0.25cm}
\begin{minipage}{0.48\textwidth}
\includegraphics[width=\textwidth]{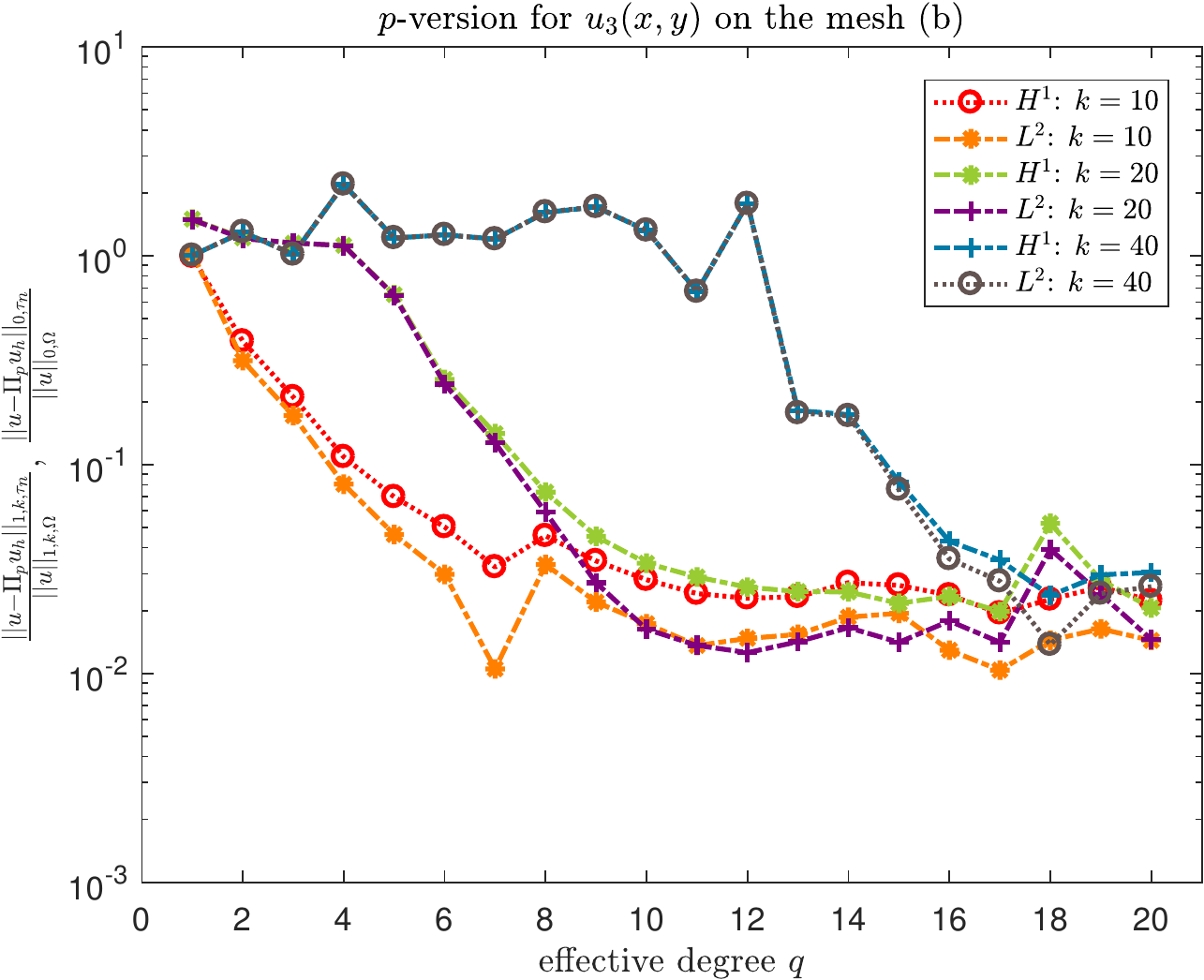}
\end{minipage}
\caption{$p$-version of the modified method for $u_3$ in \eqref{exact solutions u2 u3} on mesh (a) and (b) in Figure \ref{fig:meshes p}, from left to right.}
\label{fig:TEST7} 
\end{figure}

\subsubsection{$\h\p$-version} \label{subsubsection hp version}
We numerically investigate the $\h\p$-version of the modified nonconforming Trefftz-VEM.

By combining an $h$-refinement near solution singularities with a $p$-refinement in the elements where the solution is sufficiently smooth, exponential convergence of the errors in terms of a proper root of the number of degrees of freedom is expected.
In the framework of Trefftz methods for the Helmholtz equation, a full $\h\p$-analysis was investigated for the PWDG method~\cite{PWdG_hpversion}, where exponential convergence in terms of the square root of the number of degrees of freedom was proven.

Here, we build approximation spaces with variable number of plane wave directions element by element following the $\h\p$ approach for the Poisson problem introduced in~\cite{hpVEMcorner}.
To this end, also taking into account that one has to impose continuity elementwise in the nonconforming sense~\eqref{non conforming Sobolev}, we proceed as follows.

Let us assume that we aim at approximating the solution $u_3$ defined in \eqref{exact solutions u2 u3} on the square domain $\Omega=(0,1)^2$; such function has a singularity at $\boldsymbol{\nu} =(0,0.5)$.
We build a sequence of nested meshes that are refined towards $\boldsymbol{\nu}$. More precisely, we set $\tau_0=\{\Omega\}$. Next, for $n\in \mathbb N$,
the mesh $\taun$ is a polygonal mesh consisting of $n+1$ layers, where the $0$-th layer $L_{0,n}$ is the set of polygons abutting the singularity $\boldsymbol{\nu}$, whereas the $\ell$-th layer is defined by induction as
\[
L_{\ell,n} = \{ \E \in \taun \, : \, \overline \E \cap \overline{\E_{\ell-1}} \ne \emptyset \text{ for some } \E_{\ell-1} \in L_{\ell-1,n},\, \E \not \subset \cup_{j=0}^{\ell-1} L_j   \}.
\]
In order to achieve exponential convergence, one typically needs graded meshes towards $\boldsymbol{\nu}$. Hence, we require that the mesh size function $\hE$, for all elements $\E \in \taun$, satisfies
\begin{equation} \label{refinement cond}
\hE \approx
\begin{cases}
\mu^n & \text{if } \E \in L_{0,n},\\
\frac{1-\mu}{\mu} \text{dist} (\E,\boldsymbol{\nu}) & \text{otherwise},\\
\end{cases}
\end{equation}
where $\mu \in (0,1)$ is referred to as the \textit{grading parameter}. Moreover, we increase the dimension of the local spaces as follows.
We associate with each $\E \in \taun$ a number $\qE$, defined as
\begin{equation} \label{index polygon}
\qE = \ell+1 \quad \text{if } \E \in L_{\ell,n},\, \ell = 0,\dots,n-1,
\end{equation}
and we build the local spaces $\VE$ in \eqref{local Trefftz-VE space} by using Dirichlet/impedance traces that are edgewise in $\PWctilde(\e)$, where the space $\PWctilde(\e)$ is defined as follows.

Given $\q_{max,n} = \max_{\E \in \taun} \qE$, we first consider the set of $p_{max,n}:=2\q_{max,n} + 1$ equidistributed directions $\{\dtilde_{\ell,n}\}_{\ell=1}^{p_{max,n}}$.
On each element $\E$, we pick a set of $2\qE+1$ directions obtained by removing $2(\q_{max,n}- \qE)$ selected directions from the original set. Thus, elements abutting the singularity will have a small number of directions, which then increases linearly with the index of the layer.

In order to select such directions to be removed, we order the set $\{\dtilde_{\ell,n}\}_{\ell=1}^{p_{max,n}}$ by picking first the directions with increasing odd indices and next those with even ones, see Figure~\ref{figure renumbering directions}.
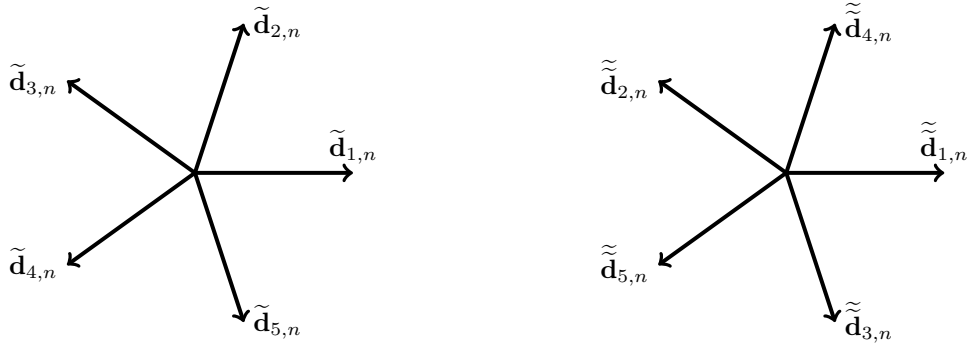
\begin{figure}[h]
\centering
\begin{minipage}{0.30\textwidth}
\begin{center}
\begin{tikzpicture}[scale=1.6]
\draw[->,line width=1.5pt] (9.5,0) -- ({9.5+1.3*cos(0)},{1.3*sin(0)});
\coordinate[label=above:$\dtilde_{1,n}$] (d1) at ({9.5+1.3*cos(0)},{1.3*sin(0)});

\draw[->,line width=1.5pt] (9.5,0) -- ({9.5+1.3*cos(72)},{1.3*sin(72)});
\coordinate[label=right:$\dtilde_{2,n}$] (d2) at ({9.5+1.3*cos(72)},{1.3*sin(72)});

\draw[->,line width=1.5pt] (9.5,0) -- ({9.5+1.3*cos(144)},{1.3*sin(144)});
\coordinate[label=left:$\dtilde_{3,n}$] (d3) at ({9.5+1.3*cos(144)},{1.3*sin(144)});

\draw[->,line width=1.5pt] (9.5,0) -- ({9.5+1.3*cos(-144)},{1.3*sin(-144)});
\coordinate[label=left:$\dtilde_{4,n}$] (d4) at  ({9.5+1.3*cos(-144)},{1.3*sin(-144)});

\draw[->,line width=1.5pt] (9.5,0) -- ({9.5+1.3*cos(-72)},{1.3*sin(-72)});
\coordinate[label=right:$\dtilde_{5,n}$] (d5) at ({9.5+1.3*cos(-72)},{1.3*sin(-72)});
\end{tikzpicture}
\end{center}
\end{minipage}
\quad\quad\quad\quad\quad\quad\quad\quad\quad
\begin{minipage}{0.30\textwidth}
\begin{center}
\begin{tikzpicture}[scale=1.6]
\draw[->,line width=1.5pt] (9.5,0) -- ({9.5+1.3*cos(0)},{1.3*sin(0)});
\coordinate[label=above:$\widetilde{\dtilde}_{1,n}$] (d1) at ({9.5+1.3*cos(0)},{1.3*sin(0)});

\draw[->,line width=1.5pt] (9.5,0) -- ({9.5+1.3*cos(72)},{1.3*sin(72)});
\coordinate[label=right:$\widetilde{\dtilde}_{4,n}$] (d4) at ({9.5+1.3*cos(72)},{1.3*sin(72)});

\draw[->,line width=1.5pt] (9.5,0) -- ({9.5+1.3*cos(144)},{1.3*sin(144)});
\coordinate[label=left:$\widetilde{\dtilde}_{2,n}$] (d2) at ({9.5+1.3*cos(144)},{1.3*sin(144)});

\draw[->,line width=1.5pt] (9.5,0) -- ({9.5+1.3*cos(-144)},{1.3*sin(-144)});
\coordinate[label=left:$\widetilde{\dtilde}_{5,n}$] (d5) at  ({9.5+1.3*cos(-144)},{1.3*sin(-144)});

\draw[->,line width=1.5pt] (9.5,0) -- ({9.5+1.3*cos(-72)},{1.3*sin(-72)});
\coordinate[label=right:$\widetilde{\dtilde}_{3,n}$] (d3) at ({9.5+1.3*cos(-72)},{1.3*sin(-72)});
\end{tikzpicture}
\end{center}
\end{minipage}
\caption{\textit{Left}: equidistributed set of directions $\{\dtilde_{\ell,n}\}_{\ell=1}^{p_{max,n}}$. \textit{Right}: reordered set of directions $\{\dtildetilde_{\ell,n}\}_{\ell=1}^{p_{max,n}}$.
Firstly, one considers the directions with odd index and next those with even index.} \label{figure renumbering directions}
\end{figure}
At this point, given the reordered set of directions $\{\dtildetilde_{\ell,n}\}_{\ell=1}^{p_{max,n}}$, we remove the $2(\q_{max,n} - \qE)$ directions having the largest indices.
This procedure allows to build elementwise nested sets of directions with different cardinality.

\medskip
We are now able to define nested spaces over each edge $\e$ of the mesh skeleton by fixing spaces of plane waves whose number of basis elements is given by the maximum of the local numbers $\qE$ in \eqref{index polygon} of the neighbouring elements:
\[
\PWctilde(\e) :=
\begin{cases}
\text{span} \left\{ e^{\im \k \dtildetilde_\ell (\x-\xe)}{}_{|_\e} \, : \, \ell=1,\dots, 2\max(\q_{\E_1},\q_{\E_2})+1   \right\} &\quad \text{if } \e \subset \EnI,\, \e \subseteq \partial \E_1 \cap \partial \E_2\\
\text{span} \left\{ e^{\im \k \dtildetilde_\ell (\x-\xe)}{}_{|_\e} \, : \, \ell=1,\dots, 2\qE+1   \right\} &\quad \text{if } \e \subset \Enb,\, \e \subseteq \partial \E,\\
\end{cases}
\]
where $K_1$ and $K_2$, and $K$, denote the elements abutting edge $\e$, if $e$ is an interior edge and a boundary edge, respectively. 
This resembles the so-called \emph{maximum rule} employed in $\h\p$-VEM \cite{hpVEMcorner}. 

A sequence of meshes satisfying the geometric refinement condition \eqref{refinement cond} towards $\boldsymbol{\nu}$, along with the distribution of effective degrees accordingly with \eqref{index polygon}, is depicted in Figure \ref{figure graded mesh}.
\begin{figure}[h]
\centering
\begin{minipage}{0.30\textwidth}
\begin{center}
\begin{tikzpicture}[scale=3.5]
\draw[black, very thick, -] (0,0) -- (1,0) -- (1,1) -- (0,1) -- (0,0); \draw(0.5, 0.5) node[black] {$1$};
\draw[red,fill=red] (0,0.5) circle (.2ex); \draw(-.1, 0.5) node[black] {$\boldsymbol{\nu}$};
\end{tikzpicture}
\end{center}
\end{minipage}
\quad
\begin{minipage}{0.30\textwidth}
\begin{center}
\begin{tikzpicture}[scale=3.5]
\draw[black, very thick, -] (0,0) -- (1,0) -- (1,1) -- (0,1) -- (0,0);
\draw[black, very thick, -] (1/3,0) -- (1/3,1); \draw[black, very thick, -] (0,1/3) -- (1, 1/3); \draw[black, very thick, -] (0,2/3) -- (1, 2/3);
\draw(1/6, 1/2) node[black] {$1$}; \draw(1/6, 0.85) node[black] {$2$};  \draw(1/6, 0.15) node[black] {$2$}; 
\draw(2/3, 1/2) node[black] {$2$}; \draw(2/3, 0.85) node[black] {$2$};  \draw(2/3, 0.15) node[black] {$2$}; 
\draw[red,fill=red] (0,0.5) circle (.2ex); \draw(-.1, 0.5) node[black] {$\boldsymbol{\nu}$};
\end{tikzpicture}
\end{center}
\end{minipage}
\quad
\begin{minipage}{0.30\textwidth}
\begin{center}
\begin{tikzpicture}[scale=3.5]
\draw[black, very thick, -] (0,0) -- (1,0) -- (1,1) -- (0,1) -- (0,0);
\draw[black, very thick, -] (1/3,0) -- (1/3,1); \draw[black, very thick, -] (0,1/3) -- (1, 1/3); \draw[black, very thick, -] (0,2/3) -- (1, 2/3);
\draw[black, very thick, -] (0,4/9) -- (1/3,4/9); \draw[black, very thick, -] (0,5/9) -- (1/3,5/9);  \draw[black, very thick, -] (1/9,1/3) -- (1/9, 2/3);
\draw(1/18, 0.61) node[black] {$2$};  \draw(1/18, 0.5) node[black] {$1$}; \draw(1/18, 0.38) node[black] {$2$}; 
\draw(2/9, 0.61) node[black] {$2$};  \draw(2/9, 0.5) node[black] {$2$}; \draw(2/9, 0.38) node[black] {$2$}; 
\draw(1/6, 0.85) node[black] {$3$};  \draw(1/6, 0.15) node[black] {$3$}; 
\draw(2/3, 1/2) node[black] {$3$}; \draw(2/3, 0.85) node[black] {$3$};  \draw(2/3, 0.15) node[black] {$3$}; 
\draw[red,fill=red] (0,0.5) circle (.2ex); \draw(-.1, 0.5) node[black] {$\boldsymbol{\nu}$};
\end{tikzpicture}
\end{center}
\end{minipage}
\caption{$\tau_0$ (\textit{left}), $\tau_1$ (\textit{center}), and $\tau_2$ (\textit{right}) of a sequence $\{\taun\}_n$ of meshes graded toward $\boldsymbol{\nu}$ with grading parameter $\mu=1/3$. The numbers inside the elements denote the effective degrees accordingly with \eqref{index polygon}.}
\label{figure graded mesh}
\end{figure}
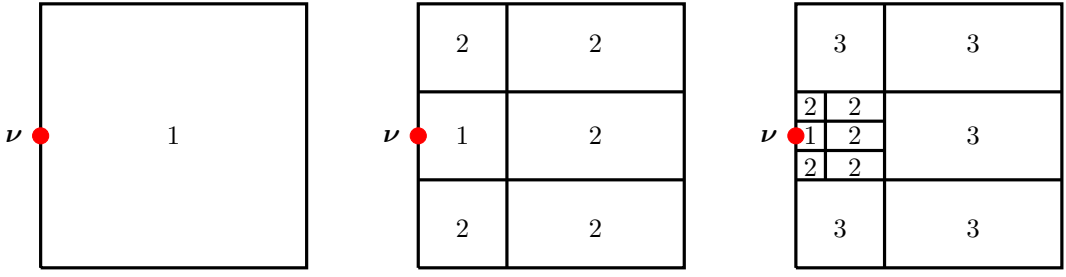

In Figure \ref{figure hp}, we present the decay of the approximate $L^2$ errors in \eqref{rel_errors} in terms of the quadratic root of the number of degrees of freedom.
Hereby, we employ the modified version of the method presented in Section \ref{subsection orthogonalized version}.
Further, we select as wave numbers $\k=10$, $20$, and $40$, and as grading parameters $\mu=0.5$ and $\mu=1/3$, see \eqref{refinement cond}.
\begin{figure}[h]
\begin{minipage}{0.49\textwidth} 
\includegraphics[width=\textwidth]{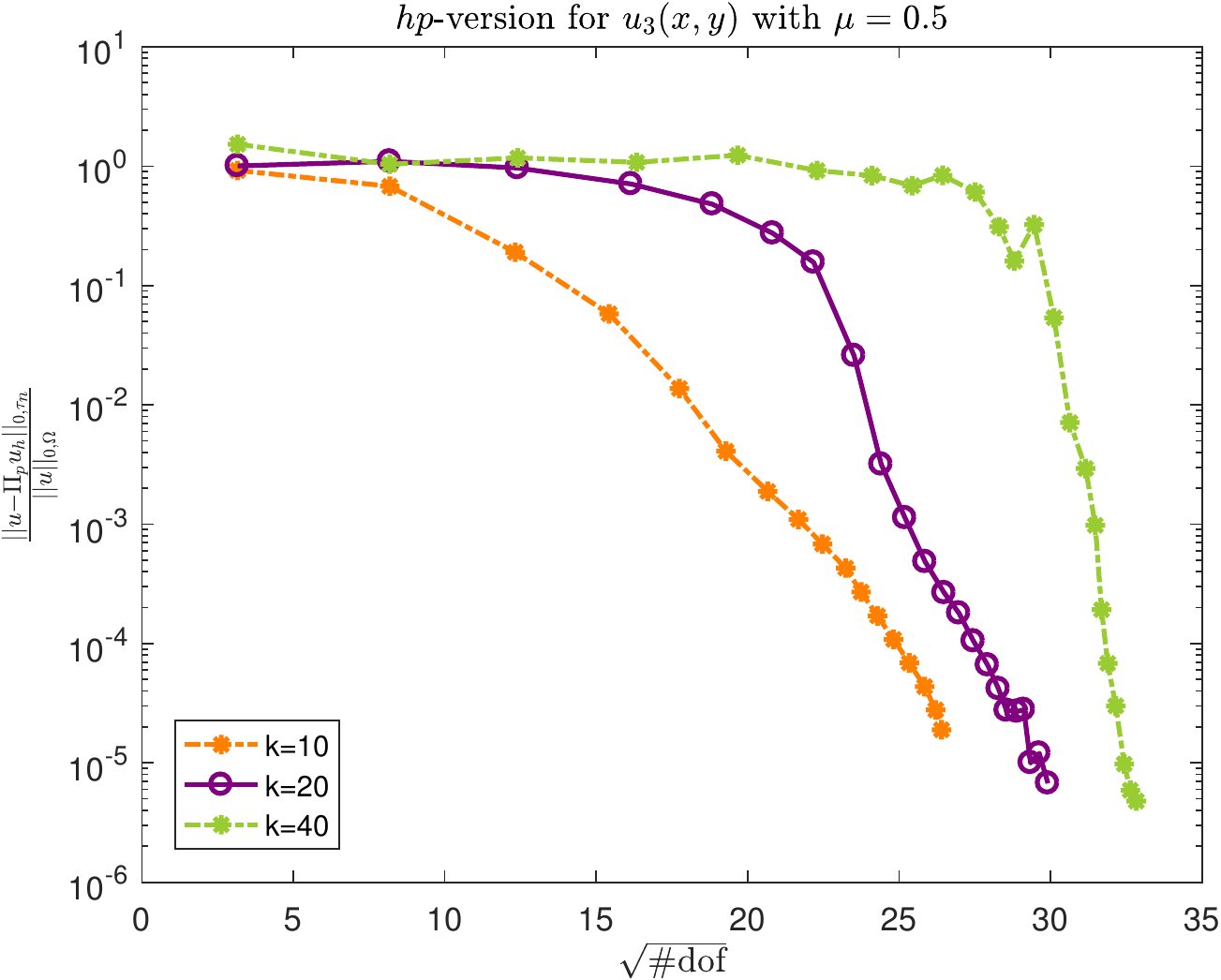}
\end{minipage}
\hspace{0.25cm}
\begin{minipage}{0.49\textwidth}
\includegraphics[width=\textwidth]{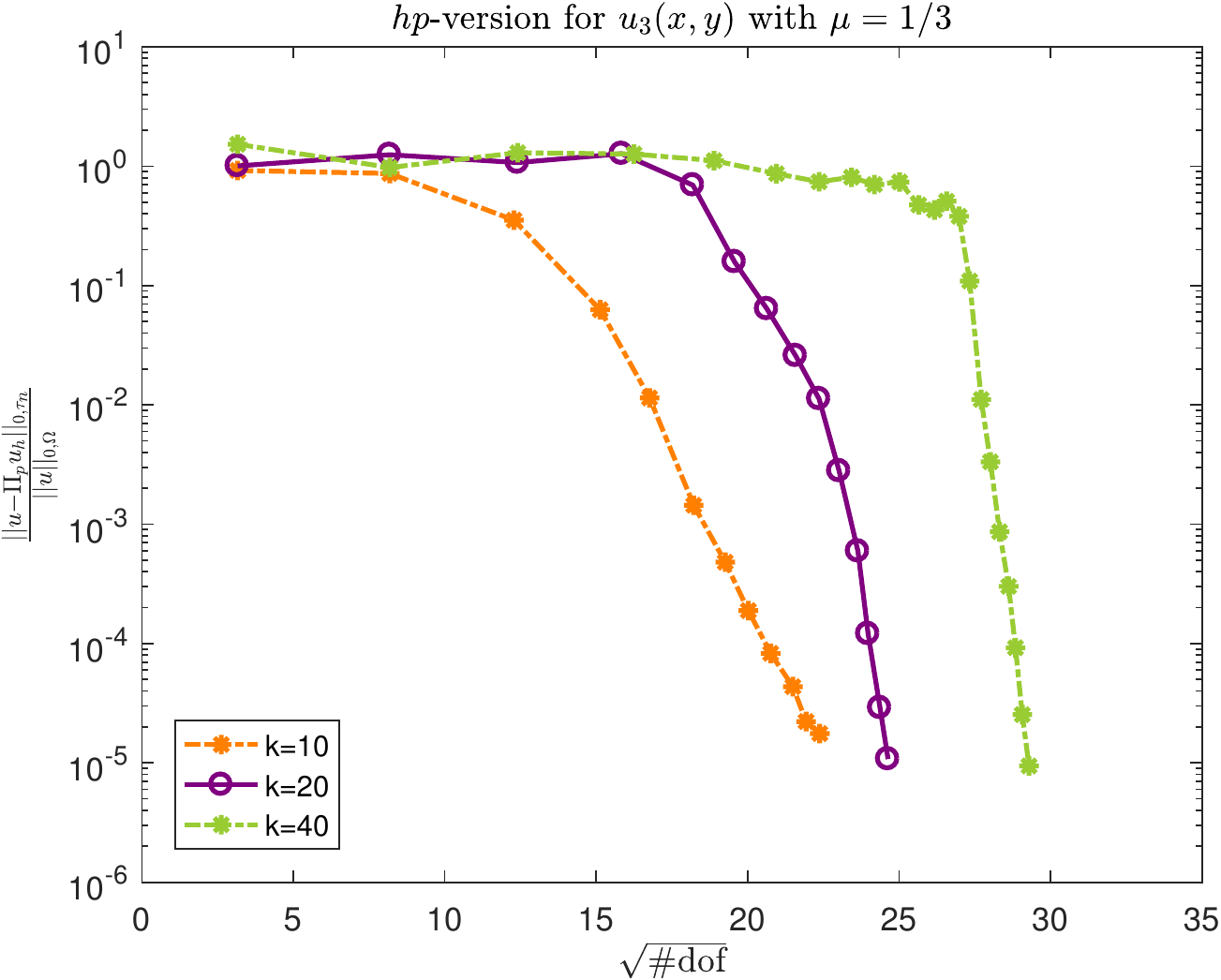}
\end{minipage}
\hspace{0.25cm}
\caption{$\h\p$-version of the modified method on the test case $u_3$ in \eqref{exact solutions u2 u3}, by employing graded meshes as those in Figure \ref{figure graded mesh} with wave numbers $\k=10$, $20$, and $40$, and grading parameters $\mu=0.5$ (\textit{left}) and $\mu=1/3$ (\textit{right}).
The distribution of the effective plane wave degree indices is as in \eqref{index polygon}. In both plots, the approximate $L^2$ error \eqref{rel_errors} is plotted against the quadratic root of the number of degrees of freedom.}
\label{figure hp} 
\end{figure}

We observe a decay of the error which is exponential in terms of the square root of the degrees of freedom
instead of the cubic root as for standard $\h\p$-FEM~\cite{SchwabpandhpFEM} and $\h\p$-VEM~\cite{hpVEMcorner}.
This is typical of the Trefftz setting, see~\cite{PWdG_hpversion, hmps_harmonicpolynomialsapproximationandTrefftzhpdgFEM} in the DG framework and~\cite{conformingHarmonicVEM, ncHVEM} in the VEM framework.

Moreover, we want to highlight that after the pre-asymptotic regime, the relative errors decay extremely rapidly in terms of the number of degrees of freedom. This can be explained by the fact that, for smaller mesh sizes, more and more redundant plane wave directions are removed by the filtering process, compensating the increase in the number of edges. The ``paradox'' here is that via the second filtering process in Algorithm \ref{algorithm orthog process}, the errors of the method decrease exponentially, while the number of degrees of freedom
seems to increase extremely slowly, especially in presence of high wave number.

\subsubsection{Comparison of the modified nonconforming Trefftz-VEM with the PWVEM} \label{section comp Trefftz PWVEM}
In this section, we compare the approximate relative $L^2$ errors in \eqref{rel_errors} of the modified nonconforming Trefftz-VEM with those of the PWVEM \cite{Helmholtz-VEM}. Note that the definition of $\Pip$ is the same for both methods. For the PWVEM, we took the stabilization proposed in \cite{Helmholtz-VEM}.

We consider a boundary value problem of the form \eqref{weak continuous problem} with $\Omega:=(0,1)^2$ and $\GammaR=\partial \Omega$, and exact solution $u_2$ in \eqref{exact solutions u2 u3}. 

\paragraph*{$\h$-version:}

To start with, we compare the $h$-versions of the methods in terms of the number of degrees of freedom when employing Voronoi meshes.
As a first test, we choose $q=6$ (which corresponds to $p=13$) and $k=20$, $40$, and $60$. Then, as a second test, we fix instead $k=20$ and choose $q=5$, $7$, and $9$. The results are shown in Figure \ref{fig:TEST8}.

\begin{figure}[h]
\begin{center}
\begin{minipage}{0.45\textwidth} 
\includegraphics[width=\textwidth]{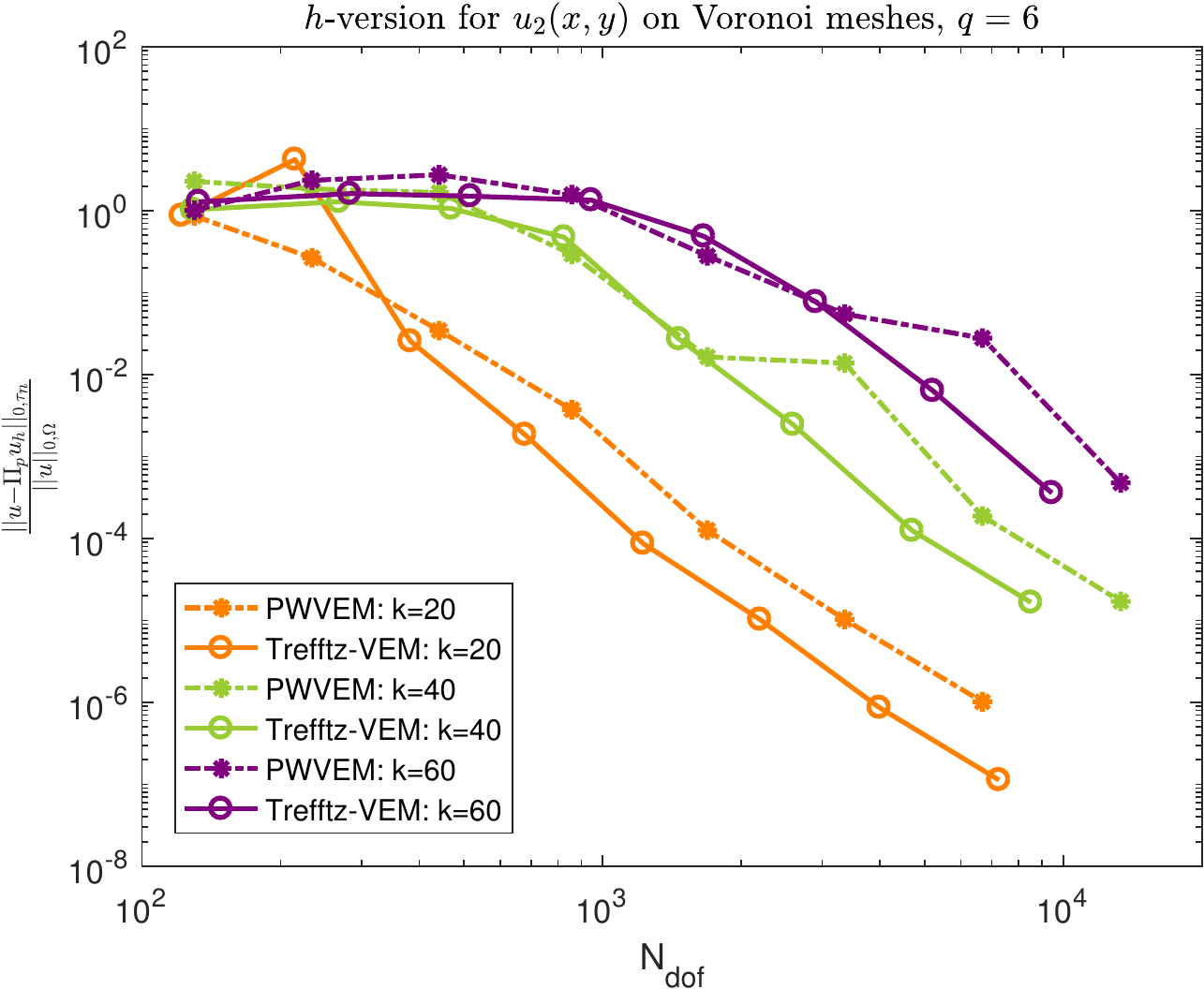}
\end{minipage}
\hspace{0.5cm}
\begin{minipage}{0.45\textwidth}
\includegraphics[width=\textwidth]{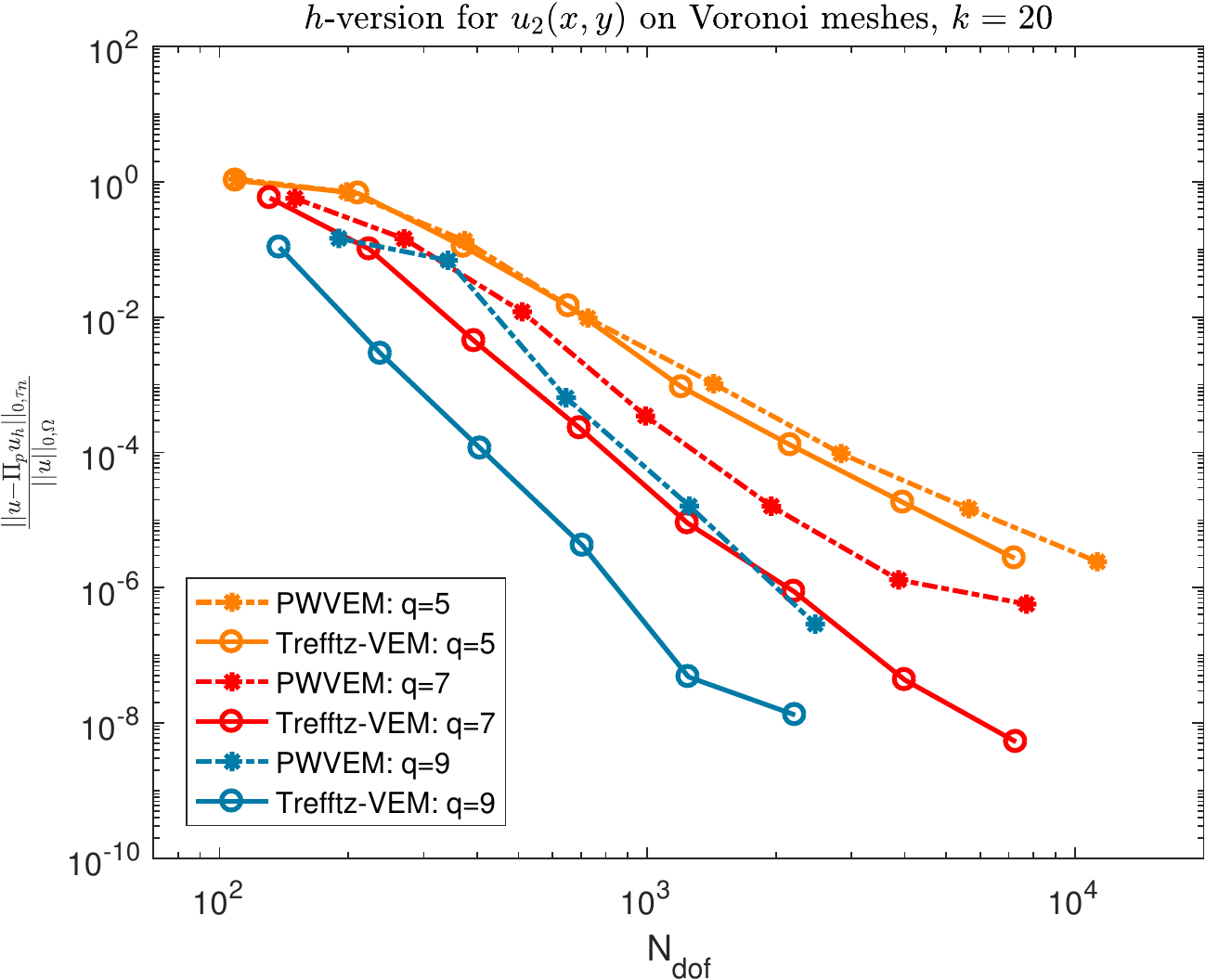}
\end{minipage}
\end{center}
\caption{Comparison of the $h$-version of the modified nonconforming Trefftz-VEM with the PWVEM for $u_2$ in \eqref{exact solutions u2 u3} on Voronoi meshes. \textit{Left}: fixed $q=6$, different values of $k=20$, $40$, and $60$. \textit{Right}: fixed $k=20$, different values of $\q=5$, $7$, and $9$.}
\label{fig:TEST8} 
\end{figure}

In all the cases, the approximate relative $L^2$ errors are smaller when using the nonconforming Trefftz-VEM. This can be traced back to the filtering process applied in the Trefftz version. 

\paragraph*{$p$-version:} For the $p$-versions, we compare the two methods with $k=20$ and $40$ for the exact solution $u_2$ in \eqref{exact solutions u2 u3} on a Voronoi mesh and a nonconvex polygonal mesh made of 16 and 100 elements, respectively.
These meshes are depicted in Figure~\ref{fig:meshes2}. In Figure~\ref{fig:TEST9}, the approximate relative $L^2$ errors are plotted.

\begin{figure}[h]
\begin{center}
\begin{minipage}{0.4\textwidth} 
\includegraphics[width=\textwidth]{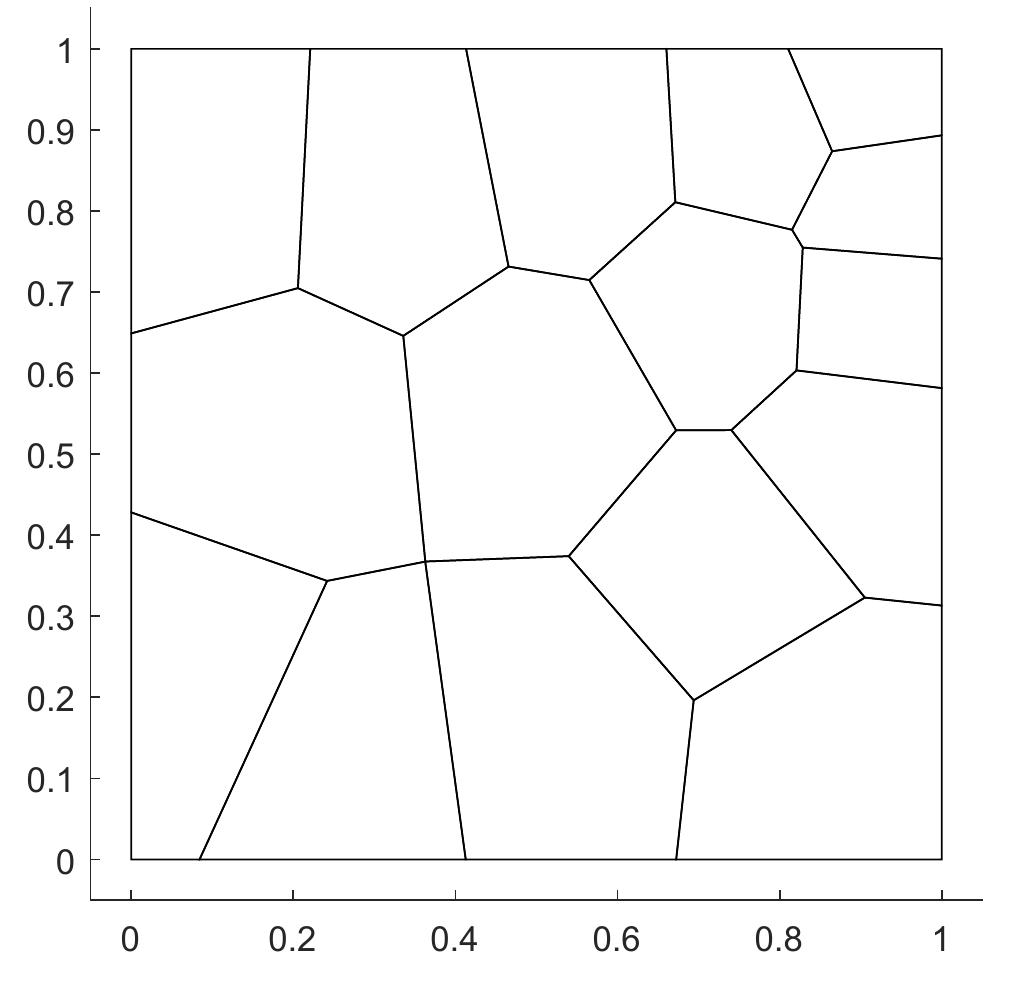}
\end{minipage}
\hspace{1cm}
\begin{minipage}{0.4\textwidth}
\includegraphics[width=\textwidth]{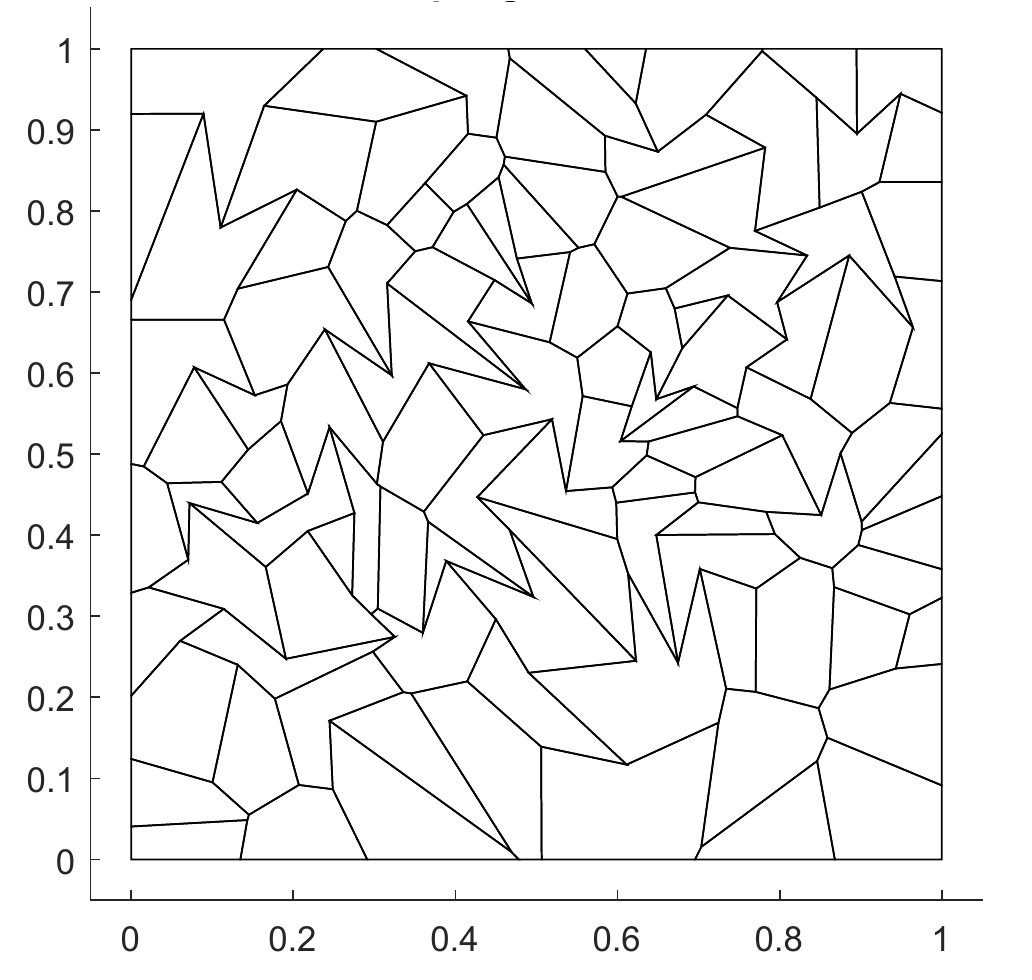}
\end{minipage}
\end{center}
\caption{Voronoi mesh with 16 elements (\textit{left}) and a polygonal mesh made of 100 (also nonconvex) elements (\textit{right}).}
\label{fig:meshes2} 
\end{figure}

\begin{figure}[h]
\begin{center}
\begin{minipage}{0.45\textwidth} 
\includegraphics[width=\textwidth]{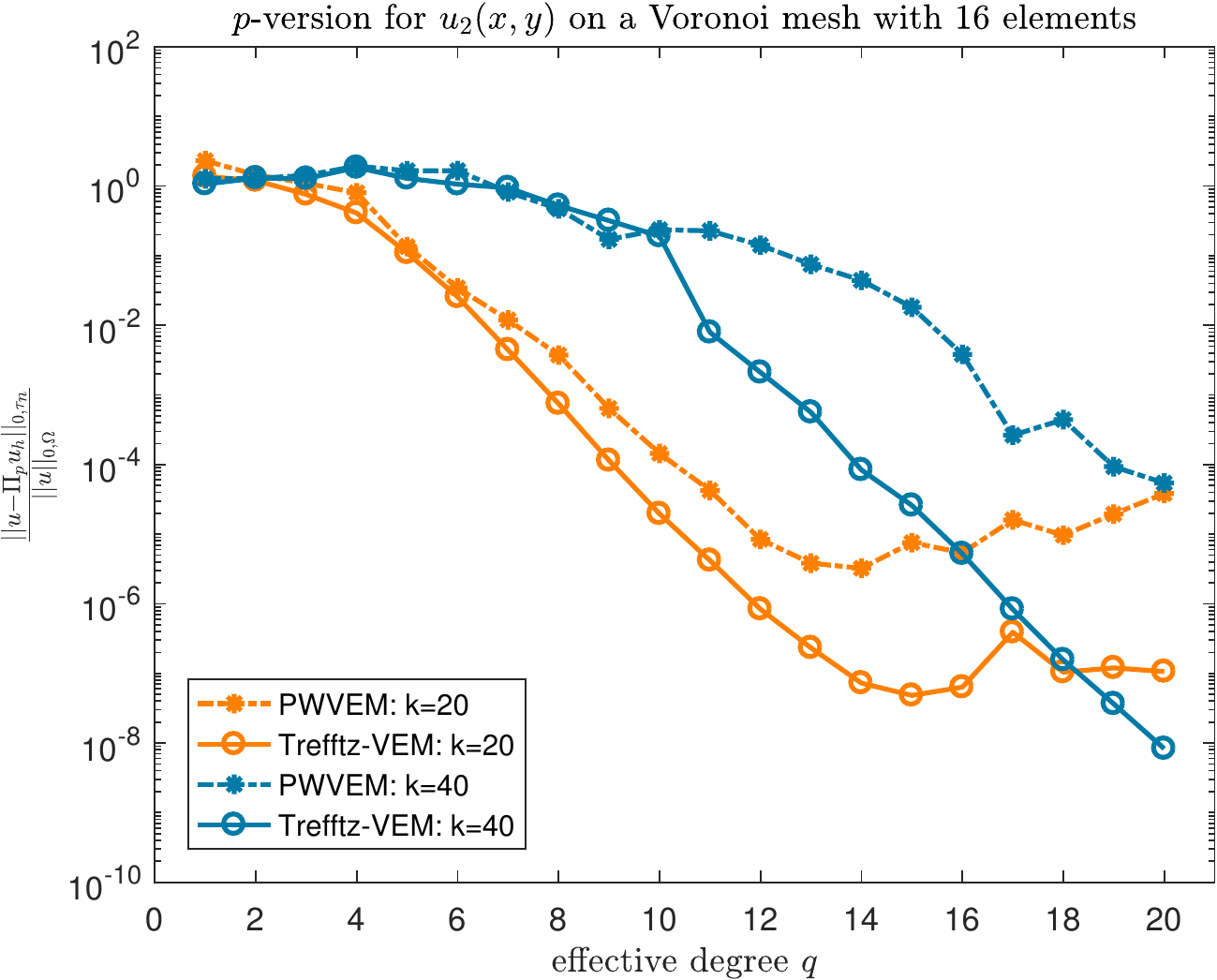}
\end{minipage}
\hspace{0.5cm}
\begin{minipage}{0.45\textwidth}
\includegraphics[width=\textwidth]{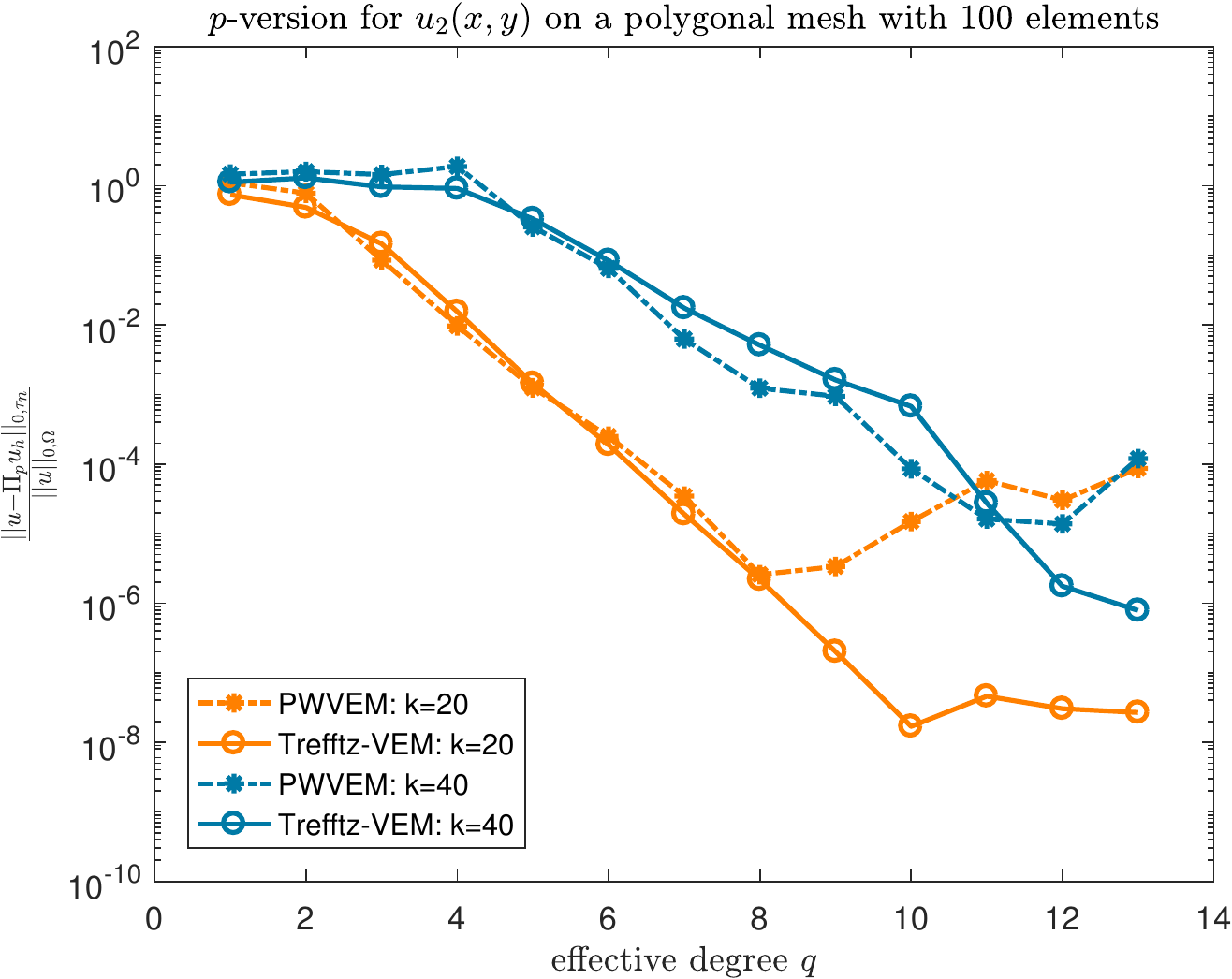}
\end{minipage}
\end{center}
\caption{Comparison of the $p$-version of the modified nonconforming Trefftz-VEM with the PWVEM for $u_2$ in \eqref{exact solutions u2 u3} and $k=20$ and $40$ on the Voronoi mesh with 16 elements (\textit{left}) and  the polygonal mesh with 100 (also nonconvex) elements (\textit{right}).}
\label{fig:TEST9} 
\end{figure}

Also in this case, the modified nonconforming Trefftz-VEM leads to better results, and allows to reach a higher accuracy before instability takes place. In particular, the method seems to be robust even in terms of the geometry of the mesh elements.

Finally, we compare the $\p$-version of the two methods on a structured triangular and a Voronoi mesh with 32 elements each, when using the solution $u_3$ given in \eqref{exact solutions u2 u3}, and $\k=10$ and $20$. This is portrayed in Figure \ref{fig:TEST10}.
\begin{figure}[h]
\begin{center}
\begin{minipage}{0.45\textwidth} 
\includegraphics[width=\textwidth]{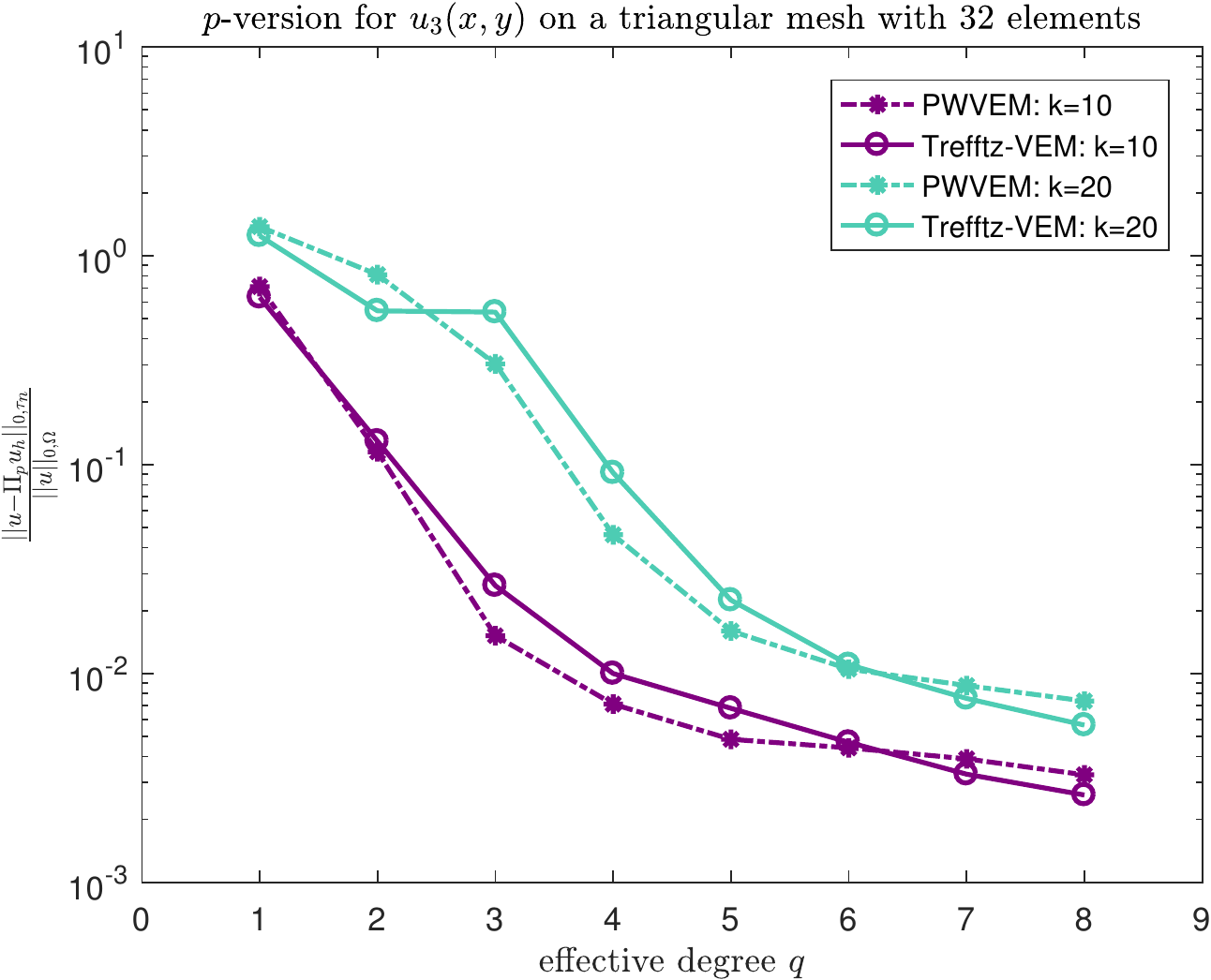}
\end{minipage}
\hspace{0.5cm}
\begin{minipage}{0.45\textwidth}
\includegraphics[width=\textwidth]{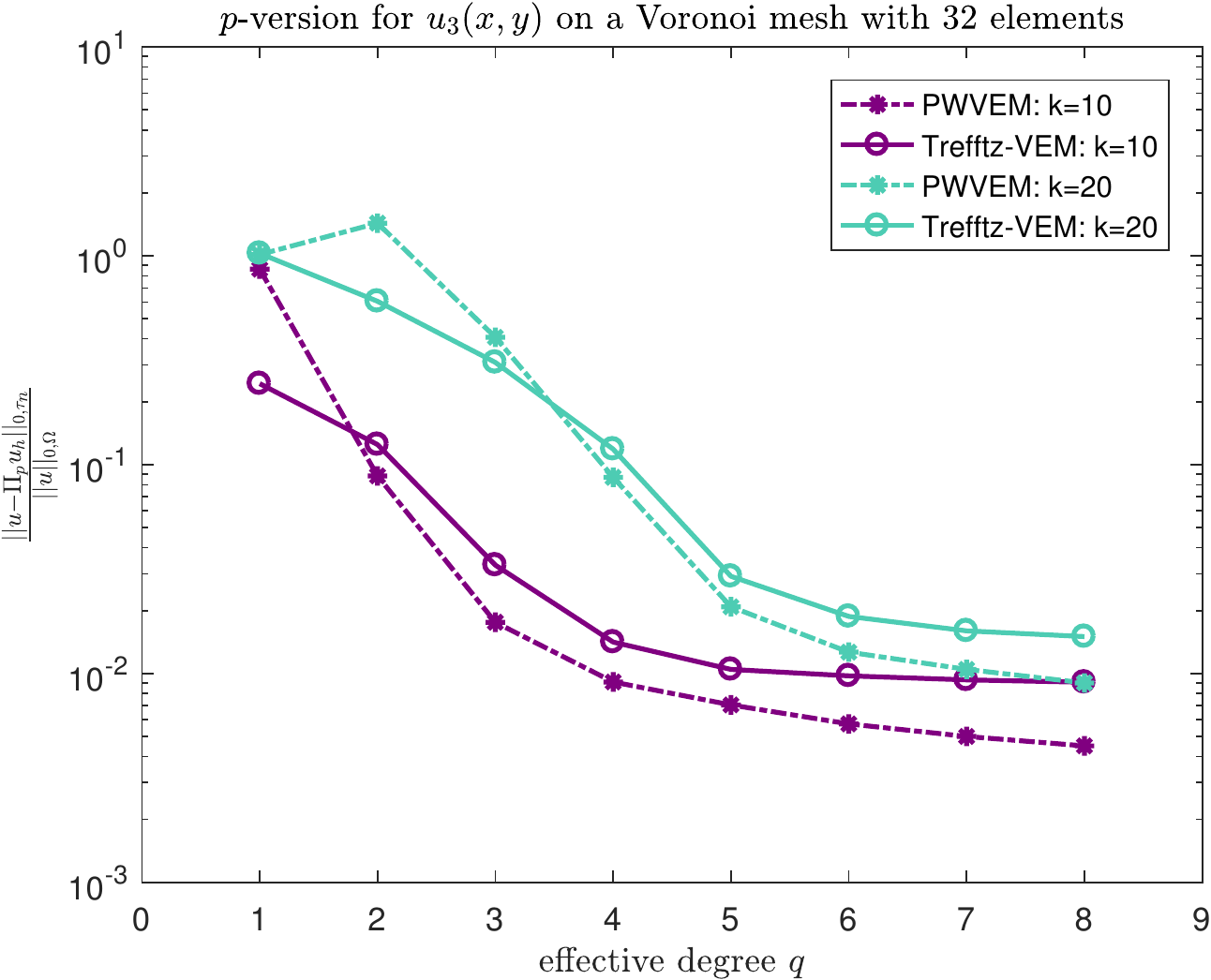}
\end{minipage}
\end{center}
\caption{Comparison of the $p$-version of the modified nonconforming Trefftz VE method with the PWVEM for $u_3$ in \eqref{exact solutions u2 u3}, $\k=10$ and $20$, on a triangular mesh and a Voronoi mesh made of 32 elements each.}
\label{fig:TEST10} 
\end{figure}
In both cases, the convergence rate stagnates after few refinement steps. This is however not surprising, owing to the fact that the solution $u_3$ in~\eqref{exact solutions u2 u3} has a low Sobolev regularity.

\subsubsection{Comparison of the modified nonconforming Trefftz-VEM with the PWDG} \label{subsection PWDG}
In this section, we compare the approximate relative $L^2$ errors of the modified nonconforming Trefftz-VEM with those of the PWDG. For the latter, we choose the penalty parameters of the ultra weak formulation in~\cite{cessenatdespres_basic}.
For all the tests, we employ sequences of Cartesian meshes. 

\paragraph*{$h$-version:}
First, we compare the $\h$-versions of the two methods for a boundary value problem of the form \eqref{weak continuous problem} on $\Omega:=(0,1)^2$ with exact solution $u_2$ given in \eqref{exact solutions u2 u3}, $\theta=1$, and $\GammaR=\partial \Omega$.
The results for fixed $\q=6$ and $\k=20$, $40$, and $60$, and fixed $\k=20$ and $\q=5$, $7$, and $9$, are reported in Figure \ref{fig:TEST11}.
\begin{figure}[h]
\begin{center}
\begin{minipage}{0.45\textwidth} 
\includegraphics[width=1.03\textwidth]{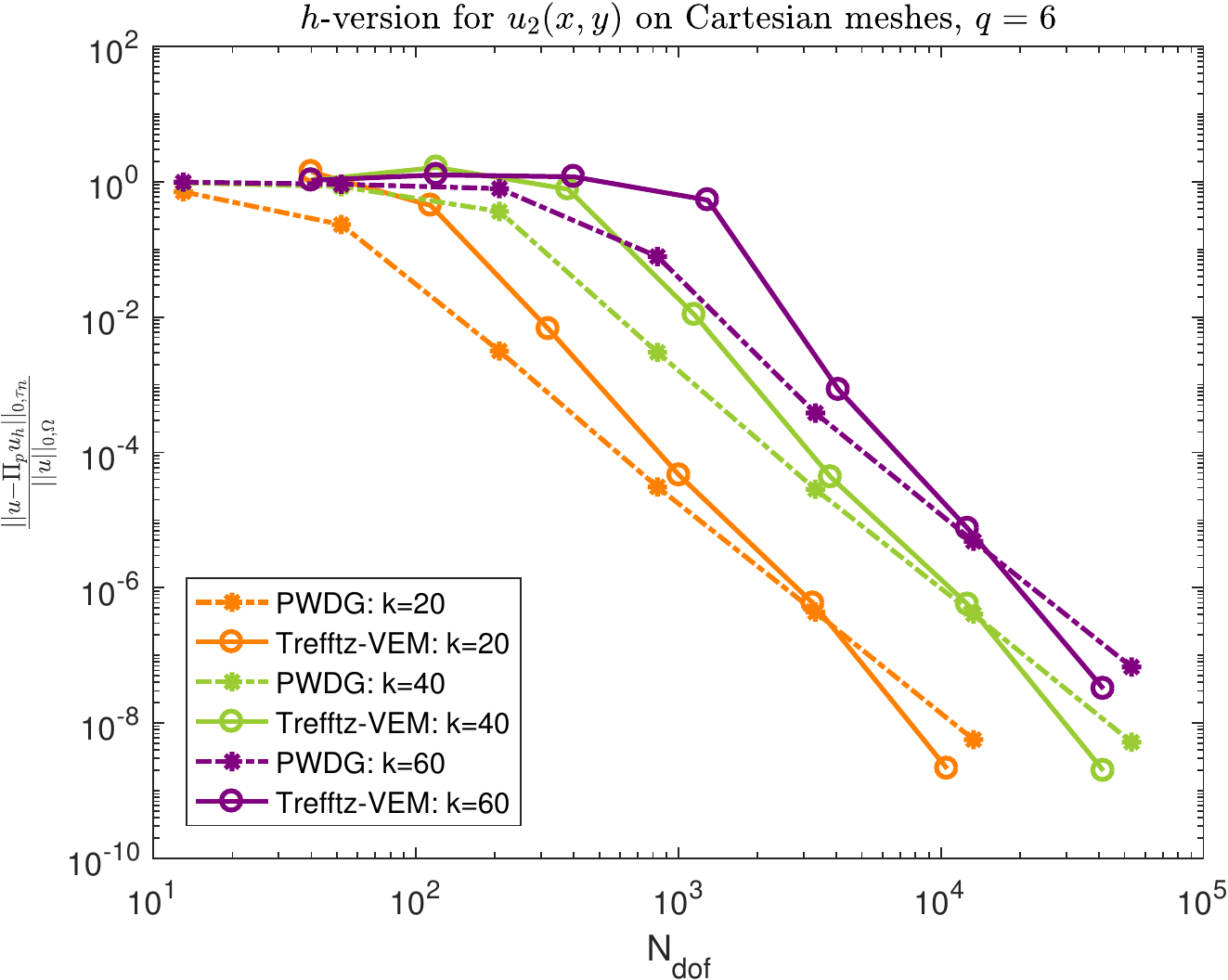}
\end{minipage}
\hspace{0.5cm}
\begin{minipage}{0.45\textwidth}
\includegraphics[width=\textwidth]{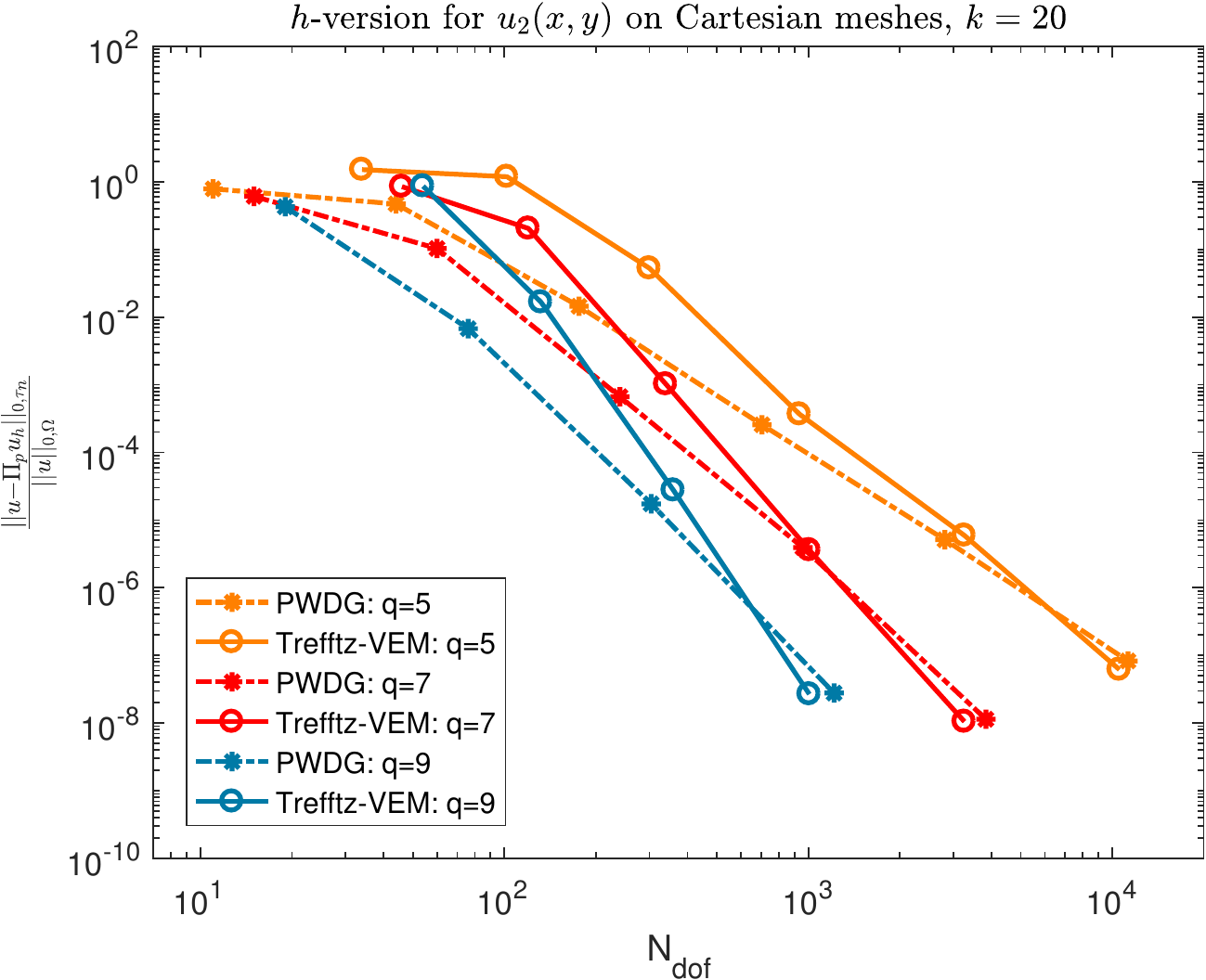}
\end{minipage}
\end{center}
\caption{Comparison of the $h$-version of the modified nonconforming Trefftz-VEM with the PWDG for $u_2$ in~\eqref{exact solutions u2 u3} on Cartesian meshes. \textit{Left}: fixed $q=6$, different values of $\k=20$, $40$, and $60$. \textit{Right}: fixed $k=20$, different values of $\q=5$, $7$, and $9$.}
\label{fig:TEST11} 
\end{figure}

It can be noticed that, with both methods, we can approximately reach the same accuracy. For the nonconforming Trefftz-VEM, the pre-asymptotic regime is broader, followed however by a ``steeper" slope of the convergence rate.
This broader pre-asymptotic area can be explained by the fact that, on coarse meshes, the removing procedure of Algorithm \ref{algorithm orthog process} is almost not performed, and thus more degrees of freedom
than in PWDG are employed, whereas for fine meshes, the removing procedure has a huge impact, see Tables~\ref{tab:TEST4_D_cart} and~\ref{tab:TEST4_voro_D}.

Secondly, we perform the same tests as before, considering now as exact solution the function $u_3$ in \eqref{exact solutions u2 u3} instead of $u_2$, see Figure \ref{fig:TEST12}. We observe a similar behaviour as for the results in Figure~\ref{fig:TEST10}.

\begin{figure}[h]
\begin{center}
\begin{minipage}{0.45\textwidth} 
\includegraphics[width=\textwidth]{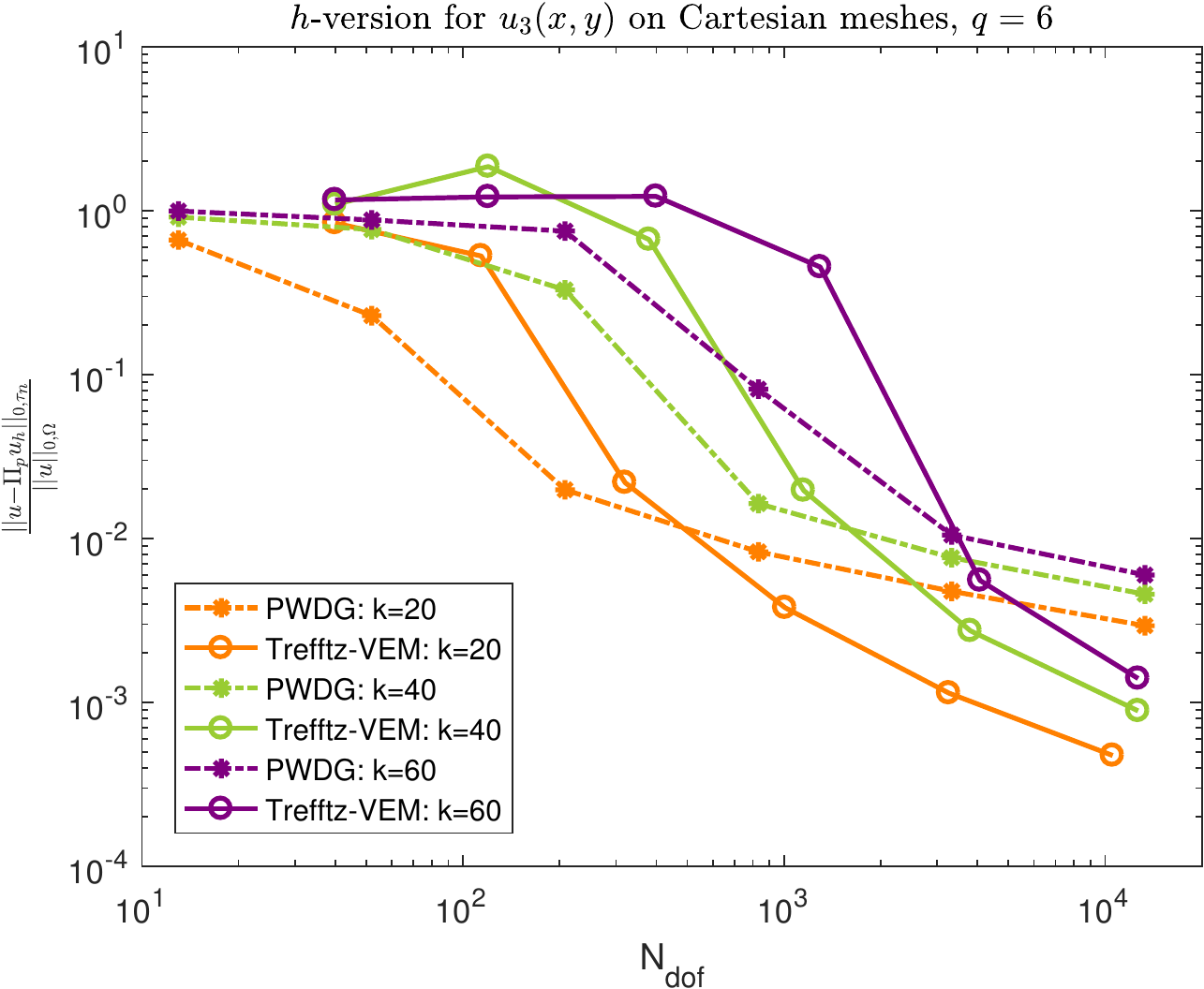}
\end{minipage}
\hspace{0.5cm}
\begin{minipage}{0.45\textwidth}
\includegraphics[width=\textwidth]{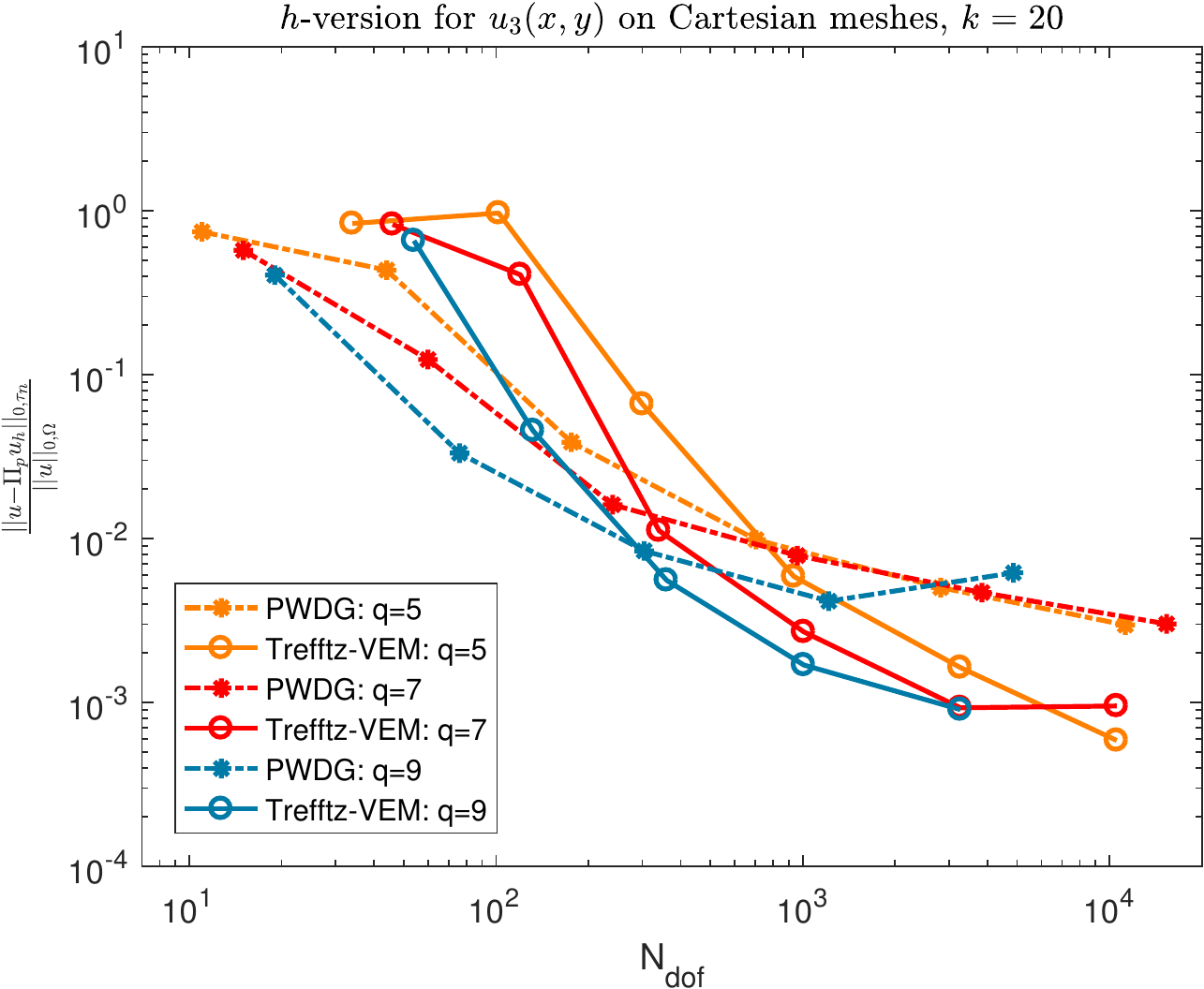}
\end{minipage}
\end{center}
\caption{Comparison of the $h$-version of the modified nonconforming Trefftz VE method with the PWDG for $u_3$ in~\eqref{exact solutions u2 u3} on Cartesian meshes. \textit{Left}: fixed $\q=6$, different values of $\k=20$, $40$, and $60$. \textit{Right}: fixed $\k=20$, different values of $\q=5$, $7$, and $9$.}
\label{fig:TEST12} 
\end{figure}

\paragraph*{$\p$-version:} 

Concerning the $\p$-version, we compare the approximate relative $L^2$ bulk errors on a Cartesian mesh made of 16 elements with exact solution given by $u_2$ in \eqref{exact solutions u2 u3} and $\k=20$, $40$ and $60$, and exact solution $u_3$ in \eqref{exact solutions u2 u3} and $\k=10$ and $20$, respectively.
The numerical results are displayed in Figure \ref{fig:TEST13}.
\begin{figure}[h]
\begin{center}
\begin{minipage}{0.45\textwidth} 
\includegraphics[width=\textwidth]{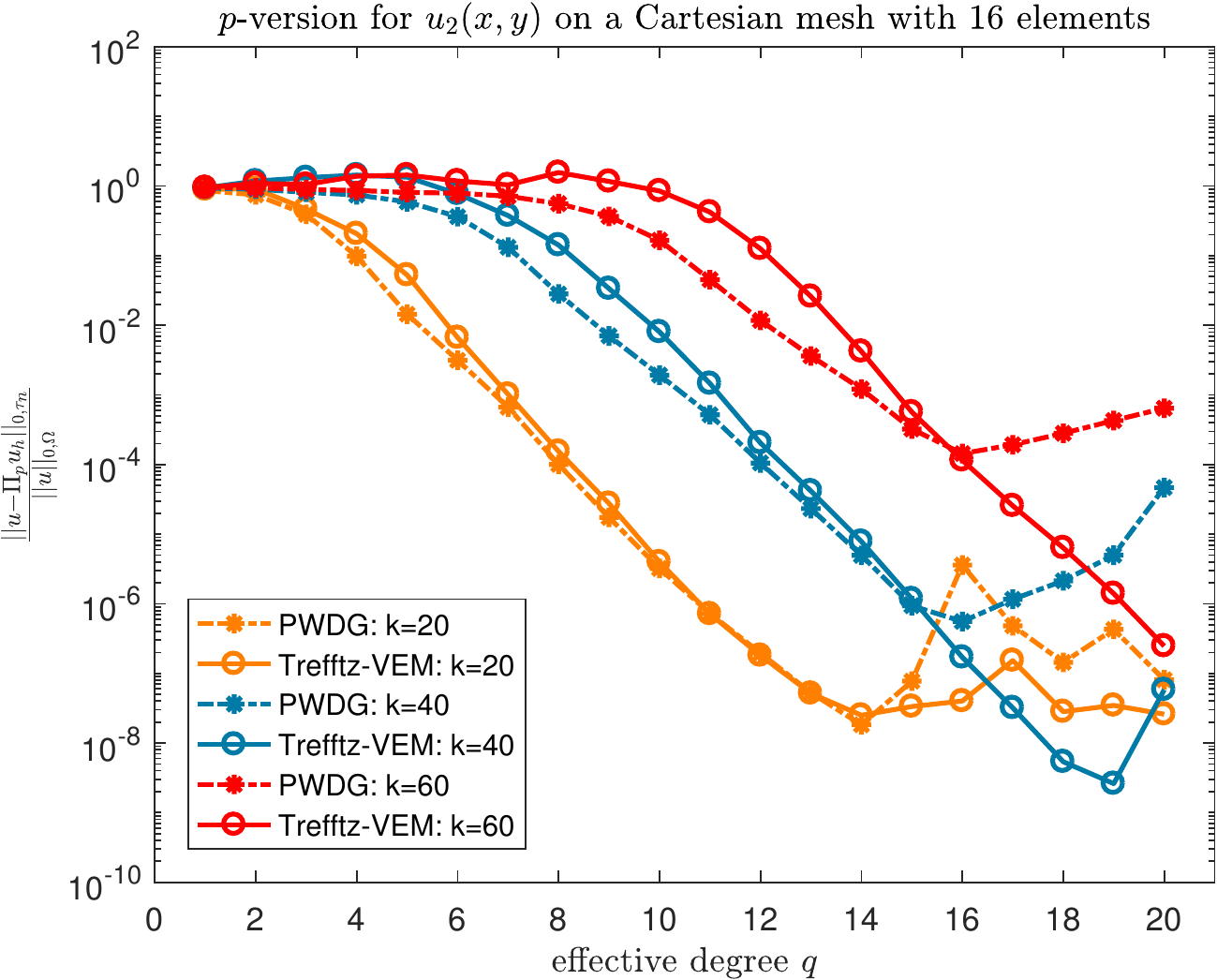}
\end{minipage}
\hspace{0.5cm}
\begin{minipage}{0.45\textwidth}
\includegraphics[width=\textwidth]{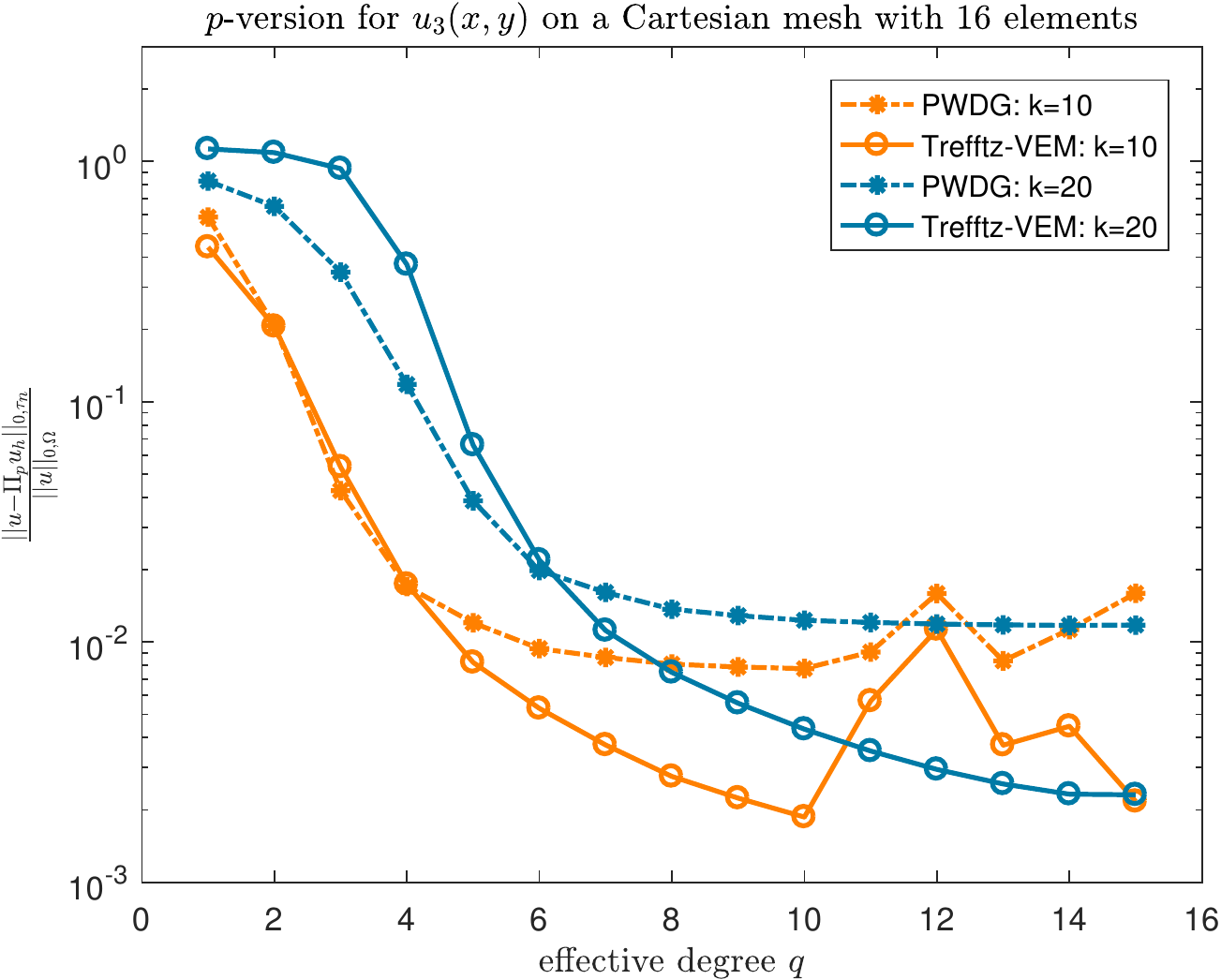}
\end{minipage}
\end{center}
\caption{Comparison of the $p$-version of the modified nonconforming Trefftz-VEM with the PWDG for $u_2$ in \eqref{exact solutions u2 u3}, $\k=20$, $40$ and $60$ (\textit{left}),
and for $u_3$ in \eqref{exact solutions u2 u3}, $\k=10$ and $20$ (\textit{right}) on a Cartesian mesh made of 16 elements.}
\label{fig:TEST13} 
\end{figure}
Very interestingly, the $\p$-version of the modified nonconforming Trefftz-VEM seems to lead to more robust performance than PWDG, especially for higher wave numbers.

\section{Conclusions} \label{section conclusions}
In this paper, we extended the nonconforming Trefftz-VEM in \cite{ncTVEM_theory} to Helmholtz boundary value problems endowed with mixed boundary conditions.
We presented a series of numerical experiments showing that the original version severely suffers of ill-conditioning, making the method practically unreliable.

In order to mitigate the lack of robustness due to the ill-conditioning related to the choice of plane wave basis functions in the design of the method,
we built up a numerical recipe based on the orthonormalization of the edge plane wave basis functions via an eigendecomposition of the associated edge mass matrices.
We highly exploited the fact that, using the nonconforming setting, it is possible to modify the basis functions edgewise without affecting their behavior on the other edges,
which could also be very appealing in regard of an extension of the method to the 3D case and to nonconforming methods for other problems.
Using the above-mentioned strategy, the modified nonconforming Trefftz-VEM becomes numerically more stable. Such a recipe also allows for a significant reduction of the number of degrees of freedom.

Numerical experiments confirm the convergence rates derived in \cite{ncTVEM_theory}. Moreover, the $p$- and $hp$-versions of this new method were discussed. We have seen that the modified version of the nonconforming Trefftz-VEM
provides in many cases better performance than other plane wave methods for the approximation of solutions to Helmholtz boundary value problems, especially in the case of both high wave numbers and effective degrees.

Extensions of the approach herein presented to the case of piecewise constant wave number are under investigation.

\section*{Acknowledgements}
The authors have been funded by the Austrian Science Fund (FWF) through the project F 65 (L.M. and I.P.) and the project P 29197-N32 (I.P. and A.P.), and by the Vienna Science and Technology Fund (WWTF) through the project MA14-006 (I.P.).\\

{\footnotesize
\bibliography{bibliogr}

\begin{thebibliography}{10}

\bibitem{AbramowitzStegun_handbook}
M.~Abramowitz and I.~A. Stegun.
\newblock {\em Handbook of {M}athematical {F}unctions with {F}ormulas,
  {G}raphs, and {M}athematical {T}ables}, volume~55.
\newblock Courier Corporation, 1964.

\bibitem{antonietti2016reviewDG}
P.~F. Antonietti, A.~Cangiani, J.~Collis, Z.~Dong, E.~H. Georgoulis, S.~Giani,
  and P.~Houston.
\newblock Review of discontinuous {G}alerkin finite element methods for partial
  differential equations on complicated domains.
\newblock In {\em Building Bridges: Connections and Challenges in Modern
  Approaches to Numerical Partial Differential Equations}, pages 279--308.
  Springer, 2016.

\bibitem{VEM_fullync_biharmonic}
P.~F. Antonietti, G.~Manzini, and M.~Verani.
\newblock The fully nonconforming virtual element method for biharmonic
  problems.
\newblock {\em Math. Models Methods Appl. Sci.}, 28(02):387--407, 2018.

\bibitem{aronszajn1957unique}
N.~Aronszajn.
\newblock A unique continuation theorem for solutions of elliptic partial
  differential equations or inequalities of second order.
\newblock {\em J. Math. Pures Appl.}, 36(9):235--249, 1957.

\bibitem{nonconformingVEMbasic}
B.~Ayuso, K.~Lipnikov, and G.~Manzini.
\newblock The nonconforming virtual element method.
\newblock {\em ESAIM Math. Model. Numer. Anal.}, 50(3):879--904, 2016.

\bibitem{BabuskaMelenk_PUMintro}
I.~Babu{\v{s}}ka and J.~M. Melenk.
\newblock The partition of unity finite element method: basic theory and
  applications.
\newblock {\em Comput. Methods Appl. Mech. Engrg.}, 139(1-4):289--314, 1996.

\bibitem{VEMvolley}
L.~Beir{\~a}o~da Veiga, F.~Brezzi, A.~Cangiani, G.~Manzini, L.D. Marini, and
  A.~Russo.
\newblock Basic principles of virtual element methods.
\newblock {\em Math. Models Methods Appl. Sci.}, 23(01):199--214, 2013.

\bibitem{hitchhikersguideVEM}
L.~Beir{\~a}o~da Veiga, F.~Brezzi, L.D. Marini, and A.~Russo.
\newblock The hitchhiker's guide to the virtual element method.
\newblock {\em Math. Models Methods Appl. Sci.}, 24(8):1541--1573, 2014.

\bibitem{hpVEMbasic}
L.~Beir{\~a}o~da Veiga, A.~Chernov, L.~Mascotto, and A.~Russo.
\newblock Basic principles of $hp$ virtual elements on quasiuniform meshes.
\newblock {\em Math. Models Methods Appl. Sci.}, 26(8):1567--1598, 2016.

\bibitem{hpVEMcorner}
L.~Beir{\~a}o~da Veiga, A.~Chernov, L.~Mascotto, and A.~Russo.
\newblock Exponential convergence of the $hp$ virtual element method with
  corner singularity.
\newblock {\em Numer. Math.}, 138(3):581--613, 2018.

\bibitem{VEM3Dbasic}
L.~Beir{\~a}o~da Veiga, F.~Dassi, and A.~Russo.
\newblock High-order virtual element method on polyhedral meshes.
\newblock {\em Comput. Math. Appl.}, 74:1110--1122, 2017.

\bibitem{BLM_MFD}
L.~Beir{\~a}o~da Veiga, K.~Lipnikov, and G.~Manzini.
\newblock {\em The {M}imetic {F}inite {D}ifference {M}ethod for elliptic
  problems}, volume~11.
\newblock Springer, 2014.

\bibitem{CGM_nonconformingStokes}
A.~Cangiani, V.~Gyrya, and G.~Manzini.
\newblock The non-conforming virtual element method for the {S}tokes equations.
\newblock {\em SIAM J. Numer. Anal.}, 54(6):3411--3435, 2016.

\bibitem{cangianimanzinisutton_VEMconformingandnonconforming}
A.~Cangiani, G.~Manzini, and O.~J. Sutton.
\newblock Conforming and nonconforming virtual element methods for elliptic
  problems.
\newblock {\em IMA J. Numer. Anal.}, 37:1317--1354, 2016.

\bibitem{cao2018anisotropic}
S.~Cao and L.~Chen.
\newblock {Anisotropic error estimates of the linear nonconforming virtual
  element methods}.
\newblock \url{https://arxiv.org/abs/1806.09054}, 2018.

\bibitem{cessenatdespres_basic}
O.~Cessenat and B.~Despr\'{e}s.
\newblock Application of an ultra weak variational formulation of elliptic
  {PDE}s to the two-dimensional {H}elmholtz problem.
\newblock {\em SIAM J. Numer. Anal.}, 35(1):255--299, 1998.

\bibitem{conformingHarmonicVEM}
A.~Chernov and L.~Mascotto.
\newblock The harmonic virtual element method: stabilization and exponential
  convergence for the {L}aplace problem on polygonal domains, 2018.
\newblock doi: \url{https://doi.org/10.1093/imanum/dry038}.

\bibitem{cockburn_HDG}
B.~Cockburn, J.~Gopalakrishnan, and R.~Lazarov.
\newblock Unified hybridization of discontinuous {G}alerkin, mixed, and
  continuous {G}alerkin methods for second order elliptic problems.
\newblock {\em SIAM J. Numer. Anal.}, 47(2):1319--1365, 2009.

\bibitem{coltoninverse}
D.~Colton and R.~Kress.
\newblock {\em Inverse {A}coustic and {E}lectromagnetic {S}cattering {T}heory},
  volume~93.
\newblock Springer, Heidelberg, 2nd edition, 1998.

\bibitem{CrouzeixRaviart}
M.~Crouzeix and P.-A. Raviart.
\newblock Conforming and nonconforming finite element methods for solving the
  stationary {S}tokes equations.
\newblock {\em RAIRO Anal. Num\'er.}, 7(R3):33--75, 1973.

\bibitem{fetishVEM3D}
F.~Dassi and L.~Mascotto.
\newblock Exploring high-order three dimensional virtual elements: bases and
  stabilizations.
\newblock {\em Comput. Math. Appl.}, 75(9):3379--3401, 2018.

\bibitem{wavebasedmethod_overview}
E.~Deckers, O.~Atak, L.~Coox, R.~D'Amico, H.~Devriendt, S.~Jonckheere, K.~Koo,
  B.~Pluymers, D.~Vandepitte, and W.~Desmet.
\newblock The wave based method: {A}n overview of 15 years of research.
\newblock {\em Wave Motion}, 51(4):550--565, 2014.

\bibitem{dipietroErn_hho}
D.~A. Di~Pietro and A.~Ern.
\newblock Hybrid high-order methods for variable-diffusion problems on general
  meshes.
\newblock {\em C. R. Math. Acad. Sci. Paris}, 353(1):31--34, 2015.

\bibitem{farhat2001discontinuous}
C.~Farhat, I.~Harari, and L.~P. Franca.
\newblock The discontinuous enrichment method.
\newblock {\em Comput. Methods Appl. Mech. Engrg.}, 190(48):6455--6479, 2001.

\bibitem{gardini2018nonconforming}
F.~Gardini, G.~Manzini, and G.~Vacca.
\newblock The nonconforming virtual element method for eigenvalue problems.
\newblock \url{http://arxiv.org/abs/1802.02942}, 2018.

\bibitem{gittelson}
C.~J. Gittelson.
\newblock {Plane wave discontinuous Galerkin methods}.
\newblock Master's thesis, SAM-ETH Z\"{u}rich, 2008.

\bibitem{GHP_PWDGFEM_hversion}
C.~J. Gittelson, R.~Hiptmair, and I.~Perugia.
\newblock Plane wave discontinuous {G}alerkin methods: analysis of the
  $h$-version.
\newblock {\em ESAIM Math. Model. Numer. Anal.}, 43(2):297--331, 2009.

\bibitem{stability_helmholtz_graham_sauter}
I.G. Graham and S.A. Sauter.
\newblock {Stability and error analysis for the Helmholtz equation with
  variable coefficients}.
\newblock \url{https://arxiv.org/abs/1803.00966}, 2018.

\bibitem{hiptmair2011plane}
R.~Hiptmair, A.~Moiola, and I.~Perugia.
\newblock {Plane wave discontinuous Galerkin methods for the 2D Helmholtz
  equation: analysis of the p-version}.
\newblock {\em SIAM J. Numer. Anal.}, 49(1):264--284, 2011.

\bibitem{PWdG_hpversion}
R.~Hiptmair, A.~Moiola, and I.~Perugia.
\newblock Plane wave-discontinuous {G}alerkin methods: exponential convergence
  of the $hp$-version.
\newblock {\em Found. Comput. Math.}, 16(3):637--675, 2016.

\bibitem{PWDE_survey}
R.~Hiptmair, A.~Moiola, and I.~Perugia.
\newblock A survey of {T}refftz methods for the {H}elmholtz equation.
\newblock In {\em Building bridges: connections and challenges in modern
  approaches to numerical partial differential equations}, pages 237--279.
  Springer, 2016.

\bibitem{hmps_harmonicpolynomialsapproximationandTrefftzhpdgFEM}
R.~Hiptmair, A.~Moiola, I.~Perugia, and C.~Schwab.
\newblock Approximation by harmonic polynomials in star-shaped domains and
  exponential convergence of {T}refftz $hp$-d{G}{F}{E}{M}.
\newblock {\em ESAIM Math. Model. Numer. Anal.}, 48(3):727--752, 2014.

\bibitem{lipnikov2014mimetic}
K.~Lipnikov, G.~Manzini, and M.~Shashkov.
\newblock Mimetic finite difference method.
\newblock {\em J. Comput. Phys.}, 257:1163--1227, 2014.

\bibitem{nc_VEM_NavierStokes}
X.~Liu and Z.~Chen.
\newblock The nonconforming virtual element method for the {N}avier-{S}tokes
  equations.
\newblock {\em Adv. Comput. Math.}, 2018.
\newblock doi: \url{https://doi.org/10.1007/s10444-018-9602-z}.

\bibitem{fetishVEM}
L.~Mascotto.
\newblock Ill-conditioning in the virtual element method: Stabilizations and
  bases.
\newblock {\em Numer. Methods Partial Differential Equations},
  34(4):1258--1281, 2018.

\bibitem{ncHVEM}
L.~Mascotto, I.~Perugia, and A.~Pichler.
\newblock Non-conforming harmonic virtual element method: $h$- and
  $p$-versions.
\newblock \url{https://arxiv.org/abs/1801.00578}, 2018.

\bibitem{ncTVEM_theory}
L.~Mascotto, I.~Perugia, and A.~Pichler.
\newblock A nonconforming {T}refftz virtual element method for the {H}elmholtz
  problem.
\newblock \url{https://arxiv.org/abs/1805.05634}, 2018.

\bibitem{mclean2000strongly}
W.~C.~H. McLean.
\newblock {\em Strongly {E}lliptic {S}ystems and {B}oundary {I}ntegral
  {E}quations}.
\newblock Cambridge University Press, 2000.

\bibitem{monk1999least}
P.~Monk and D.-Q. Wang.
\newblock A least-squares method for the {H}elmholtz equation.
\newblock {\em Comput. Methods Appl. Mech. Engrg.}, 175(1-2):121--136, 1999.

\bibitem{Helmholtz-VEM}
I.~Perugia, P.~Pietra, and A.~Russo.
\newblock A plane wave virtual element method for the {H}elmholtz problem.
\newblock {\em ESAIM Math. Model. Numer. Anal.}, 50(3):783--808, 2016.

\bibitem{riou2008multiscale}
H.~Riou, P.~Ladeveze, and B.~Sourcis.
\newblock The multiscale {VTCR} approach applied to acoustics problems.
\newblock {\em J. Comput. Acoust.}, 16(04):487--505, 2008.

\bibitem{Weisser_basic}
S.~Rjasanow and S.~Wei{\ss}er.
\newblock Higher order {BEM}-based {FEM} on polygonal meshes.
\newblock {\em SIAM J. Numer. Anal.}, 50(5):2357--2378, 2012.

\bibitem{SchwabpandhpFEM}
C.~Schwab.
\newblock {\em $p$- and $hp$- {F}inite {E}lement {M}ethods: Theory and
  Applications in Solid and Fluid Mechanics}.
\newblock Clarendon Press Oxford, 1998.

\bibitem{paulinotestnumericipolygonalmeshes}
C.~Talischi, G.~H. Paulino, A.~Pereira, and I.~F.~M. Menezes.
\newblock Poly{M}esher: a general-purpose mesh generator for polygonal elements
  written in {M}atlab.
\newblock {\em Struct. Multidiscip. Optim.}, 45:309--328, 2012.

\bibitem{Triebel}
H.~Triebel.
\newblock {\em Interpolation theory, function spaces, differential operators}.
\newblock North-Holland, 1978.

\bibitem{zhao2016nonconforming}
J.~Zhao, S.~Chen, and B.~Zhang.
\newblock The nonconforming virtual element method for plate bending problems.
\newblock {\em Math. Models Methods Appl. Sci.}, 26(09):1671--1687, 2016.

\end{thebibliography}
}
\bibliographystyle{plain}

\end{document}